\documentclass[a4paper,reqno]{amsart}
\usepackage{amssymb}

\newtheorem{theorem}{Theorem}
\newtheorem{lemma}[theorem]{Lemma}
\newtheorem{corollary}[theorem]{Corollary}
\newtheorem{proposition}[theorem]{Proposition}

\numberwithin{equation}{section}
\numberwithin{theorem}{section}

\theoremstyle{definition}
\newtheorem{definition}[theorem]{Definition}
\newtheorem*{example*}{Example}
\newtheorem{example}[theorem]{Example}

\newtheorem{remark}[theorem]{Remark}
\newtheorem*{remark*}{Remark}

\newcommand{\bF}{{\mathbb F}}
\newcommand{\bZ}{{\mathbb Z}}

\newcommand{\bC}{{\mathbb C}}

\newcommand{\frg}{{\mathfrak g}}

\newcommand{\frh}{{\mathfrak h}}
\newcommand{\frc}{{\mathfrak c}}
\newcommand{\frf}{{\mathfrak f}}

\newcommand{\frs}{{\mathfrak s}}

\newcommand{\calB}{{\mathcal B}}
\newcommand{\calI}{{\mathcal I}}

\newcommand{\subo}{_{\bar 0}}

\providecommand{\espan}[1]{\text{span}\left\{ #1\right\}}
\providecommand{\alg}[1]{\text{alg}\left\langle #1\right\rangle}
\providecommand{\group}[1]{\text{group}\left\langle #1\right\rangle}

 \DeclareMathOperator{\tri}{\mathfrak{tri}}

 \DeclareMathOperator{\fro}{\mathfrak{o}}
 \DeclareMathOperator{\frsl}{{\mathfrak{sl}}}
 \DeclareMathOperator{\frsp}{{\mathfrak{sp}}}
 \DeclareMathOperator{\frso}{{\mathfrak{so}}}

 \DeclareMathOperator{\trace}{trace}
  
 \DeclareMathOperator{\ad}{ad}
  
  \DeclareMathOperator{\Tri}{Tri}
  \DeclareMathOperator{\Centr}{Centr}
  \DeclareMathOperator{\Norm}{Norm}
 \DeclareMathOperator{\der}{\mathfrak{der}}
 
 \DeclareMathOperator{\End}{End}
 
 \DeclareMathOperator{\Mat}{Mat}
 \DeclareMathOperator{\Aut}{Aut}
 \DeclareMathOperator{\Antiaut}{Antiaut}
 \DeclareMathOperator{\Cent}{Cent}
 \DeclareMathOperator{\Diag}{Diag}
 \DeclareMathOperator{\diag}{diag}
   \DeclareMathOperator{\Supp}{Supp}
 \DeclareMathOperator{\degree}{deg}

 \DeclareMathOperator{\ex}{ex}

\newenvironment{romanenumerate}
 {\begin{enumerate}
 \renewcommand{\theenumi}{\roman{enumi}}
 }{\end{enumerate}}

\begin{document}

\title{Fine gradings on simple classical Lie algebras}

\author[Alberto Elduque]{Alberto Elduque$^{\star}$}
 \thanks{$^{\star}$ Supported by the Spanish Ministerio de
 Educaci\'{o}n y Ciencia
 and FEDER (MTM 2007-67884-C04-02) and by the
Diputaci\'on General de Arag\'on (Grupo de Investigaci\'on de
\'Algebra)}
 \address{Departamento de Matem\'aticas e
 Instituto Universitario de Matem\'aticas y Aplicaciones,
 Universidad de Zaragoza, 50009 Zaragoza, Spain}
 \email{elduque@unizar.es}

%\dedicatory{Preliminary version}

\date{October 19, 2009}

\subjclass[2000]{Primary 17B70; Secondary 17B60}

\keywords{Grading, fine, simple, Lie algebra, graded division algebra}

\begin{abstract}
The fine abelian group gradings on the simple classical Lie algebras (including $D_4$) over algebraically closed fields of characteristic $0$ are determined up to equivalence. This is achieved by assigning certain invariant to such gradings that involve central graded division algebras and suitable sesquilinear forms on free modules over them.
\end{abstract}

\maketitle

\section{Introduction}\label{se:Introduction}

The purpose of this paper is the complete determination of the fine gradings on the simple classical Lie algebras over an arbitrary ground field $\bF$ of characteristic $0$, up to equivalence.

\smallskip

Fine gradings on the classical simple complex Lie algebras other than $\frso_8(\bC)$ (type $D_4$) were considered in \cite{HPP}. The arguments there are computational and use explicitly that the ground field is the field of complex numbers. The problem of whether two of the fine gradings obtained are equivalent is not completely settled, so the question of obtaining a complete and irredundant list of the fine gradings remained open.

On the other hand, the gradings on the simple Lie algebra $\frso_8(\bF)$, for an algebraically closed field of characteristic $0$, have been determined in \cite{DV} and \cite{DMVpr}, where the authors use the results in \cite{HPP}. However there are some missing fine gradings in this last reference. The complete list appears in our Theorem \ref{th:D4}.

A more conceptual approach to gradings on the simple classical Lie algebras have been considered in \cite{BSZ01}, \cite{BShZ}, \cite{BZ} or \cite{BG}, based on the study of the gradings on associative algebras. (Note that \cite[Theorem 2]{BShZ} contains an error which is corrected in \cite[theorem 3.6]{BG}.) Still this more conceptual approach does not give criteria to determine when two fine gradings are equivalent. Again type $D_4$ is excluded in these papers, and a minor point in \cite[Proposition 6.4]{BZ} is corrected in our Example \ref{ex:counterex_Bahturinetal}.

\smallskip

The assumption on the ground field to be algebraically closed and of characteristic $0$ will be kept all over the paper. Our approach will be closer to this more conceptual approach initiated by Bahturin and his collaborators.

\smallskip

The main idea here is to assign an invariant to each fine grading on a simple associative algebra. A key point in defining these invariants is that given a group grading $\Gamma:R=\oplus_{g\in G}R_g$ of the simple associative algebra $R$, there is a central graded division algebra $D=\oplus_{g\in G}D_g$ and a graded (necessarily free) right module $V=\oplus_{g\in G}V_g$ over $D$ such that $R$ is isomorphic, as a $G$-graded algebra, to the algebra of endomorphisms $\End_D(V)$, endowed with the natural grading over $G$ induced by the one in $V$. It is this point of view the one that indicates easily the right invariants to be used.

The central graded division algebra becomes the first component of the invariants attached to the grading $\Gamma$. Under further restrictions, the right module $V$ is endowed with a nondegenerate sesquilinear form. Diagonalizations of such forms will give  families of homogeneous elements which appear as the second components of the invariants. Equivalence classes of fine gradings will then be in bijection with suitable equivalence classes of pairs.

A quite self contained proof of the $D_4$ case will be given, following some of the ideas in \cite{DMVpr}, and using the previous results, in order to complete the list of nonequivalent fine gradings in this case too. The existence of outer automorphisms or order $3$ for the simple Lie algebra of type $D_4$ makes this case the most difficult and interesting one.

\medskip

But let us start with a review of the definitions.

Let $A$ be an algebra (not necessarily associative) over our ground field $\bF$, a \emph{grading} on $A$ is a decomposition
\begin{equation}\label{eq:Gamma_grading}
\Gamma: A=\oplus_{s\in S}A_s
\end{equation}
of $A$ into a direct sum of subspaces such that for any $s_1,s_2\in S$ there exists a $s_3\in S$ with $A_{s_1}A_{s_2}\subseteq A_{s_3}$. The grading $\Gamma$ is said to be a \emph{group grading} if there is a group $G$ containing $S$ such that $A_{s_1}A_{s_2}\subseteq A_{s_1s_2}$ (multiplication of indices in the group $G$) for any $s_1,s_2\in S$. Then we can write $\Gamma: A=\oplus_{g\in G}A_g$. The subset $\{g\in G: A_g\ne 0\}$ is called the \emph{support} of the grading and denoted by $\Supp\Gamma$ (or $\Supp A$ if the context is clear). If the group $G$ is abelian the grading is said to be an \emph{abelian group grading}. Semigroup gradings are defined in the same way. It was asserted in \cite{PZ} that any grading on a Lie algebra is a semigroup grading, but counterexamples were found in \cite{Eld06} and \cite{Eldpra}.

Two gradings $\Gamma: A=\oplus_{s\in S}A_s$ and $\Gamma':A'=\oplus_{s'\in S'}A'_{s'}$ are \emph{equivalent} if there is an isomorphism $\phi:A\rightarrow A'$ such that for any $s\in S$ there is a $s'\in S'$ with $\phi(A_s)=A'_{s'}$. The isomorphism $\phi$ will be called a \emph{graded isomorphism}.

Given two group gradings over the same group, $\Gamma: A=\oplus_{g\in G}A_s$ and $\Gamma':A'=\oplus_{g\in G}A'_g$, a $G$-homomorphism (or $G$-graded homomorphism) will indicate an algebra homomorphism $\phi:A\rightarrow A'$ such that $\phi(A_g)\subseteq A'_g$ for any $g\in G$. If the $G$-homomorphism $\phi$ is an algebra isomorphism, then it will be said to be a $G$-isomorphism (or $G$-graded isomorphism).

The \emph{type} of a grading $\Gamma$ is the sequence of numbers $(n_1,n_2,\ldots,n_r)$ where $n_i$ denotes the number of homogeneous spaces of dimension $i$, $i=1,\ldots,r$, $n_r\ne 0$. (Thus $\dim A=\sum_{i=1}^r in_i$.)

Let $\Gamma$ and $\Gamma'$ be two gradings on $A$ as above. The grading $\Gamma$ is said to be a \emph{refinement} of $\Gamma'$ (or $\Gamma'$ a \emph{coarsening} of $\Gamma$) if for any $s\in S$ there is an index $s'\in S'$ such that $A_s\subseteq A_{s'}$. In other words, any homogeneous space in $\Gamma'$ is a (direct) sum of some homogeneous spaces in $\Gamma$. A grading is called \emph{fine} if it admits no proper refinement.

We will restrict ourselves to abelian group gradings, and hence additive notation will be used. Actually, any group grading on a finite dimensional simple Lie algebra is an abelian group grading (see
\cite[Proposition 3]{Misha}, where a very complete survey of many results and references on gradings on Lie algebras can be found). Hence a fine grading will refer to an abelian group grading which admits no proper refinement in the class of abelian group gradings. Moreover, since all the algebras considered will be finite dimensional, the abelian grading groups will always be considered, without loss of generality, to be finitely generated.

Given a grading as in \eqref{eq:Gamma_grading}, one may consider the abelian group $G$ generated by $\{s\in S: A_s\ne 0\}$ subject only to the relations $s_1+s_2=s_3$ if $0\ne A_{s_1}A_{s_2}\subseteq A_{s_3}$. Then $A$ is graded over $G$ (or $G$-graded); that is,
\[
A=\oplus_{g\in G}A_g,
\]
where $A_g$ is the sum of the homogeneous components $A_s$ such that the class of $s$ in $G$ is $g$. Note that if $\Gamma$ is already an abelian group grading there is at most one such homogeneous component. This group $G$ has the following property: given any group grading $A=\oplus_{h\in H}A_h$ for an abelian group $H$ which is a coarsening of $\Gamma$, then there exists a unique homomorphism $f:G\rightarrow H$  such that $A_h=\oplus_{g\in f^{-1}(h)}A_g$ (see \cite{Misha}). The group $G$ is called the \emph{universal grading group of $\Gamma$}. The universal grading groups of two equivalent gradings are isomorphic.

Again, given a grading as in \eqref{eq:Gamma_grading}, the \emph{diagonal group of $\Gamma$}, denoted by $\Diag_\Gamma(A)$ is defined by:
\[
\Diag_\Gamma(A)=\{\Phi\in\Aut A: \Phi\vert_{A_s}\ \text{is a scalar multiple of the identity}\ \forall s\in S\}.
\]
If $\Gamma$ is a fine grading (in the class of abelian group gradings) then $\Diag_\Gamma(A)$ is a maximal abelian diagonalizable subgroup (or just \emph{MAD} for short) of the group $\Aut A$ of automorphisms of $A$ (see \cite{PZ}). Conversely, given a MAD $M$ in $\Aut A$, then the algebra $A$ decomposes as $A=\oplus_{\chi\in \hat M}A\chi$, where $\hat M$ is the set of characters of the group $M$ (that is,  continuous homomorphisms $\chi:M\rightarrow \bF^\times$ in the Zariski topology) and $A_\chi=\{x\in A: f(x)=\chi(f)x\ \forall f\in M\}$ for any $\chi\in\hat M$, and this is a fine grading.

\smallskip

A description of the fine gradings for the classical simple complex Lie algebras (except type $D_4$) has been obtained in \cite{HPP}. An explicit and irredundant description was later given for some classical simple Lie algebras of small rank (see the nice survey \cite{Svobodova} and references therein). Such descriptions are also known for the octonions \cite{Eld98}, the exceptional simple Jordan algebra (or Albert algebra) \cite{DMF4} and for the exceptional simple Lie algebras of type $G_2$ \cite{DMG2}, and independently \cite{BTG2}, and of type $F_4$ \cite{DMF4}.

\smallskip

Here such an explicit and irredundant description will be given for all the simple classical Lie algebras, including type $D_4$.

\medskip

The paper is structured as follows. Next section will be devoted to study the (abelian group) gradings on associative matrix algebras. Although the main result  (Theorem \ref{th:fine_gradings_R}) is already known (\cite[Theorem 6]{BSZ01}),  in our proof we will introduce some of the methods and notations that will be needed in the later sections. Then Section \ref{se:matrix_anti} will deal with simple associative graded algebras endowed with a graded antiautomorphism (that is, all the homogeneous spaces are invariant for the antiautomorphism). Some invariants will be attached to these gradings, which will allow us to distinguish non equivalent gradings. Section \ref{se:A} will then be devoted to the classification, up to equivalence, of the fine gradings on the simple Lie algebras of type $A$ (the special linear Lie algebras), and Section \ref{se:finesosp} to the corresponding classification for types $B$, $C$ and $D$, with the exception of $D_4$ which, as mentioned above, is exceptional due to the existence of outer order $3$ automorphisms. The final Section \ref{se:D4} will deal with this exceptional case.

\bigskip

\section{Gradings on matrix algebras}\label{se:matrix}

Recall that \emph{the ground field $\bF$ will always be assumed to be algebraically closed and of characteristic zero}.

\smallskip

Given a natural number $n$, consider the associative algebra
\begin{equation}\label{eq:divisionAn}
A_n=\alg{x,y : x^n=1=y^n,\ xy=\epsilon_n yx},
\end{equation}
where $\epsilon_n$ is a primitive $n$th root of unity. The algebra $A_n$ is isomorphic to the matrix algebra $\Mat_n(\bF)$, and it is naturally $\bZ_n\times\bZ_n$-graded, with $\degree(x)=(\bar 1,\bar 0)$ and $\degree(y)=(\bar 0,\bar 1)$. Moreover, $A_n$ is a graded division algebra (that is, any nonzero homogeneous element is invertible) and it is central.

The next result appears essentially in \cite[Theorem 5]{BSZ01}. We include it for completeness.

\begin{proposition}\label{pr:graded_division}
Let $\Gamma: D=\oplus_{g\in G}D_g$ be a central graded division algebra over the group $G$. Then $H=\Supp\Gamma$ is a subgroup of $G$ and there is a decomposition $H=H_1\times \cdots\times H_r$, where for each $i=1,\ldots,r$, $H_i$ is isomorphic to $\bZ_{n_i}\times\bZ_{n_i}$, $n_i$ a power of a prime, such that $D=D_1\cdot\cdots\cdot D_r\simeq D_1\otimes \cdots\otimes D_r$, where each $D_i=\oplus_{h\in H_i}D_h$ is graded isomorphic to the algebra $A_{n_i}$ above.
\end{proposition}

(Unadorned tensor products are always consider over the ground field: $\otimes =\otimes_{\bF}$.)

\begin{proof}
Since the zero homogeneous component $D_0$ is a finite dimensional division algebra over $\bF$ it follows that $D_0=\bF 1$ and for any $g\in G$ with $D_g\ne 0$, the dimension of $D_g$ is $1$. For any $g_1,g_2\in G$ with $D_{g_1}\ne 0\ne D_{g_2}$, the fact that $D$ is a graded division algebra forces $D_{g_1}D_{g_2}$ to be $\ne 0$ and hence equal to $D_{g_1+g_2}$. This shows that $H=\Supp\Gamma$ is a subgroup of $G$.

Assume that $H=H_1\times H_2$ with $\gcd\bigl(\vert H_1\vert,\vert H_2\vert\bigr)=1$. Then for any $g\in H_1$, $h\in H_2$ and $0\ne x\in D_{g}$, $0\ne y\in D_h$, there are natural numbers $n,m$ such that $0\ne x^n\in \bF 1$, $0\ne y^m\in\bF 1$, and $\gcd(n,m)=1$. Then $xyx^{-1}y^{-1}\in D_0$, so $xyx^{-1}y^{-1}=\alpha 1$ for some $0\ne \alpha\in \bF$. But since $xyx^{-1}\in D_h$, $xyx^{-1}=\mu y$ for some $0\ne \mu\in \bF$ such that $\mu^n=1$. It follows that $\alpha=\mu$ satisfies $\alpha^n=1$. In a similar vein, dealing with $yx^{-1}y^{-1}$ we get $\alpha^m=1$ too. Thus $\alpha=1$ and hence $D_1=\oplus_{g\in H_1}D_g$ commutes elementwise with $D_2=\oplus_{h\in H_2}D_{h}$, so that $D=D_1D_2$ which, by dimension count, is isomorphic to $D_1\otimes D_2$.

Therefore it is enough to assume the order of $H$ to be a power of a prime $p$. Then $H$ decomposes as a product of cyclic groups $H=H_1\times \cdots \times H_r$, with $H_i$ generated by an element $h_i$ of order $p^{m_i}$ and $m_1\geq m_2\geq \cdots\geq m_r$. If $m_1>m_2$ holds and $0\ne x\in D_{h_1}$, $0\ne y\in D_{h_i}$ ($i\geq 2$), as before we obtain $xyx^{-1}y^{-1}=\alpha 1$  with $\alpha^{p^{m_i}}=1$, so that $\alpha^{p^{m_2}}=1$ as $m_2\geq m_i$. It follows that $x^{p^{m_2}}$ is an homogeneous central element of $D$. But $D$ is central, so $x^{p^{m_2}}\in D_{p^{m_2}h_1}\cap\bF 1=0$, a contradiction. Hence $m_1=m_2$ and there exists an index $i\geq 2$ with $m_1=m_i$ such that $xzx^{-1}z^{-1}=\alpha 1$ with $z\in D_{h_i}$ and $\alpha$ a primitive $p^{m_i}$th root of unity (otherwise $x^{p^{m_1-1}}$ would be central). Then $\tilde H_1=\group{h_1,h_i}$ is isomorphic to $\bZ_{p^{m_1}}\times \bZ_{p^{m_1}}$, and $D_1=\oplus_{h\in \tilde H_1}D_h$ is a central graded division algebra, which is graded isomorphic to $A_{p^{m_1}}$. In particular, $D_1$ is a central simple subalgebra of $D$, so that by the double centralizer property $D=D_1\Cent_{D}(D_1)=D_1D_2\simeq D_1\otimes D_2$, where $D_2=\Cent_{D}(D_1)$ is the centralizer of $D_1$ in $D$. Thus the algebra $D_2$ is again a central graded division algebra and the result follows.
\end{proof}

\begin{remark}
The central graded division algebra $(D,\Gamma)$ in Proposition \ref{pr:graded_division} is completely determined (as a graded division algebra) by the group $H=\Supp\Gamma$, which in turn is the universal grading group. These gradings on $D\simeq\Mat_{\sqrt{\vert H\vert}}(\bF)$ will be referred to as \emph{division gradings}. (These are called \emph{Pauli gradings} in \cite{Svobodova}.) \qed
\end{remark}

\begin{definition}
Let $\Gamma:R=\oplus_{g\in G}R_g$ be a group grading of the matrix algebra $R=\Mat_n(\bF)$. A nonzero idempotent $e=e^2\in R_0$ is said to be a \emph{graded primitive idempotent} of $(R,\Gamma)$ if the graded subalgebra $eRe$ is a graded division algebra. \qed
\end{definition}

Note that given any nonzero idempotent $e$ of $R=\Mat_n(\bF)$, the element $e$ is conjugated to a diagonal matrix with only $0$'s and $1$'s on the diagonal, and hence the subalgebra $eRe$ is isomorphic to the matrix algebra $\Mat_r(\bF)$ with $r$ the number of $1$'s on the diagonal. In particular, $eRe$ is always a central algebra (with unity $e$).

\begin{proposition}\label{pr:graded_primitive_idempotent}
Let $\Gamma:R=\oplus_{g\in G}R_g$ be a group grading of the matrix algebra $R=\Mat_n(\bF)$ and let $e\in R_0$ be a nonzero idempotent. Then $e$ is a graded primitive idempotent if and only if $Re$ is a minimal graded left ideal of $R$.
\end{proposition}
\begin{proof}
First, if $I$ is a minimal graded left ideal of $R$, $I^2$ is not $0$ since $R$ is simple, so there is an homogeneous element $0\ne x\in I$ such that $Ix\ne 0$ and, by minimality of $I$, we have $I=Ix$. Hence there is an element $e\in I_0$ such that $x=ex$. Again by minimality, $I\cap\{r\in R: rx=0\}=0$ holds, so $e^2-e=0$ and $I=Re$.

By the natural graded version of Schur's Lemma, the minimality of $I$ forces the endomorphism ring $D=\End_R(I)$ to be a $G$-graded division algebra. (The action of the elements of $D$ on $I$ will be considered on the right.)

Note finally that the map
\[
\begin{split}
\End_R(I)&\longrightarrow eRe\\
\varphi&\mapsto e\varphi=e^2\varphi=e(e\varphi)\in eI=eRe,
\end{split}
\]
is a $G$-graded isomorphism.

\smallskip

Conversely, if $e$ is a graded primitive idempotent, and $J$ is a minimal graded left ideal contained in $Re$, by the above arguments $J=Rf$ for a graded primitive idempotent $f$. But $f\in Re$, so $fe=f$ and $efe=ef$ is a degree $0$ idempotent in the graded division algebra $eRe$. Hence either $ef=0$ or $ef=e$. In the first case we get $f=f^2=(fe)^2=fefe=0$, a contradiction, while in the second case we get $e=ef\in J$, so $J=Re$ is minimal.
\end{proof}

The left ideal $I=Re$ in the proposition above is an irreducible $G$-graded left module for $R$, meaning that $I$ is a left module for $R$ with $RI\ne 0$, it is $G$-graded with $R_{g_1}I_{g_2}\subseteq I_{g_1+g_2}$ for any $g_1$ and $g_2$ in $G$, and it contains no proper $G$-graded submodule.

Given two $G$-graded modules $V^1$ and $V^2$ for $R$, an \emph{homogeneous homomorphism} of degree $g\in G$ is an $R$-module homomorphism $\phi:V^1\rightarrow V^2$ such that $\phi(V_{g'}^1)\subseteq V_{g'+g}^2$ for any $g'\in G$.

\begin{proposition}\label{pr:uniqueirreducible} Let $\Gamma:R=\oplus_{g\in G}R_g$ be a group grading of the matrix algebra $R=\Mat_n(\bF)$, and let $V^1$ and $V^2$ be two irreducible  $G$-graded left modules for $R$. Then there is an homogeneous isomorphism $\phi:V^1\rightarrow V^2$. Moreover, if $D^i=\End_R(V^i)$ denotes the centralizer of the action of $R$ on $V^i$ for $i=1,2$, then there is a $G$-graded isomorphism $\varphi:D^1\rightarrow D^2$ such that $\phi(vd)=\phi(v)\varphi(d)$ for any $v\in V^1$ and $d\in D^1$. (Here $D^1$ and $D^2$ are graded with the natural gradings induced by the gradings on $V^1$ and $V^2$, so that $V_{g_1}^iD_{g_2}^i\subseteq V_{g_1+g_2}^i$ for any $i=1,2$ and $g_1,g_2\in G$.)
\end{proposition}
\begin{proof}
Let $e$ be a graded primitive idempotent of $R$, and let $V$ be an irreducible $G$-graded left module for $R$. Then since $R$ is simple, the action on $V$ is faithful, so with $I=Re$ we have $0\ne IV$, and hence there is an homogeneous element $v\in V$ such that $0\ne Iv$. But $V$ is irreducible, so we get $V=Iv$, and the map $\phi:I\rightarrow V$ such that $\phi(x)=xv$ for any $x\in I$ is an homogeneous isomorphism. This proves the first part of the proposition, while the second is clear.
\end{proof}

\begin{definition}
Given a central graded division algebra $D=\oplus_{h\in H}D_h$ with $\Supp D=H$, a group $G$ containing $H$ as a subgroup, and a $G$-graded right $D$-module $V$ (that is, $V_gD_h\subseteq V_{g+h}$ for any $g\in G$ and $h\in H$, and note that the division property of $D$ forces $V$ to be a free right $D$-module containing bases consisting of homogeneous elements), the $G$-grading induced on $R=\End_D(V)$, where $f\in R_g$ if $f(V_{g'})\subseteq V_{g+g'}$ for any $g'\in G$, is said to be an \emph{induced grading}. \qed
\end{definition}

\begin{remark} We may ``shift'' the grading in $V$ in the previous definition by fixing an element $h\in G$ and defining $\tilde V_{g}=V_{g+h}$ for any $g\in G$. Then $V=\oplus_{g\in G}\tilde V_g$ is another grading which gives the same induced grading on $R=\End_D(V)$.

Also note that in the situation of the previous definition, the algebra $R=\End_D(V)$ is isomorphic to $\Mat_m(D)$ ($m=\dim_DV$) which in turn, since $D$ is isomorphic to $\Mat_r(\bF)$ for some $r$ by Proposition \ref{pr:graded_division}, is isomorphic to $\Mat_{mr}(\bF)$. \qed
\end{remark}

\smallskip

The next result shows that any group grading on a finite dimensional simple algebra is isomorphic to an induced grading from a right module for a central graded division algebra.

\begin{theorem}\label{th:induced_grading} Let $\Gamma:R=\oplus_{g\in G}R_g$ be a group grading of the matrix algebra $R=\Mat_n(\bF)$. Then there is a unique irreducible $G$-graded left module $V$ for $R$, up to homogeneous isomorphisms. Moreover, the action of $R$ on $V$ gives a $G$-graded isomorphism $\Phi:R\rightarrow \End_D(V): r\mapsto \rho_r$, where $\rho_r(v)=rv$ for any $r\in R$ and $v\in V$, $D=\End_R(V)$ is the centralizer of the action of $R$, and $\End_D(V)$ is endowed with the induced grading.
\end{theorem}
\begin{proof}
The uniqueness is an immediate consequence of Proposition \ref{pr:uniqueirreducible}, and it is clear that $\Phi$ is a one-to-one $G$-homomorphism. To show that $\Phi$ is an isomorphism, we may assume that $V$ is a graded minimal left ideal $I=Re$ for a graded primitive idempotent $e$. But the map $R\rightarrow \End_{eRe}(I)$: $r\mapsto \varphi_r(:x\mapsto rx)$, is an isomorphism, as the image $\varphi_R$ equals $\varphi_{IR}=\varphi_I\varphi_R$, which is a left ideal of $\End_{eRe}(I)$ containing the identity element $1=\varphi_1$, and hence it is the whole $\End_{eRe}(I)$.
\end{proof}

\begin{corollary} Let $\Gamma_1:R=\oplus_{g_1\in G_1}R_{g_1}$ and $\Gamma_2:R=\oplus_{g_2\in G_2}R_{g_1}$ be two group gradings of the matrix algebra $R=\Mat_n(\bF)$, and let $\Phi:(R,\Gamma_1)\rightarrow (R,\Gamma_2)$ be a graded isomorphism. Let $V^i$ be an irreducible $G_i$-graded left module for $R$, $i=1,2$. Then there is an isomorphism $\phi:V^1\rightarrow V^2$ such that for any $g_1\in G_1$ there is an element $g_2\in G_2$ with $\phi(V_{g_1}^1)=V_{g_2}^2$ such that $\phi(rv)=\Phi(r)\phi(v)$ for any $r\in R$ and $v\in V^1$.

Moreover, if both $\Gamma_1$ and $\Gamma_2$ are gradings over the same group $G$ and $\Phi$ is a $G$-graded isomorphism, then $\phi$ can be taken to be an homogeneous isomorphism.
\end{corollary}
\begin{proof}
Let $e$ be a graded primitive idempotent for $(R,\Gamma_1)$. Then $\Phi(e)$ is a graded primitive idempotent for $(R,\Gamma_2)$ and by Proposition \ref{pr:uniqueirreducible} we may assume $V^1=Re$ and $V^2=R\Phi(e)$. Then the restriction of $\Phi$ to $Re$ gives the desired isomorphism.
\end{proof}

Let us consider a central graded division algebra $D=\oplus_{h\in H}D_h$ with $\Supp D=H$, a group $G$ containing $H$ as a subgroup, and a $G$-graded right $D$-module $V$ as above. Take a basis $\calB$ of the free $D$-module $V$ consisting of homogeneous elements: $\calB=\{v_1,\ldots,v_m\}$, and let $g_i=\degree(v_i)$ for any $i=1,\ldots,m$. If for some $i\ne j$ we have $0\ne g_i-g_j\in H$, then we may change $v_j$ by $v_jd_h$ for $h=g_i-g_j$ and $0\ne d\in D_h$, and hence assume that $g_i=g_j$. Therefore an homogeneous basis $\calB$ can always be taken so that $\{g_i-g_j: 1\leq i\leq j\leq m\}\cap H=\{0\}$. In this case, let $V^0$ be the $\bF$-linear space spanned by the elements of $\calB$: $V^0=\bF v_1\oplus\cdots\oplus\bF v_m$, so that $V$ is naturally isomorphic to $V^0\otimes D$, and $\End_D(V)\simeq \End_{\bF}(V^0)\otimes D$. The induced grading on $\End_D(V)$ corresponds to the grading on $\End_{\bF}(V^0)\otimes D$ obtained as the tensor product of the induced grading on $\End_{\bF}(V^0)$ and the division grading on $D$. The first one can be refined to a $\bZ^{m-1}$-grading, where if $\epsilon_i$ denotes the element in $\bZ^{m-1}$ with a $1$ in the $i$th position and $0$'s elsewhere, then $\degree(v_1)=0$ and $\degree(v_i)=\epsilon_{i-1}$ for any $i=2,\ldots,m$. Note that $\Supp \End_{\bF}(V^0)=\{g_i-g_j:1\leq i\leq j\leq m\}$ and the assignment $\epsilon_i\mapsto g_{i+1}-g_1$ gives an epimorphism $\bZ^{m-1}\rightarrow \group{\Supp \End_{\bF}(V^0)}$ from $\bZ^{m-1}$ onto the subgroup generated by the support. The finer grading thus obtained in $\End_{\bF}(V^0)$ is said to be a \emph{Cartan grading}. Therefore, our initial grading on $\End_D(V)$ is refined to a $\bZ^{m-1}\times H$-grading which is fine (and $\bZ^{m-1}\times H$ turns out to be the universal grading group).

We summarize our arguments in the following result:

\begin{theorem}\label{th:fine_gradings_R}
\textbf{(Fine gradings on $\Mat_n(\bF)$ \cite[Theorem 6]{BSZ01}) } \quad Any fine grading on $\Mat_n(\bF)$ is equivalent to a tensor product of a Cartan grading and a division grading. \qed
\end{theorem}

\smallskip

Given a fine grading $\Gamma: R=\oplus_{g\in G}R_g$ of the matrix algebra $R=\Mat_n(\bF)$ consider the element
\[
\calI(R,\Gamma)=[eRe],
\]
where $e$ is a graded primitive idempotent of $(R,\Gamma)$ and $[eRe]$ denotes the isomorphism class, as a graded algebra, of the central graded division algebra $eRe$. By Proposition \ref{pr:uniqueirreducible} this is well defined.

\begin{theorem}\label{th:fine_gradings_R_invariant}
Two fine gradings $\Gamma_1$ and $\Gamma_2$ of the matrix algebra $R=\Mat_n(\bF)$ are equivalent if and only if $\calI(R,\Gamma_1)=\calI(R,\Gamma_2)$.
\end{theorem}
\begin{proof}
Write $\Gamma_1:R=\oplus_{g\in G}R_g$ and $\Gamma_2:R=\oplus_{g'\in G'}R_{g'}$. If $\Phi:(R,\Gamma_1)\rightarrow (R,\Gamma_2)$ is an equivalence of gradings and $e$ is a graded primitive idempotent of $(R,\Gamma_1)$, then so is $\Phi(e)$ relative to $(R,\Gamma_2)$, and $\Phi$ restricts to a graded isomorphism from $eRe$ onto $\Phi(e)R\Phi(e)=\Phi(eRe)$. Thus $\calI(R,\Gamma_1)=\calI(R,\Gamma_2)$.

Conversely, assume $\calI(R,\Gamma_1)=\calI(R,\Gamma_2)$ holds. Both $\Gamma_1$ and $\Gamma_2$ are fine so, up to equivalence (Theorem \ref{th:fine_gradings_R}), $R\simeq \Mat_{m_i}(\bF)\otimes D_i$, $i=1,2$, with $\Gamma_i$ corresponding to a Cartan grading on $\Mat_{m_i}(\bF)$ and a division grading on $D_i$. But then $\calI(R,\Gamma_i)=[D_i]$, $i=1,2$, and it follows that $D_1$ and $D_2$ are graded isomorphic. Then $m_1=m_2$ and any graded isomorphism between $D_1$ and $D_2$ extends (using the identity map from $\Mat_{m_1}(\bF)=\Mat_{m_2}(\bF)$ on itself) to an equivalence between $\Gamma_1$ and $\Gamma_2$.
\end{proof}

\medskip

\section{Gradings on matrix algebras with an antiautomorphism}\label{se:matrix_anti}

\begin{lemma}\label{le:graded_division_anti}
Let $D=\oplus_{h\in H}D_h$ be a central graded division algebra with $H=\Supp D$, and let $\tau$ be a graded antiautomorphism (that is, $\tau$ is an antiautomorphism with $D_h^\tau=D_h$ for any $h\in H$). Then $H$ is a $2$-elementary group and $\tau$ is an involution ($\tau^2=id$).
\end{lemma}
\begin{proof}
For any $h\in H$, $D_h=\bF x_h$ and $x_h^\tau\in\bF x_h$, so $x_h^\tau=\alpha_{x_h}x_h$ for some scalar $\alpha_{x_h}$. If $H$ were not $2$-elementary, the proof of Proposition \ref{pr:graded_division} shows that there are homogeneous elements $x,y\in D$ such that $xyx^{-1}y^{-1}=\epsilon 1$, where $\epsilon$ is a primitive $n$th root of unity, $n>2$. Then $xy=\epsilon yx$, so $(xy)^\tau=\epsilon(yx)^\tau$, or $\alpha_x\alpha_y yx=\epsilon\alpha_x\alpha_y xy$. Thus, $yx=\epsilon xy=\epsilon^2yx$, and $\epsilon^2=1$, a contradiction. Therefore $H$ is $2$-elementary and for any homogeneous element $x\in D$, $x^2\in D_0=\bF 1$, so $x^2=(x^2)^\tau=(x^\tau)^2=\alpha_x^2 x^2$. Hence $x^\tau=\pm x$ and $\tau^2=id$.
\end{proof}

The quaternion algebra $Q=\Mat_2(\bF)$ is a central $\bZ_2\times\bZ_2$-graded division algebra with:
\begin{equation}\label{eq:quaternion_units}
Q_{(\bar 0,\bar 0)}=\bF 1,\quad
 Q_{(\bar 1,\bar 0)}=\bF q_1,\quad
 Q_{(\bar 0,\bar 1)}=\bF q_2,\quad
 Q_{(\bar 1,\bar 1)}=\bF q_3,
\end{equation}
where
\[
q_1=\begin{pmatrix} 1&0\\ 0&-1\end{pmatrix},\quad
q_2=\begin{pmatrix} 0&1\\ 1&0\end{pmatrix},\quad
q_3=\begin{pmatrix} 0&1\\ -1&0\end{pmatrix},
\]
so that $q_1^2=1=q_2^2$ and $q_1q_2=-q_2q_1=q_3$.

The involution $\tau_{o}$ given by
\begin{equation}\label{eq:tau_o}
q_1^{\tau_o}=q_1,\quad q_2^{\tau_o}=q_2,\quad q_3^{\tau_o}=-q_3,
\end{equation}
is the usual transpose involution, which is orthogonal; while the involution $\tau_s$ given by
\begin{equation}\label{eq:tau_s}
q_i^{\tau_s}=-q_i,\quad i=1,2,3,
\end{equation}
is the standard conjugation of the quaternion algebra $Q$, and this is a symplectic involution.

\begin{corollary}\label{co:graded_division_involution}
Let $D=\oplus_{h\in H}D_h$ be a central graded division algebra, $H=\Supp D$, endowed with a graded involution $\tau$. Then:
\begin{romanenumerate}
\item If $\tau$ is orthogonal, then $(D,\tau)$ is graded isomorphic to $(Q,\tau_o)^{\otimes n}$ for some $n\geq 0$. (This means that $D$ is graded isomorphic to the $\bZ_2^{2n}$-graded algebra $Q^{\otimes n}$ through an isomorphism which takes $\tau$ to $\tau_o^{\otimes n}$.)
\item If $\tau$ is symplectic, then $(D,\tau)$ is graded isomorphic to $(Q,\tau_s)\otimes (Q,\tau_o)^{\otimes n}$ for some $n\geq 0$.
\end{romanenumerate}
\end{corollary}
\begin{proof}
If $\dim D\leq 4$, then either $D=\bF$ or $D$ is the quaternion algebra and the result is clear. Otherwise ($\dim D>4$) take an homogeneous element $x\in D_h$, $h\ne 0$, with $x^\tau=x$ and $x^2=1$ (this is always possible). Since $D$ is central, there is another homogeneous element $y\in D_g$ with $y^2=1$ and $xyx^{-1}=-y$. Then if $y^\tau=y$, the assignment $x\mapsto q_1$, $y\mapsto q_2$ gives an isomorphism from $(\alg{x,y},\tau)$ onto $(Q,\tau_o)$ and, since $D=\alg{x,y}\Cent_D(\alg{x,y})\simeq \alg{x,y}\otimes \Cent_D(\alg{x,y})$ we can proceed by an inductive argument. In case $y^\tau=-y$, then with $x\mapsto q_1$ and $y\mapsto q_3$ the same argument above works.
\end{proof}

\begin{lemma}\label{le:graded_involutions_D}
If $D$ is a central graded division algebra and $\tau$ and $\tau'$ are two graded involutions, then there is a nonzero homogeneous element $d\in D$ such that $x^{\tau'}=dx^\tau d^{-1}$ for any $x\in D$.
\end{lemma}
\begin{proof}
The map $\psi=\tau\tau'$ is a degree $0$ automorphism of $D$. But as an ungraded algebra $D$ is a matrix algebra, so any automorphism is inner. Hence there is an invertible element $a\in D$ such that $\psi(x)=axa^{-1}$. Since the degree of $\psi$ is $0$ and $x^2\in \bF 1$ for any homogeneous $x$, it follows that $axa^{-1}=\pm x$ for any homogeneous $x$. If $a$ were not homogeneous, then we would have $a=a_1+\cdots +a_r$, with $0\ne a_i\in D_{h_i}$ $i=1,\ldots,r$, the $h_i$'s being different. But if $r>1$, the centralizer of $a_1$ is different from the centralizer of $a_2$ (as the order $2$ inner automorphism induced by $a_1$ is different from the one induced by $a_2$)  and if $x$ is an homogeneous element in $\Cent(a_1)\setminus\Cent(a_2)$, then $ax\ne \pm xa$, a contradiction. Hence the degree $0$ automorphisms of $D$ are the inner automorphisms obtained by conjugation by homogeneous elements.

In particular, there is a nonzero homogeneous element $d\in D$ such that $\tau'=\tau^d$, where $x^{\tau^d}=dx^\tau d^{-1}$ for any $x\in D$.
\end{proof}

\medskip

\begin{definition}\label{de:sesquilinearbalanced}
Let $G$ be an abelian group, let $D$ be a central $G$-graded division algebra over $\bF$ with support $H$ endowed with a graded involution $\tau$, let $V$ be a $G$-graded right module for $D$. A \emph{$\tau$-sesquilinear nondegenerate homogeneous form of degree $g\in G$} is a map
\[
B:V\times V\rightarrow D
\]
satisfying the following conditions:
\begin{romanenumerate}
\item $B$ is $\bF$-bilinear,
\item $B(V_{g_1},V_{g_2})\subseteq D_{g_1+g_2+g}$ for any $g_1,g_2\in G$,
\item $B(xd,y)=d^\tau B(x,y)$, $B(x,yd)=B(x,y)d$ for any $x,y\in V$ and $d\in D$,
\item $\{v\in V: B(v,V)=0\}=0=\{v\in V: B(V,v)=0\}$.
\end{romanenumerate}
Moreover, this $\tau$-sesquilinear form is said to be \emph{balanced} if for any homogeneous elements $v$ and $w$ in $V$, $B(v,w)$ equals $0$ if and only if so does $B(w,v)$. \qed
\end{definition}

\begin{theorem}\label{th:grading_varphi_B}
Let $G$ be an abelian group, let $D$ be a central $G$-graded division algebra over $\bF$ with support $H$, let $V$ be a $G$-graded right module for $D$, and let $\varphi$ be a $G$-antiautomorphism of the $G$-graded algebra $R=\End_\bF (V)$ with the induced grading. (Hence $\varphi$ is an antiautomorphism of $R$ such that $\varphi(R_g)=R_g$ for any $g\in G$.) Then the following conditions hold:
\begin{enumerate}
\item There exists a graded involution $\tau$ of $D$, an element $g\in G$, and a $\tau$-sesquilinear nondegenerate homogeneous form of degree $g$: $B:V\times V\rightarrow D$ such that $B(rv,w)=B(v,\varphi(r)w)$ for any $r\in R$, and $v,w\in V$.
\item If $\varphi^2$ acts as a scalar on each homogeneous component, then any $B$ as in the first item is balanced.
\item If $\tau'$ is another graded involution of $D$, and $B':V\times V\rightarrow D$ is another $\tau'$-sesquilinear nondegenerate homogeneous form of degree $g'\in G$ such that $B'(rv,w)=B'(v,\varphi(r)w)$ for any $r\in R$ and $v,w\in V$, then there exists an homogeneous element $d\in D_{g-g'}$ such that $\tau'=\tau^d$ (that is, $x^{\tau'}=dx^\tau d^{-1}$ for any $x\in D$) and $B'(v,w)=dB(v,w)$ for any $v,w\in V$.
\end{enumerate}
\end{theorem}
\begin{proof}
By Proposition \ref{pr:uniqueirreducible} we may take $V=I=Re$, for a graded primitive idempotent $e$ of $R$, and hence we take $D=eRe$. Then $0\ne \varphi(e)\in R_0$ is another graded primitive idempotent. Since $eR\varphi(e)\ne 0$, there is an homogeneous element $a\in R_g$ such that $ea\varphi(e)\ne 0$. Now $ea\varphi(e)Re\ne 0$, so there is an homogeneous element $b'$ with $ea\varphi(e)b'e\ne 0$, and since $eRe$ is a graded division algebra (with unity $e$) there is an homogeneous element $d\in eRe$ with $ea\varphi(e)b'ede=e$. Take $b=b'ed$, then $ea\varphi(e)be=e$ and $b\in R_{-g}$.
Besides,  $(\varphi(e)bea\varphi(e))^2=\varphi(e)b(ea\varphi(e)be)a\varphi(e)=\varphi(e)bea\varphi(e)$ is a degree $0$ nonzero idempotent of the graded division algebra $\varphi(e)R\varphi(e)$, so we have
\begin{equation}\label{eq:ab}
\varphi(e)=\varphi(e)bea\varphi(e).
\end{equation}
Consider the following $\bF$-bilinear map, which is homogeneous of degree $g$:
\[
\begin{split}
B:I\times I&\longrightarrow D,\\
 (x,y)&\mapsto B(x,y)=ea\varphi(x)y.
\end{split}
\]
Note that the subspace $K_l=\{ x\in I: B(x,I)=0\}$ is a graded left ideal of $R$ which does not contain $e$ (as $ea\varphi(e)\ne 0$, so $ea\varphi(e)Re\ne 0$). Since $I$ is minimal, $K_l=0$ follows. In the same vein, the subspace $K_r=\{x\in I: B(I,x)=0\}$ is trivial too. On the other hand, for any $x,y\in I$ and $d\in D$:
\[
\begin{split}
B(x,yd)&=B(x,y)d\\[3pt]
B(xd,y)&=ea\varphi(d)\varphi(x)y=ea\varphi(d)\varphi(e)\varphi(x)y\qquad\text{(as $x=xe$)}\\
 &=ea\varphi(d)\varphi(e)bea\varphi(e)\varphi(x)y\qquad\text{(because of \eqref{eq:ab})}\\
 &=ea\varphi(d)bea\varphi(x)y=ea\varphi(d)beB(x,y)\\
 &=d^\tau B(x,y),
\end{split}
\]
where
\begin{equation}\label{eq:dtau}
d^\tau=ea\varphi(d)be\in eRe=D.
\end{equation}
The linear map $\tau:D\rightarrow D$ thus defined is homogeneous of degree $0$ and for any $d_1,d_2\in D$:
\[
\begin{split}
(d_1d_2)^\tau&=ea\varphi(d_2)\varphi(d_1)be
    =ea\varphi(d_2)\varphi(e)\varphi(d_1)be\qquad\text{(as $d_1=d_1e$)}\\
 &=ea\varphi(d_2)\varphi(e)bea\varphi(e)\varphi(d_1)be\\
 &=d_2^\tau d_1^\tau.
\end{split}
\]
Hence, by Lemma \ref{le:graded_division_anti}, $H=\Supp D$ is a $2$-elementary group and $\tau$ is an involution.

We conclude that $B$ is a $\tau$-sesquilinear nondegenerate homogeneous form of degree $g$ on $I$. Besides, for any $r\in R$, and $x,y\in I$:
\[
B(rx,y)=ea\varphi(rx)y=ea\varphi(x)\varphi(r)y=B(x,\varphi(r)y),
\]
so that the antiautomorphism $\varphi$ is given by the adjoint map relative to $B$. This proves (1).

\smallskip

Now, in case the hypotheses of item (2) are fulfilled, for any homogeneous elements $x,y\in I=Re$, one has:
\[
\begin{split}
B(x,y)=0\quad&\Longleftrightarrow\quad ea\varphi(x)y=0\\
 &\Longleftrightarrow\quad ea\varphi(e)\varphi(x)y=0\qquad\text{(as $x=xe$)}\\
 &\Longleftrightarrow\quad \varphi(x)y=0\qquad\text{(as $\varphi(e)bea\varphi(e)=\varphi(e)$)}\\
 &\Longleftrightarrow\quad \varphi(y)\varphi^2(x)=0\\
 &\Longleftrightarrow\quad\varphi(y)x=0\qquad\text{(as $\varphi^2(x)=\alpha_g x$ if $x\in I_g$)}\\
 &\Longleftrightarrow\quad B(y,x)=0,
\end{split}
\]
and this proves (2).

\smallskip

Now assume that $B_i$ is a $\tau_i$-sesquilinear nondegenerate form of degree $g_i$ whose adjoint map is $\varphi$ for $i=1,2$. By Lemma \ref{le:graded_involutions_D} there is a nonzero homogeneous element $d\in D_h$ such that $x^{\tau_1}=dx^{\tau_2}d^{-1}$. Consider the map $\bar B_1:V\times V\rightarrow D$ given by $\bar B_1(v,w)=dB_1(v,w)$. It is trivial to check that this is a $\tau_2$-sesquilinear nondegenerate homogeneous form of degree $g+h$ whose adjoint map is $\varphi$. Take an homogeneous basis $\{v_1,\dotsc,v_m\}$ of $V$ as a (free) $D$-module, and take the coordinate matrices $\Delta_i=\Bigl(B_i(v_i,v_j)\Bigr)\in\Mat_m(D)$ of $B_i$, $i=1,2$. The coordinate matrix of $\bar B_1$ is then $\bar\Delta_1=d\Delta_1$. For elements $v=\sum_{i=1}^m v_ix_i$ and $w=\sum_{i=1}^m v_iy_i$ we have:
\[
B_1(v,w)=\bigl(x_1^{\tau_1},\dotsc,x_m^{\tau_1}\bigr)\Delta_1\begin{pmatrix} y_1\\ \vdots\\ y_m\end{pmatrix}
=(X^T)^{\tau_1} \Delta Y,
\]
where $X=\bigl(x_1,\dotsc,x_m\bigr)^T$, $Y=\bigl(y_1,\dotsc,y_m\bigr)^T$, $T$ denotes the transpose, and $\Delta^{\tau}$ denotes the matrix obtained by applying $\tau$ to all the entries in $\Delta$.
In a similar vein we get:
\[
\begin{split}
\bar B_1(v,w)&=(X^T)^{\tau_2}\bar\Delta_1 Y,\\
B_2(v,w)&=(X^T)^{\tau_2}\bar\Delta_2 Y.
\end{split}
\]
Denote by $M_r$ the coordinate matrix of the action of $r\in R=\End_D(V)$ on $V$. Then, since $\bar B_1(rv,w)=\bar B_1(v,\varphi(r)w)$ for any $r\in R$ and $v,w\in V$, we get:
\[
(X^T)^{\tau_2}(M_r^T)^{\tau_2}\bar \Delta_1 Y= (X^T)^{\tau_2}\bar \Delta_1 M_{\varphi(r)}Y,
\]
for any $X,Y$, and hence we obtain $M_{\varphi(r)}=\bar\Delta_1^{-1}\bigl(M_r)^{\tau_2}\bar\Delta_1$, and also
$M_{\varphi(r)}=\Delta_2^{-1}\bigl(M_r)^{\tau_2}\Delta_2$. We conclude that $\Delta_2\bar\Delta_1^{-1}$ is in the center of $R$, which is $\bF 1$. Thus  there is a scalar $0\ne \mu\in\bF$ such that $\Delta_2=\mu\bar\Delta_1=\mu d\Delta_1$. But we may change $d$ by $\mu d$. This proves (3).
\end{proof}

\begin{definition}\label{de:varphi_grading}
Let $\varphi$ be an antiautomorphism of the matrix algebra $R=\Mat_n(\bF)$. A grading $\Gamma: R=\oplus_{g\in G}R_g$ is said to be a \emph{$\varphi$-grading} in case the following conditions hold:
\[
\varphi(R_g)=R_g\ \forall g\in G,\quad \varphi^2\vert_{R_g}=\alpha_g id\ \text{for some $0\ne\alpha_g\in \bF$}\ \forall g\in G,\quad \alpha_0=1.
\]

If $\varphi_1$ and $\varphi_2$ are two antiautomorphisms of $R$ and $\Gamma_i=\oplus_{g_i\in G_i}R_{g_i}$ is a $\varphi_i$-grading for $i=1,2$, then the $\varphi_1$-grading $\Gamma_1$ is said to be equivalent to the $\varphi_2$-grading $\Gamma_2$ if there is a graded isomorphism $\Phi:(R,\Gamma_1)\rightarrow (R,\Gamma_2)$ such that $\Phi\varphi_1\Phi^{-1}=\varphi_2$.\quad\qed
\end{definition}

The relevance of this definition will become clear when we study in the next section the fine gradings of the simple Lie algebras of type $A$.

\smallskip

We may talk about refinements and coarsenings of $\varphi$-gradings, of fine $\varphi$-gradings, ...

\smallskip
Because of Theorem \ref{th:induced_grading}, the first part of the previous theorem can be recast as follows:

\begin{corollary}\label{co:grading_varphi_B}
Let $\varphi$ be an antiautomorphism of the matrix algebra $R=\Mat_n(\bF)$, and let $\Gamma: R=\oplus_{g\in G} R_g$ be a $\varphi$-grading. Then there is a $2$-elementary subgroup $H$ of $G$, a central graded division algebra $D=\oplus_{h\in H}D_h$ (with $\Supp D=H$) endowed with a graded involution $\tau$, and a $G$-graded right $D$-module $V$ endowed with a $\tau$-sesquilinear nondegenerate and balanced homogeneous form $B:V\times V\rightarrow D$, such that $(R,\Gamma)$ is $G$-isomorphic to $(\End_D(V),\Gamma_V)$, where $\Gamma_V$ denotes the induced grading, through an isomorphism $\Phi$ such that $\Phi\varphi=\varphi_B\Phi$, where $\varphi_B$ denotes the antiautomorphism of $\End_D(V)$ given by the adjoint map relative to $B$. \qed
\end{corollary}

\smallskip

Let $D=\oplus_{h\in H}D_h$ be a central graded division algebra, with $\Supp D=H$, endowed with a graded involution $\tau$, let $G$ be a group containing $H$ as a subgroup, and let $V$ be a $G$-graded right $D$-module $V$ endowed with a $\tau$-sesquilinear nondegenerate balanced homogeneous form $B$ of degree $g\in G$. Then for any graded $D$-submodule $W$ of $V$ it makes sense to define the orthogonal $D$-submodule
\[
W^\perp=\{ x\in V: B(w,x)=0\ \forall w\in W\}=\{x\in V: B(x,w)=0\ \forall w\in W\}.
\]
In particular, if $x$ is an homogeneous element of $V$ with $B(x,x)\ne 0$, it follows that $V$ decomposes as:
\[
V=xD\oplus(xD)^\perp.
\]
Also, for any homogeneous elements $x,y\in V$ with $B(x,x)=0=B(y,y)$ and $B(x,y)\ne 0$, so that $y\not\in xD$, the element $B(x,y)=d$ is an homogeneous element of $D$ and $B(x,yd^{-1})=1$. Then $B(yd^{-1},x)=\mu\in \bF^\times$ (as $D_0=\bF 1$). We may scale $x$ by $\frac{1}{\sqrt{\mu}}$ and change $y$ to $yd^{-1}$ to get $B(x,y)=\nu$ and $B(y,x)=\nu^{-1}$, for $\nu=\frac{1}{\sqrt{\mu}}\in \bF$. In this case
\[
V=(xD\oplus yD)\oplus(xD\oplus yD)^\perp.
\]
In this way, an homogeneous $D$-basis of $V$: $\calB=\{v_1,\ldots,v_p,v_{p+1},\ldots,v_{p+2s}\}$ can be obtained such that
\begin{equation}\label{eq:good_basis}
\left\{\begin{aligned}
 &B(v_i,v_i)=d_i\in D_{h_i},\ i=1,\ldots,p,\ \text{with $d_i^2=1$},\\
 &B(v_{p+2j-1},v_{p+2j})=\nu_j,\ B(v_{p+2j},v_{p+2j-1})=\nu_j^{-1},\ 0\ne \nu_j\in\bF,\ j=1,\ldots,s,\\
 &B(v_i,v_j)=0\ \text{otherwise.}
\end{aligned}\right.
\end{equation}
Such a basis will be called a \emph{good basis} of $V$ relative to $B$. Let $g_i=\degree(v_i)$ ($v_i\in V_{g_i}$) for any $i=1,\ldots,p+2s$. Then if the degree of $B$ is $g_B\in G$ we have (note that $h=-h$ for any $h\in H$):
\begin{equation}\label{eq:relations}
2g_1+h_1=\cdots=2g_p+h_p=g_{p+1}+g_{p+2}=\cdots=g_{p+2s-1}+g_{p+2s}=-g_B.
\end{equation}

The matrix algebra $R=\End_D(V)$ is endowed with the induced grading, and the good basis $\calB$ allows us to identify $R$ with $\Mat_{p+2s}(D)\simeq \Mat_{p+2s}(\bF)\otimes D$. Note that different bases provide different identifications and hence different embeddings of $D$ into $R$. The element $E_{ij}\otimes d$ ($1\leq i,j\leq p+2s$, $d\in D$) will denote the $D$-endomorphism of $V$ which takes $v_j$ to $v_id$ and $v_l$ to $0$ for any $l\ne j$.

Moreover, given two elements $x=\sum_{i=1}^{p+2s} v_ix_i$, and $y=\sum_{i=1}^{p+2s}v_iy_i$, ($x_i,y_i\in D$ for any $i$), then
\[
B(x,y)=\begin{pmatrix}
x_1^\tau&\cdots,x_{p+2s}^\tau
\end{pmatrix}
\Delta\begin{pmatrix} y_1\\ \vdots\\ y_{p+2s}\end{pmatrix}
=(X^T)^\tau \Delta Y,
\]
where $\Delta$ is the block diagonal matrix:
\begin{equation}\label{eq:GammaBlocks}
\Delta=\begin{pmatrix}
d_1&&&&&&& \\
&\ddots&&&&&&\\
&&d_p&&&&&\\
&&&0&\nu_1&&&\\
&&&\nu_1^{-1}&0&&&\\
&&&&&\ddots&&\\
&&&&&&0&\nu_s\\
&&&&&&\nu_s^{-1}&0
\end{pmatrix}.
\end{equation}
For simplicity we will write
\[
\Delta=\diag\left(d_1,\ldots,d_p,\bigl(\begin{smallmatrix} 0&\nu_1\\ \nu_1^{-1}&0\end{smallmatrix}\bigr),\ldots,\bigl(\begin{smallmatrix} 0&\nu_s\\ \nu_s^{-1}&0\end{smallmatrix}\bigr)\right).
\]
Then for any $A\in\End_D(V)\simeq\Mat_{p+2s}(D)$, the adjoint map relative to $B$ is given by:
\begin{equation}\label{eq:varphiBDelta}
\varphi_B(A)=\Delta^{-1}(A^T)^\tau\Delta,
\end{equation}
so that
\begin{equation}\label{eq:varphiB2}
\varphi_B^2(A)=\Gamma A\Gamma^{-1},
\end{equation}
where
\[
\Gamma=\Delta^{-1}(\Delta^T)^\tau
\]
is the diagonal matrix
\begin{equation}\label{eq:varphiB2bis}
\diag(
\epsilon_1,\ldots,\epsilon_{p+2s})
\end{equation}
with $\epsilon_{p+2j-1}=\nu_j^2$ and $\epsilon_{p+2j}=\nu_j^{-2}$ for $j=1,\ldots,s$, and $\epsilon_i\in\{1,-1\}$ for $i=1,\ldots,p$
($d_i^\tau\in\bF d_i$, so $d_i^\tau=\epsilon_id_i$ with $\epsilon_i=\pm 1$ for any $i$).

Note that the restriction of $\varphi_B^2$ to $E_{ij}\otimes D$ is $(\epsilon_i\epsilon_j^{-1}) id$ and that the subspace $E_{ij}\otimes D_h$ is contained in the homogeneous component $R_{g_i-g_j+h}$ for any $1\leq i,j\leq p+2s$ and $h\in H$.

This proves the first part of the next result, the second part being clear too.

\begin{proposition}\label{pr:involutions}
Let $D=\oplus_{h\in H}D_h$ be a central graded division algebra, with $\Supp D=H$, endowed with a graded involution $\tau$, let $G$ be a group containing $H$ as a subgroup, and let $V$ be a $G$-graded right $D$-module $V$ endowed with a $\tau$-sesquilinear nondegenerate balanced homogeneous form $B$ of degree $g_B\in G$. Take a good basis $\{v_1,\dotsc,v_{p+2s}\}$ of $V$ over $D$ where $\deg v_i=g_i$ for any $i$, so that the coordinate matrix of $B$ in this basis is
\[
\Delta=\diag\left(d_1,\ldots,d_p,\bigl(\begin{smallmatrix} 0&\nu_1\\ \nu_1^{-1}&0\end{smallmatrix}\bigr),\ldots,\bigl(\begin{smallmatrix} 0&\nu_s\\ \nu_s^{-1}&0\end{smallmatrix}\bigr)\right).
\]
Write $\Gamma=\Delta^{-1}(\Delta^T)^\tau=\diag(
\epsilon_1,\ldots,\epsilon_{p+2s})$ as in \eqref{eq:varphiB2bis}. Then the following conditions hold:
\begin{enumerate}
\renewcommand*{\labelenumii}{(\theenumi.\theenumii)}
\item The induced grading $\Gamma_V$ is a $\varphi_B$-grading if and only if $\epsilon_i\epsilon_j^{-1}=\epsilon_{i'}\epsilon_{j'}^{-1}$ for any $i,j,i',j'$ such that $(g_i-g_j)-(g_{i'}-g_{j'})\in H$.
\item $\varphi_B$ is an involution ($\varphi_B^2=id$) if and only if $\epsilon_i=\epsilon_j=\pm 1$ for any $i\ne j$, and this happens if and only if either:
\begin{enumerate}
\item $d_i^\tau=d_i$ for any $i=1,\ldots,p$, and $\nu_j\in \{1,-1\}$ for any $j=1,\ldots,s$. Note that in this case, by multiplying by $-1$ the vector $v_{p+2j}$ if necessary, we may assume that $\nu_j=1$ for any $j=1,\ldots,s$, so that all the blocks $\left(\begin{smallmatrix} 0&\nu_j\\\nu_j^{-1}&0\end{smallmatrix}\right)$ may be assumed to be
    $\left(\begin{smallmatrix} 0&1\\ 1&0\end{smallmatrix}\right)$.
\item Or $d_i^\tau=-d_i$ for any $i=1,\ldots,p$, and $\nu_j$ is a primitive fourth root of unity for any $j=1,\ldots,s$. In this case, by multiplying $v_{p+2j}$ by $\nu_j^{-1}$ we may change all the blocks  $\left(\begin{smallmatrix} 0&\nu_j\\\nu_j^{-1}&0\end{smallmatrix}\right)$ in \eqref{eq:GammaBlocks} to
    $\left(\begin{smallmatrix} 0&1\\ -1&0\end{smallmatrix}\right)$.
\end{enumerate}
    In case \emph{(2.a)} the form $B$ is hermitian (that is, $B(y,x)=B(x,y)^\tau$ for any $x,y\in V$) and $\varphi_B$ is an orthogonal (respectively symplectic) involution if and only if so is $\tau$, while in case \emph{(2.b)} $B$ is skew-hermitian ($B(y,x)=-B(x,y)^\tau$) and $\varphi_B$ is an orthogonal (respectively symplectic) involution if and only if $\tau$ is symplectic (respectively orthogonal). \qed
\end{enumerate}
\end{proposition}

\smallskip

In general, assuming that the induced grading $\Gamma_V$ is a $\varphi_B$-grading, it fails to be a fine $\varphi_B$-grading, but it can be refined as follows.

Consider the group $\tilde G$ generated by elements $\tilde g_1,\ldots,\tilde g_{p+2s}$ and $H$ subject only to the relations (see \eqref{eq:relations})
\begin{equation}\label{eq:relationstilde}
2\tilde g_1+h_1=\cdots=2\tilde g_p+h_p=\tilde g_{p+1}+\tilde g_{p+2}=\cdots=\tilde g_{p+2s-1}+\tilde g_{p+2s}=0.
\end{equation}
Then $V$ is a $\tilde G$-graded right $D$-module by means of the assignment
\[
\degree (v_i)=\tilde g_i\quad\forall i=1,\ldots,p+2s.
\]
Denote by $\tilde \Gamma_V$ the induced grading on $\End_D(V)$. The $\tau$-sesquilinear nondegenerate balanced  form $B$ becomes homogeneous of degree $0$ relative to the $\tilde G$-grading.

\begin{remark}\label{re:GbarGtilde}
If $\bar G$ is any abelian group with $2\bar G=G$ (such group always exist since $G$ is finitely generated), then there is a well defined homomorphism
\[
\begin{split}
\chi:\tilde G&\longrightarrow \bar G\\
\tilde g_i&\mapsto g_i+\frac{1}{2}g_B,\quad\text{for any $i=1,\ldots,p+2s$}\\
h&\mapsto h\qquad\text{for any $h\in H$},
\end{split}
\]
where $\frac{1}{2}g_B$ denotes an element of $\bar G$ such that $2\bigl(\frac{1}{2}g_B\bigr)=g_B$ (the degree of $B$).

Moreover, we have $\Supp(\Gamma_V)=\{g_i-g_j+h: 1\leq i,j\leq p+2s,\ h\in H\}$, while $\Supp(\tilde\Gamma_V)=\{\tilde g_i-\tilde g_j+h: 1\leq i,j\leq p+2s,\ h\in H\}$, and $\chi\bigl(\Supp(\tilde\Gamma_V)\bigr)=\Supp(\Gamma_V)(\subseteq G)$.

Also, the group $\tilde G$ is a Cartesian product $\tilde H\times \bZ^s$, where $\tilde H$ is the group generated by $\tilde g_1,\ldots,\tilde g_p$ and $H$, and $\bZ^s$ is (isomorphic to) the subgroup (freely) generated by $\{\tilde g_{p+2j-1}: j=1,\ldots,s\}$. Because of \eqref{eq:relationstilde}, $\tilde H$ is a product of copies of $\bZ_4$ and of $\bZ_2$. The quotient $\tilde H/H$ is isomorphic to $\bZ_2^p$.

Finally, for $p\geq 1$ the subgroup generated by $\Supp(\tilde\Gamma_V)$ is a Cartesian product $\hat H\times \bZ^s$, where  $\hat H$ is the subgroup generated by $\tilde g_1-\tilde g_2,\ldots,\tilde g_1-\tilde g_p$ and $H$, and $\bZ^s$ is (isomorphic to) the subgroup (freely) generated by the elements $\tilde g_{p+2j-1}-\tilde g_1$, $j=1,\ldots, s$. And for $p=0$ the subgroup generated by $\Supp(\tilde\Gamma_V)$ is the Cartesian product $H\times \bZ^s$, where $\bZ^s$ is the subgroup generated by the elements $\tilde g_{p+2j-1}$, $j=1\ldots,s$.

We may ``shift'' the grading on $V$ without modifying the grading on $\End_D(V)$ by defining $\degree(v_i)=\tilde g_i-\tilde g_1$ for any $i$ (that is, by imposing the degree of $v_1$ to be $0$). \qed
\end{remark}

\begin{remark}\label{re:homogeneous_subspaces}
Note too that for any $h\in H$, the homogeneous component $R_h$ in the $\tilde\Gamma_V$ grading is $R_h=\sum_{i=1}^{p+2s}E_{ii}\otimes D_h$, whose dimension is $p+2s$. On the other hand, we have the one-dimensional homogeneous components
\[
R_{2\tilde g_{2p+j-1}+h}=E_{p+2j-1,p+2j}\otimes D_h,\qquad
R_{-2\tilde g_{2p+j-1}+h}=E_{p+2j,p+2j-1}\otimes D_h,
\]
and the dimension of all the other homogeneous components is exactly $2$. \qed
\end{remark}

\begin{proposition}\label{pr:GammatildeV}
Under the conditions of Remark \ref{re:GbarGtilde} the following properties hold:
\begin{enumerate}
\item $\tilde\Gamma_V$ is a refinement of $\Gamma_V$.
\item $\tilde\Gamma_V$ is a $\varphi_B$-grading.
\item $\tilde\Gamma_V$ is a fine $\varphi_B$-grading unless $p=2$, $s=0$ and $h_1=h_2$. In this case $\tilde\Gamma_V$ can be refined to a $\varphi_B$-grading which makes $\End_D(V)$ a graded division algebra.
\end{enumerate}
\end{proposition}
\begin{proof}
Part (1) is clear as it is part (2) since the conditions in item (1) of Proposition \ref{pr:involutions} are satisfied because $(\tilde g_i-\tilde g_j)-(\tilde g_{i'}-\tilde g_{j'})$ is in  $H$ only if either $i=i'$, $j=j'$ or $1\leq i,j\leq p$ and $i=j'$, $j=i'$. In the latter case, since $\epsilon_i,\epsilon_j\in\{1,-1\}$, $\epsilon_j=\epsilon_j^{-1}$ and the condition $\epsilon_i\epsilon_j^{-1}=\epsilon_{i'}\epsilon_{j'}^{-1}$ is satisfied.

\smallskip

The proof of (3) is quite technical, although the main idea used is simple. Given any homogeneous component $R_g$ of our $\varphi_B$-grading, then there is a scalar $\alpha_g\ne 0$ such that $\varphi_B^2(x)=\alpha_g x$ for any $x\in R_g$ (Definition \ref{de:varphi_grading}). Then for any $x\in R_g$, $x\pm\sqrt{\alpha_g}^{-1}\varphi_B(x)$ is an eigenvector for $\varphi_B$. If the dimension of $R_g$ is $2$, then these eigenvectors must be homogeneous in any refinement (as $\varphi_B$-gradings) of our grading.

\smallskip

Consider first the case $p=2$, $s=0$ and $h_1=h_2$, then we may take $d_1=d_2=d$ with $d^2=1$ and hence for any homogeneous $x\in D$, $\varphi_B(E_{ij}\otimes x)=E_{ji}\otimes dx^\tau d$ for $1\leq i,j\leq 2$, and $\varphi_B^2=id$, so the homogeneous spaces of $\tilde\Gamma_V$ are exactly the following spaces:
\[
R_h=(E_{11}\otimes D_h)\oplus(E_{22}\otimes D_h),\quad
R_{\tilde g_1-\tilde g_2+h}=(E_{12}\otimes D_h)\oplus (E_{21}\otimes D_h),
\]
for $h\in H$. This grading can be refined to the $\varphi_B$-grading over the group $\bZ_2^2\times H$ with (note that $\varphi_B^2=id$ here):
\[
\left\{%
\begin{aligned}
R_{(\bar 0,\bar 0)+h}&=(E_{11}+E_{22})\otimes D_h,\\
R_{(\bar 1,\bar 1)+h}&=(E_{11}-E_{22})\otimes D_h,\\
R_{(\bar 1,\bar 0)+h}&=(E_{12}+E_{21})\otimes D_h,\\
R_{(\bar 0,\bar 1)+h}&=(E_{12}-E_{21})\otimes D_h.
\end{aligned}
\right.
\]
This gives a proper refinement of $\varphi_B$-gradings which makes $R=\End_D(V)$ a graded division algebra.

Now, in general, consider the element $\delta_j=\tilde g_{p+2j-1}\in \tilde G$, whose order is infinite, and assume first that $p=0$. Then with $R=\End_D(V)$ the homogeneous spaces $R_{2\delta_j+h}=E_{2j-1,2j}\otimes D_h$ and $R_{-2\delta_j+h}=E_{2j,2j-1}\otimes D_h$ have dimension $1$ for any $h\in H$ (see Remark \ref{re:homogeneous_subspaces}). Then $E_{2j-1,2j-1}\otimes 1=(E_{2j-1,2j}\otimes 1)(E_{2j,2j-1}\otimes 1)$ and $E_{2j,2j}\otimes 1$ are homogeneous idempotents in any refinement of $\tilde\Gamma_V$, and hence belong to the zeroth component in any refinement. For $i\ne j$, the dimension of $R_{\delta_i+\delta_j+h}$ is $2$ for any $h\in H$, as $R_{\delta_i+\delta_j+h}=(E_{2i-1,2j}\otimes D_h)\oplus (E_{2j-1,2i}\otimes D_h)$, and $\varphi_B^2$ acts on these spaces as $(\nu_i\nu_j)^2$ times the identity map since
\[
\varphi_B(E_{2i-1,2j}\otimes x_h)=E_{2j-1,2i}\otimes \nu_i\nu_jx_h^\tau
\]
for $x_h\in D_h$. Hence the elements $E_{2i-1,2j}\otimes x_h\pm E_{2j-1,2i}\otimes x_h^\tau$ (which are eigenvectors for $\varphi_B$ in $R_{\delta_i+\delta_j+h}$) are homogeneous in any refinement (of $\varphi_B$-gradings!); but then the elements:
\[
\begin{split}
&(E_{2i-1,2i-1}\otimes 1)(E_{2i-1,2j}\otimes x_h+E_{2j-1,2i}\otimes x_h^\tau)=E_{2i-1,2j}\otimes x_h,\\
&(E_{2i-1,2j}\otimes x_h+E_{2j-1,2i}\otimes x_h^\tau)(E_{2i,2i}\otimes 1)=E_{2j-1,2i}\otimes x_h^\tau=\pm E_{2j-1,2i}\otimes x_h
\end{split}
\]
are homogeneous of the same degree in any refinement. Therefore the homogeneous space $R_{\delta_i+\delta_j+h}$ does not split when considering a refinement. The same argument applies to the other homogeneous spaces, and this shows that $\tilde\Gamma_V$ is a fine $\varphi_B$-grading.

Now if $p>0$ the same arguments above show that $E_{p+j,p+j}\otimes 1$ is a degree $0$ idempotent in any refinement, and that the homogeneous subspaces $R_{\pm\delta_i\pm\delta_j+h}$ do not split. If $p=1$, as $1=\sum_{i=1}^{p+2s}E_{ii}\otimes 1$ is always homogeneous of degree $0$, it follows that $E_{11}\otimes 1$ is homogeneous of degree $0$ too in any refinement, and the same sort of arguments work. If $p\geq 3$ note that for $x\in D_h$, we have
\[
\varphi_B(E_{12}\otimes x)=E_{21}\otimes d_2^{-1}x^\tau d_1,\quad
\varphi_B(E_{21}\otimes x)=E_{12}\otimes d_1^{-1}x^\tau d_2,
\]
and since $R_{\tilde g_1-\tilde g_2+h}=(E_{12}\otimes D_h)\oplus(E_{21}\otimes D_{h+h_1+h_2})$, we have that the restriction of $\varphi_B^2$ to $R_{\tilde g_1+\tilde g_2+h}$ is either the identity or minus the identity: $\varphi_B^2\vert_{R_{\tilde g_1+\tilde g_2+h}}=\epsilon id$, $\epsilon=\pm 1$, depending on $d_i^\tau$ being equal to $d_i$ or to $-d_i$, for $i=1,2$ (and hence $\epsilon$ is independent of $h$). In particular, the eigenvectors (for  $\varphi_B$) $E_{12}\otimes 1+\sqrt{\epsilon}E_{21}\otimes d_1^{-1}d_2$ and $E_{21}\otimes 1\pm\sqrt{\epsilon}E_{12}\otimes d_1^{-1}d_2$ are homogeneous in any refinement and so is their product $(E_{11}\pm E_{22})\otimes 1$. In the same vein $(E_{11}-E_{33})\otimes 1$ is homogeneous, and the product $E_{11}\otimes 1$ is homogeneous too in any refinement. Now similar arguments as above work.

If $p=2$ and $s>0$ one gets that $(E_{11}\pm E_{22})\otimes 1$ is homogeneous in any refinement, but also $(E_{14}+E_{31})\otimes 1$ is homogeneous. Hence, multiplying by the degree $0$ idempotents $E_{44}\otimes 1$ and $E_{33}\otimes 1$ one gets that $E_{14}\otimes 1$ is homogeneous in any refinement. Then  so are $E_{41}\otimes 1$ and the product $E_{11}\otimes 1=(E_{14}\otimes 1)(E_{41}\otimes 1)$, and the result follows in this case too.

We are left with the case $p=2$, $s=0$ and $h_1\ne h_2$. As above for any $h\in H$ we have
$
R_h=(E_{11}\otimes D_h)\oplus(E_{22}\otimes D_h)$ and
$R_{\tilde g_1-\tilde g_2+h}=(E_{12}\otimes D_h)\oplus(E_{21}\otimes D_{h_1+h_2+h})$, and the restriction of $\varphi_B^2$ to $R_{\tilde g_1+\tilde g_2+h}$ is $\epsilon id$, for $\epsilon=\pm 1$ independent of $h$. Also, for any homogeneous $x\in D$, $\varphi_B(E_{21}\otimes x)=E_{12}\otimes d_1^{-1}x^\tau d_2=\pm E_{12}\otimes d_1^{-1}d_2x$ (since $xd_2x^{-1}=\pm d_2$). Thus, the eigenvectors for $\varphi_B$ given by $(E_{21}\otimes 1\pm \sqrt{\epsilon} E_{12}\otimes d_1^{-1}d_2)(1\otimes x)$, for homogeneous $x$, remain homogeneous in any refinement (as a $\varphi_B$-grading) and so are the elements
\[
(E_{21}\otimes 1\pm \sqrt{\epsilon}E_{12}\otimes d_1^{-1}d_2)(E_{21}\otimes 1\pm \sqrt{\epsilon}E_{12}\otimes d_1^{-1}d_2)(1\otimes x).
\]
We conclude that the elements of the form $(E_{11}\pm E_{22})\otimes x$, for homogeneous $x\in D$, are homogeneous in any refinement. One may take now an homogeneous element $z\in D$ with $zd_1=d_1z$ but $zd_2=-d_2z$ (since the centralizers $\Cent(d_i)$ of $d_1$ and $d_2$ are different as $\Cent(d_1)=\Cent(\bF 1+\bF d_1)$ and $\Cent(\Cent(\bF 1+\bF d_1))=\bF 1+\bF d_1$ by the double centralizer property). Then the elements
\[
\begin{split}
((E_{11}-E_{22})\otimes z)(E_{21}\otimes 1+\sqrt{\epsilon}E_{12}\otimes d_1^{-1}d_2)&=-E_{21}\otimes z+\sqrt{\epsilon}E_{12}\otimes zd_1^{-1}d_2,\\[3pt]
(E_{21}\otimes 1+\sqrt{\epsilon}E_{12}\otimes d_1^{-1}d_2)((E_{11}-E_{22})\otimes z)
 &=E_{21}\otimes z-\sqrt{\epsilon}E_{12}\otimes d_1^{-1}d_2z\\
 &=E_{21}\otimes z+\sqrt{\epsilon}E_{12}\otimes zd_1^{-1}d_2,
\end{split}
\]
are homogeneous of the same degree in any refinement. Hence $E_{21}\otimes z$ and $E_{12}\otimes zd_1^{-1}d_2$ are homogeneous of the same degree and hence $((E_{11}+E_{22})\otimes z)(E_{21}\otimes z)=\pm E_{21}\otimes 1$ is homogeneous in any refinement. In the same vein we get that $E_{12}\otimes 1$ is homogeneous and so are the idempotent elements $E_{11}\otimes 1=(E_{12}\otimes 1)(E_{21}\otimes 1)$ and $E_{22}\otimes 1=(E_{21}\otimes 1)(E_{12}\otimes 1)$. Being idempotent elements, their degree is $0$ in any refinement and now some previous arguments finish the proof.
\end{proof}

\begin{remark}\label{re:fineGamma}
Therefore, given an antiautomorphism $\varphi$ and a fine $\varphi$-grading $\Gamma$ of a matrix algebra $R$, Theorem \ref{th:grading_varphi_B} shows that $\Gamma$ is equivalent to a fine $\varphi_B$-grading $\Gamma_V$ on $\End_D(V)$ for some graded division algebra $D$ endowed with a graded involution $\tau$, right graded $D$-module $V$, and balanced $\tau$-sesquilinear nondegenerate homogeneous form $B:V\times V\rightarrow D$, where the fine $\varphi_B$-grading $\Gamma_V$ is the grading induced by the grading on $V$. Moreover, since this $\varphi_B$-grading is fine, by Proposition \ref{pr:GammatildeV} it is equivalent to the grading $\tilde\Gamma_V$ in Remark \ref{re:GbarGtilde} and, in particular, the case $p=2$, $s=0$ and $h_1=h_2$ is not possible. \qed
\end{remark}

\begin{remark}\label{re:goodbasis}
Consider the grading $\tilde\Gamma_V$ on $\End_D(V)$ above. Then the nonzero homogeneous elements of $V$ are
\[
V^{hom}=\cup_{i=1}^{p+2s} v_iD^{hom},
\]
where $D^{hom}=\cup_{h\in H}(D_h\setminus\{0\})$ denotes the set of nonzero homogeneous elements in the central graded division algebra $D$.
Therefore the good basis is completely determined by the $\tilde\Gamma_V$-grading up to a reordering and scaling by homogeneous elements of $D$.

Moreover, for any $i=1,\ldots,p+2s$ and $x_h\in D_h$,
\[
B(v_ix_h,v_ix_h)=x_h^\tau B(v_i,v_i)x_h\in\bF B(v_i,v_i),
\]
since $x_h^\tau=\pm x_h$ and $x_hD_{h'}x_h=x_h^{-1}D_{h'}x_h=D_{h'}$ for any $h'\in H$. Hence the tuple $(d_1,\ldots,d_p)$ is completely determined up to reordering and scaling by nonzero elements in $\bF$. \qed
\end{remark}

\medskip

Consider now the class of pairs $\bigl(D,(d_1,\ldots,d_p)\bigr)$ consisting of a central graded division algebras with support a $2$-elementary group and a $p$-tuple ($p\geq 0$) of nonzero homogeneous elements in $D$.

\begin{definition}\label{de:equivalenceDx1xp}
Two such pairs $\bigl(D,(d_1,\ldots,d_p)\bigr)$ and $\bigl(D',(d_1',\ldots,d_q')\bigr)$ are said to be equivalent if $p=q$, and there is a graded isomorphism $\Phi:D\rightarrow D'$, a permutation $\pi\in S_p$ and a nonzero homogeneous element $z\in D$ such that
\[
\Phi(zd_i)\in \bF d_{\pi(i)}'\quad \forall i=1,\ldots,p.
\]

This is an equivalence relation. The equivalence class of the pair $\bigl(D,(d_1,\ldots,d_p)\bigr)$ will be denoted by $[D,(d_1,\ldots,d_p)]$. \qed
\end{definition}

Given an antiautomorphism $\varphi$ and a fine $\varphi$-grading $\Gamma$ of a matrix algebra $R$, our results (see Remark \ref{re:fineGamma}) show that $\Gamma$ is equivalent to a fine $\varphi_B$-grading on some $\End_D(V)$ as above with a good basis $\calB$ such that \eqref{eq:good_basis} is satisfied.
Then define
\begin{equation}\label{eq:IRvarphiGamma}
\calI(R,\varphi,\Gamma)=[D,(B(v_1,v_1),\ldots,B(v_p,v_p))].
\end{equation}

\smallskip

Consider also the class of pairs $\bigl((D,\tau),(d_1,\ldots,d_p)\bigr)$ consisting of a central graded division algebra $D$ endowed with a graded involution $\tau$ and a $p$-tuple ($p\geq 0$) of nonzero homogeneous elements such that either $d_i^\tau=d_i$ for any $i$ or $d_i^\tau=-d_i$ for any $i$.

The relevance of the next definition will become clear in Theorems \ref{th:grading_involutions} and \ref{th:KR*}.

\begin{definition}\label{de:equivalenceDtaux1xp}
Two such pairs $\bigl((D,\tau),(d_1,\ldots,d_p)\bigr)$ and $\bigl((D',\tau'),(d_1',\ldots,d_q')\bigr)$ are said to be equivalent if $p=q$ and there exists a nonzero homogeneous element $z\in D$ and an isomorphism of graded algebras with involution $\Phi:(D,\tau^z)\rightarrow (D',\tau')$ (that is $\Phi\tau^z=\tau'\Phi$, where $x^{\tau^z}=zx^\tau z^{-1}$), and a permutation $\pi\in S_p$ such that
\[
\Phi(zd_i)\in \bF d_{\pi(i)}'\quad \forall i=1,\ldots,p.
\]

This is too an equivalence relation and the equivalence class of a pair like $\bigl((D,\tau),(d_1,\ldots,d_p)\bigr)$ will be denoted by $[(D,\tau),(d_1,\ldots,d_p)]$. \qed
\end{definition}

Note that since $z^2\in \bF 1$, the equation above is equivalent to the condition $\Phi(z^{-1}d_i)\in\bF d_{\pi(i)}'$.

\smallskip

Given an involution $\varphi$ and a fine $\varphi$-grading $\Gamma$ of a matrix algebra $R$, our results show that $\Gamma$ is equivalent to a fine $\varphi_B$-grading on some $\End_D(V)$ as above, with $B$ being hermitian or skew-hermitian (see Proposition \ref{pr:involutions}) relative to a graded involution $\tau$ of $D$, and with a good basis $\calB$ satisfying \eqref{eq:good_basis}. Then define
\begin{equation}\label{eq:I2RvarphiGamma}
\calI_2(R,\varphi,\Gamma)=[(D,\tau),(d_1,\ldots,d_p)].
\end{equation}

\begin{theorem}\label{th:calI_cal2I}
Let $\varphi$ be an antiautomorphism and $\Gamma$ a fine $\varphi$-grading of a matrix algebra $R$. Then $\calI(R,\varphi,\Gamma)$ is well defined.

Moreover, if $\varphi$ is an involution, then $\calI_2(R,\varphi,\Gamma)$ is well defined too.
\end{theorem}
\begin{proof}
For the first part, once $R$ is identified with $\End_D(V)$, where $V$ is the unique, up to isomorphism, irreducible graded module (Proposition \ref{pr:uniqueirreducible} and Theorem \ref{th:induced_grading}), and $D$ its centralizer, then the fine grading $\Gamma$ is identified with the fine $\varphi_B$-grading $\Gamma_V$ for a suitable sesquilinear nondegenerate balanced form. Now the results follows from the two following facts:
\begin{itemize}
\item $B$ is uniquely determined up to multiplication on the left by nonzero homogeneous elements in $D$ (Theorem \ref{th:grading_varphi_B}).
\item The elements of a good basis of $V$ are uniquely determined up to reordering and scaling by nonzero homogeneous elements in $D$ (Remark \ref{re:goodbasis}).
\end{itemize}

The same arguments are valid for the second part of the Theorem.
\end{proof}

\smallskip

The invariant $\calI(R,\varphi,\Gamma)$ does not classify $\varphi$-gradings up to equivalence, but up to a weaker condition, which is the relevant one when dealing with simple Lie algebras of type $A$.

\begin{theorem}\label{th:equivalence_varphi_gradings}
\null\quad\null
\begin{enumerate}
\item
Let $\varphi$ be an antiautomorphism of the matrix algebra $R$, let $\Gamma: R=\oplus_{g\in G}R_g$ be a fine $\varphi$-grading and let $\Psi\in\Diag_\Gamma(R)$ (the diagonal group of $\Gamma$).
%, where
%\[
%\Diag_\Gamma(R)=\{\Phi\in\Aut R: \Phi\vert_{R_g}\ \text{is a scalar multiple %of the identity}\ \forall g\in G\}.
%\]
Then $\Gamma$ is a fine $\Psi\varphi$-grading too and $\calI(R,\varphi,\Gamma)=\calI(R,\Psi\varphi,\Gamma)$.

\item
Let $\varphi_1$ and $\varphi_2$ be two antiautomorphisms of $R$, let $\Gamma_i$ be a fine $\varphi_i$-grading, $i=1,2$, then $\calI(R,\varphi_1,\Gamma_1)$ equals $\calI(R,\varphi_2,\Gamma_2)$ if and only if there exists a grading automorphism $\Phi:(R,\Gamma_1)\rightarrow (R,\Gamma_2)$ and an automorphism $\Psi\in\Diag_{\Gamma_2}(R)$ such that $\Phi\varphi_1\Phi^{-1}=\Psi\varphi_2$.
\end{enumerate}
\end{theorem}
\begin{proof}
For the first part note that since $\varphi$ respects the grading, then $\varphi\Psi=\Psi\varphi$ is another antiautomorphism respecting the grading, so that $\Gamma$ is a $\Psi\varphi$-grading and it is fine too. Identify $R$ with $\End_D(V)$ as before and, once a good basis $\calB$ is chosen, with $\Mat_{p+2s}(D)$. Since $\Psi$ is an automorphism of $R$, there is an invertible element $a\in \Mat_{p+2s}(D)$ such that $\Psi(u)=aua^{-1}$ for any $u$. But $1\in R_0$, so $\Psi(1)=1$ and hence the restriction of $\Psi$ to $R_0$ is the identity map. It follows that $a$ belongs to the centralizer $\Cent_R(R_0)$. Note that $R_0=\sum_{i=1}^{p+2s} E_{ii}\otimes \bF1$ and hence $a$ is of the form $a=\sum_{i=1}^{p+2s}E_{ii}\otimes a_i$ with $0\ne a_i\in D$ for any $i$. Also, all the subspaces $E_{ij}\otimes D_h$ are homogeneous (of degree $\degree(v_i)-\degree(v_j)+h$) so $\Psi$ acts as a scalar on each of them. In particular, considering the action of $\Psi$ on $E_{ii}\otimes D$ it follows that $z\mapsto a_iza_i^{-1}$ is a graded isomorphism of $D$, which by Lemma \ref{le:graded_involutions_D} gives that $a_i$ is homogeneous for any $i$. For any homogeneous $x\in D$ and $i\ne j$, $a(E_{ij}\otimes x)a^{-1}=\alpha E_{ij}\otimes x$ for some $0\ne\alpha\in \bF$, so $a_ixa_j^{-1}\in \bF x$ for any $i\ne j$ and hence $\degree(a_i)=\degree(a_j)$, It follows that $a=(\sum_{i=1}^{p+2s}\alpha_iE_{ii})\otimes d$ for a fixed nonzero homogeneous element $d$ of $D$ and nonzero scalars $\alpha_1,\ldots,\alpha_{p+2s}$. In particular $a$ is homogeneous of degree $\degree(d)$. Consider the $\tau$-sesquilinear nondegenerate balanced form given by
\[
\bar B(x,y)=B(x,a^{-1}y).
\]
For any $r\in R$,
\[
\bar B(rx,y)=B(rx,a^{-1}y)=B(x,\varphi(r)a^{-1}y)=B(x,a^{-1}\Psi\varphi(r)y)
 =\bar B(x,\Psi\varphi(r)y).
\]
Up to scaling by homogeneous elements in $D$, $\calB$ is a good basis for $\bar B$ too, and $\bar B(v_i,v_i)\in\bF B(v_i,v_i)d=\bF dB(v_i,v_i)$, which shows that $\calI(R,\varphi,\Gamma)=\calI(R,\Psi\varphi,\Gamma)$.

\smallskip

For the second part, if there exists a grading isomorphism $\Phi:(R,\Gamma_1)\rightarrow (R,\Gamma_2)$ then $\calI(R,\varphi_1,\Gamma_1)=\calI(R,\Phi\varphi_1\Phi^{-1},\Gamma_2)$. Moreover, if there is an automorphism $\Psi\in \Diag_{\Gamma_2}(R)$ such that $\Phi\varphi_1\Phi^{-1}=\Psi\varphi_2$, then we have $\calI(R,\varphi_1,\Gamma_1)=\calI(R,\Phi\varphi_1\Phi^{-1},\Gamma_2)=
\calI(R,\Psi\varphi_2,\Gamma_2)$ which equals $\calI(R,\varphi_2,\Gamma_2)$ by the the first part of the Theorem.

Conversely, assume $\calI(R,\varphi_1,\Gamma_1)=\calI(R,\varphi_2,\Gamma_2)$. This means that $R$ is graded isomorphic to $\End_D(V)$ with $\varphi_i$ corresponding to $\varphi_{B_i}$, $i=1,2$, for $\tau_i$-sesquilinear nondegenerate balanced forms $B_i$ on $V$ with good bases $\calB_1=\{v_i: i=1,\ldots,p+2s\}$ and $\calB_2=\{ w_i: i=1,\ldots,p+2s\}$ such that (after perhaps some reordering) the coordinate matrices of $B_1$ and $B_2$ are
\[
\Delta_1=\diag\left(d_1,\ldots,d_p,\bigl(\begin{smallmatrix} 0&\nu_1\\ \nu_1^{-1}&0\end{smallmatrix}\bigr),\ldots,\bigl(\begin{smallmatrix} 0&\nu_s\\ \nu_s^{-1}&0\end{smallmatrix}\bigr)\right)
\]
and
\[
\Delta_2=\diag\left(zd_1,\ldots,zd_p,\bigl(\begin{smallmatrix} 0&\gamma_1\\ \gamma_1^{-1}&0\end{smallmatrix}\bigr),\ldots,\bigl(\begin{smallmatrix} 0&\gamma_s\\ \gamma_s^{-1}&0\end{smallmatrix}\bigr)\right)
\]
for some nonzero homogeneous elements $z,d_1,\ldots,d_p\in D$ and nonzero scalars $\nu_1,\ldots,\nu_s$, $\gamma_1,\ldots,\gamma_s$.

There exists an homogeneous element $d\in D$ such that $x^{\tau_1}=dx^{\tau_2}d^{-1}$, and substituting $B_2(.,.)$ by $dB_2(.,.)$, we may assume that $\tau_1=\tau_2=\tau$. (The element $z$ and the $\gamma_i$'s may change in this process.)

Take now an element $a=(\sum_{i=1}^{p+2s}\alpha_iE_{ii})\otimes z$, with $0\ne \alpha_i\in \bF$ to be determined, and consider the $\tau$-sesquilinear form
\[
\bar B_2(x,y)=B_2(x,a^{-1}y),
\]
and the automorphism $\Psi$ of $R$ such that $\Psi(u)=aua^{-1}$. The adjoint map relative to $\bar B_2$ is $\Psi\varphi$ and we already know that $\Psi$ acts as a scalar on each subspace $E_{ij}\otimes D_h$ for $h\in H$ and on $R_h=\sum_{i=1}^{p+2s}\bigl(E_{ii}\otimes D_h\bigr)$ (see Remark \ref{re:homogeneous_subspaces}). In order to ensure that $\Psi\in\Diag_{\Gamma_2}R$ it is enough that $\Psi$ commute with $\varphi_2$, as any other homogeneous subspace is of the form $\bigl(E_{ij}\otimes D_h\bigr)+\varphi_B\bigl(E_{ij}\otimes D_h\bigr)$. Note that
\[
\Psi\varphi_2\Psi^{-1}(r)=\Psi\varphi_2(a^{-1}ra)
 =a\varphi_2(a)\varphi_2(r)\varphi_2(a^{-1})a^{-1},
\]
so $\Psi\varphi_2\Psi^{-1}$ equals $\varphi_2$ if and only if $a\varphi_2(a)\in\bF^\times 1$. We may restrict to the case $a\varphi_2(a)=1$ and using \eqref{eq:varphiBDelta} this is equivalent to the conditions:
\begin{equation}\label{eq:alphasz}
\left\{%
\begin{aligned}
&\alpha_i^2z(zd_i)^{-1}z^\tau zd_i=1\quad\forall i=1,\ldots,p,\\
&\alpha_{p+2j-1}\alpha_{p+2j}zz^\tau=1\quad\forall j=1,\ldots,s.
\end{aligned}
\right.
\end{equation}
Fix $\alpha_i$, $i=1,\ldots,p$, satisfying the first $p$ equations in \eqref{eq:alphasz} (they are uniquely determined up to a $\pm$ sign).
Note that $\bar B_2(w_i,w_i)=
B_2(w_i,\alpha_i^{-1}w_iz^{-1})
=\alpha_i^{-1}zd_iz^{-1}\in\bF d_i$ for any $i=1,\ldots,s$, so we may scale $w_i$ to $\bar w_i$ such that $\bar B_2(\bar w_i,\bar w_i)=d_i$. Also,
\[
\begin{split}
\bar B_2(w_{p+2j-1},w_{p+2j})
 &=B_2(w_{p+2j-1},a^{-1}w_{p+2j})=\gamma_j\alpha_{p+2j}^{-1}z^{-1},\\
\bar B_2(w_{p+2j},w_{p+2j-1})
 &=B_2(w_{p+2j},a^{-1}w_{p+2j-1})=\gamma_j^{-1}\alpha_{p+2j-1}^{-1}z^{-1}
  =\gamma_j^{-1}\alpha_{p+2j}z^\tau.
\end{split}
\]
(Use \eqref{eq:alphasz}.)

Let $\eta=z^2(\in\bF 1)$ and consider the elements $\bar w_{p+2j-1}=\sqrt{\eta^{-1}}w_{p+2j-1}$, $\bar w_{p+2j}=w_{p+2j}z$, then
\[
\begin{split}
\bar B_2(\bar w_{p+2j-1},\bar w_{p+2j})
 &=\sqrt{\eta^{-1}}\gamma_j\alpha_{p+2j}^{-1},\\
\bar B_2(\bar w_{p+2j},\bar w_{p+2j-1})
 &=(z^\tau)^2\gamma_j^{-1}\alpha_{p+2j}\sqrt{\eta^{-1}}
 =\sqrt{\eta}\gamma_j^{-1}\alpha_{p+2j},
\end{split}
\]
so with $\alpha_{p+2j}=\sqrt{\eta^{-1}}\gamma_j\nu_j^{-1}$ we get
\[
\bar B_2(\bar w_{p+2j-1},\bar w_{p+2j})=\nu_j,\quad
\bar B_2(\bar w_{p+2j},\bar w_{p+2j-1})=\nu_j^{-1}.
\]
Hence the good basis $\{\bar w_i,i=1,\ldots,p+2s\}$ for $\bar B_2$ has the same coordinate matrix $\Delta_1$ and adjoint map given by $\Psi\varphi_2$. The $D$-linear map $v_i\mapsto \bar w_i$, $i=1,\ldots,p+2s$, induces the required automorphism $\Phi:(R,\Gamma_1)\rightarrow (R,\Gamma_2)$ such that $\Phi\varphi_1\Phi^{-1}=\Psi\varphi_2$.
\end{proof}

\begin{corollary}\label{co:varphiB_gradings} \textbf{(of the proof)}
Let $V$ be a graded right module for a central graded division algebra $D$. Let $\tau_1$, $\tau_2$ be two graded involutions of $D$.  Consider nonzero homogeneous elements $z,d_1,\ldots,d_p\in D$ and nonzero scalars $\nu_1,\ldots,\nu_s\in \bF$. Then the fine $\varphi_B$-grading induced by the $\tau_1$-sesquilinear form $B$  with coordinate matrix
\[
\diag\left(d_1,\ldots,d_p,\bigl(\begin{smallmatrix} 0&\nu_1\\ \nu_1^{-1}&0\end{smallmatrix}\bigr),\ldots,\bigl(\begin{smallmatrix} 0&\nu_s\\ \nu_s^{-1}&0\end{smallmatrix}\bigr)\right)
\]
in a good basis is equivalent to the fine $\varphi_{B'}$-grading on $\End_D(V)$ induced by the $\tau_2$-sesquilinear form $B'$ with coordinate matrix
\[
\diag\left(zd_1,\ldots,zd_p,\bigl(\begin{smallmatrix} 0&1\\ 1&0\end{smallmatrix}\bigr),\ldots,\bigl(\begin{smallmatrix} 0&1\\ 1&0\end{smallmatrix}\bigr)\right). \qed
\]
\end{corollary}

\begin{remark}\label{re:varphiB_gradings}
Under the conditions above, notice that $\varphi_{B'}^2(u)=\Gamma u\Gamma^{-1}$ with
\[
\Gamma=\diag(
\epsilon_1,\ldots,\epsilon_p,1,1,\ldots,1,1),
\]
with $\epsilon_i=1$ if $(zd_i)^{\tau'}=zd_i$ and $\epsilon_i=-1$ if $(zd_i)^{\tau'}=-zd_i$.

Therefore, given an antiautomorphism $\varphi$ and a fine $\varphi$-grading $\Gamma$ on the matrix algebra $R$ with $\calI(R,\varphi,\Gamma)=[D,(d_1,\ldots,d_p)]$, if there is an homogeneous element $z\in D$ and a graded involution $\tau'$ of $D$ such that $(zd_i)^{\tau'}=zd_i$ for all $i$, then $\Gamma$ is equivalent to a fine $\varphi'$-grading for an involution $\varphi'$. Otherwise, $\Gamma$ is equivalent to a fine $\varphi'$-grading for an antiautomorphism $\varphi'$ of order $4$. \qed
\end{remark}

\begin{example}\label{ex:counterex_Bahturinetal}
Let $D=Q$ be our simplest nontrivial graded division algebra , as in \eqref{eq:quaternion_units}, and let $V$ be a free right $D$-module of dimension $4$, and let $B:V\times V\rightarrow D$ be the sesquilinear form relative to the orthogonal involution $\tau_o$ in equation \eqref{eq:tau_o} with coordinate matrix
\[
\diag(1,q_1,q_2,q_3)
\]
in an homogeneous basis $\{v_1,v_2,v_3,v_4\}$ of $V$. Given any homogeneous element $z\in Q$, the tuple $\{z1,zq_0,zq_1,zq_2\}$ is again an homogeneous basis of $D=Q$, and hence there is no involution in $Q$ for which these basic elements are all symmetric. Consider the $\bZ_4^2\times\bZ_2$-grading on $R=\End_D(V)$ with
\[
\degree(v_1)=0,\ \degree(v_2)=(\bar 1,0,0),\ \degree(v_3)=(0,\bar 1,0),\ \degree(v_4)=(\bar 3,\bar 3,\bar 1),
\]
which is a fine $\varphi_B$-grading. Then there is no involution $\varphi'$ of $R$ such that this grading be isomorphic to a $\varphi'$-grading (Theorem \ref{th:equivalence_varphi_gradings} and Corollary \ref{co:varphiB_gradings}).
This provides a counterexample to \cite[Proposition 6.4]{BZ}. \qed
\end{example}

\medskip

For gradings related to involutions of matrix algebras we have the following result, which shows that the invariant $\calI_2$ classifies gradings up to equivalence:

\begin{theorem}\label{th:grading_involutions}
Let $\varphi_1$ and $\varphi_2$ be two involutions of the matrix algebra $R$, let $\Gamma_i$ be a fine $\varphi_i$-grading for $i=1,2$. Then  $\calI_2(R,\varphi_1,\Gamma_1)$ equals $\calI_2(R,\varphi_2,\Gamma_2)$ if and only if there is a graded automorphism $\Phi:(R,\Gamma_1)\rightarrow (R,\Gamma_2)$ such that $\Phi\varphi_1\Phi^{-1}=\varphi_2$.
\end{theorem}
\begin{proof}
Assume first $\calI_2(R,\varphi_1,\Gamma_1)=\calI_2(R,\varphi_2,\Gamma_2)$. Our algebra
$R$ is graded isomorphic to $\End_D(V)$ with $\varphi_i$ corresponding to $B_i$, $i=1,2$, for a $\tau_i$-hermitian or skew-hermitian nondegenerate form $B_i$ defined on $V$, and there is a nonzero homogeneous element $z\in D$ such that $x^{\tau_2}=zx^{\tau_1}z^{-1}$ for any $x\in D$, and there are good bases $\calB_1=\{v_1,\ldots,v_{p+2s}\}$ and $\calB_2=\{w_1,\ldots,w_{p+2s}\}$ of $V$ such that the coordinate matrices of $B_1$ and $B_2$ are
\[
\Delta_1=\diag\left(d_1,\ldots,d_p,\bigl(\begin{smallmatrix} 0&1\\ \pm 1&0\end{smallmatrix}\bigr),\ldots,\bigl(\begin{smallmatrix} 0&1\\ \pm 1&0\end{smallmatrix}\bigr)\right)
\]
and
\[
\Delta_2=\diag\left(zd_1,\ldots,zd_p,\bigl(\begin{smallmatrix} 0&1\\ \pm 1&0\end{smallmatrix}\bigr),\ldots,\bigl(\begin{smallmatrix} 0&1\\ \pm 1&0\end{smallmatrix}\bigr)\right)
\]
for suitable homogeneous elements $d_1,\ldots,d_p$ in $D$. Substituting $B_2$ by $z^{-1}B_2$ we may assume that $\tau_1=\tau_2$ and  (after adjusting $w_{p+2},\ldots,w_{p+2s}$) that $\Delta_1=\Delta_2$. Hence the result follows at once, and the converse is clear.
\end{proof}

\smallskip

\begin{example}\label{ex:R*_D4}
Let $\varphi$ be an orthogonal involution of the algebra $R=\Mat_8(\bF)$. According to Theorem \ref{th:grading_involutions}, the fine $\varphi$-gradings of $R$ are determined, up to equivalence, by the different possibilities for $\calI_2(R,\varphi,\Gamma)$. The possibilities for the associated graded division algebra are the following:
\begin{description}
\item[\fbox{$D=\bF$}] Here the possibilities for $\calI_2(R,\varphi,\Gamma)$ are $[(\bF,id),(1,\stackrel{r}{\ldots},1)]$ , with $r=0,2,4,6$ or $8$, which correspond, respectively, to gradings over $\bZ^4$, $\bZ_2\times\bZ^3$, $\bZ_2^3\times\bZ^2$, $\bZ_2^5\times\bZ$ and $\bZ_2^7$. ($5$ possibilities.)

\item[\fbox{$D=Q$}] Definition \ref{de:equivalenceDtaux1xp} shows that $\tau$ may be taken to be $\tau_o$ and hence $d_1,\ldots,d_p$ must be symmetric (as $\varphi$ is orthogonal). Note that we can reorder the elements $d_1,\ldots,d_p$, that these can be scaled and that if, for instance $d_1=\cdots=d_q$, then by multiplying by $z=d_1^{-1}$ we may assume $d_1=\cdots=d_q=1$. In this way we may always assume that if $p\geq 1$, then $d_1=1$ and $1$ is the element in the sequence that appears a greater number of times. Thus the possibilities for $\calI_2(R,\varphi,\Gamma)$ are $[(Q,\tau_o),(d_1,\ldots,d_p)]$ with $(d_1,\ldots,d_p)=\emptyset, (1,1), (1,q_1), (1,1,1,1),(1,1,1,q_1),(1,1,q_1,q_1)$ or $(1,1,q_1,q_2)$, which correspond respectively to gradings over $\bZ^2\times\bZ_2^2$, $\bZ\times \bZ_2^3$, $\bZ\times\bZ_2\times\bZ_4$, $\bZ_2^5$, $\bZ_2^3\times\bZ_4$, $\bZ_2^3\times\bZ_4$ and $\bZ_2\times\bZ_4^2$. ($7$ possibilities.)

\item[\fbox{$D=Q^{\otimes 2}$}] Here $\tau$ may be taken to be $\tau_o\otimes\tau_o$ and
    \[
    \calI_2(R,\varphi,\Gamma)=[(Q^{\otimes 2},\tau),(d_1,\ldots,d_p)],
     \]
     for $p=0$ or $2$. If $p=2$, $d_1$ and $d_2$ are symmetric for $\tau$ and we may multiply $d_1$ and $d_2$ by $d_1^{-1}$ and assume $d_1=1$ ($\tau^{d_1}$ remains orthogonal so it is again isomorphic to $\tau_o\otimes\tau_o$ by Corollary \ref{co:graded_division_involution}). The argument in the proof of this Corollary shows that we may assume that $d_2=1$ or $d_2=1\otimes q_1$. But the case of $d_2=1$ does not give a fine grading (Proposition \ref{pr:GammatildeV}), so we are left with two possibilities: $p=0$ or $p=2$ and $(d_1,d_2)=(1\otimes 1,1\otimes q_1)$, which correspond respectively to gradings over $\bZ\times\bZ_2^4$ and $\bZ_2^3\times\bZ_4$. ($2$ possibilities.)

\item[\fbox{$D=Q^{\otimes 3}$}] Here $\calI_2(R,*,\Gamma)=[(Q^{\otimes 3},\tau),(d_1)]$ and, as above, we may assume $\tau=\tau_o\otimes\tau_o\otimes\tau_o$ and $d_1=1$. This gives a $\bZ_2^6$-grading. ($1$ possibility.)
\end{description}
Therefore, there are $15$ fine $\varphi$-gradings, up to equivalence. \qed
\end{example}

\medskip

There appears the question of which fine gradings of the matrix algebra $R=\Mat_n(\bF)$ are $\varphi$-gradings for some antiautomorphism $\varphi$. Only a few of them are. This follows from the previous results:

\begin{proposition}\label{pr:fine_gradings_varphi}
Let $\Gamma:R=\oplus_{g\in G}R_g$ be a fine grading of the matrix algebra $R=\Mat_n(\bF)$ and let $e\in R$ be a graded primitive idempotent of $(R,\Gamma)$. Then $\Gamma$ is a $\varphi$-grading for some antiautomorphism $\varphi$ if and only if the central graded division algebra $D=eRe$ is graded isomorphic to $Q^{\otimes m}$ for some $m$, and either $n=2^m$ (so that $R=D$, $e=1$) or $n=2^{m+1}$ (so that $R=\Mat_2(\bF)\otimes D$).
\end{proposition}
\begin{proof}
In case $R=D\simeq Q^{\otimes m}$ it is clear that $\Gamma$ is a $\varphi$-grading with $\varphi\simeq \tau_o^{\otimes m}$, while in case $R\simeq \Mat_2(\bF)\otimes D$, with $D\simeq Q^{\otimes m}$ then (Theorem \ref{th:fine_gradings_R}) $\Gamma$ is the tensor product of the Cartan grading on $\Mat_2(\bF)$ and the grading on $D$, and hence $\Gamma$ is a $\varphi$-grading with $\varphi=\tau_s\otimes \tau$, where $\tau_s$ is the symplectic involution in \eqref{eq:tau_s} and $\tau$ is an arbitrary graded involution on $D$.

Conversely, if $\Gamma$ is a $\varphi$-grading, then $D=eRe$ is endowed with a graded involution and Corollary \ref{co:graded_division_involution} shows that $D$ is graded isomorphic to $Q^{\otimes m}$ for some $m$. Moreover, since $\Gamma$ is fine, it is the tensor product of a Cartan grading (a $\bZ^{r-1}$-grading on $\Mat_r(\bF)$) and a $\bZ_2^{2m}$-grading on $D$. But \eqref{eq:relations} shows that this grading is a $\varphi$-grading only if $r=1,2$.
\end{proof}

\bigskip

\section{Fine gradings on simple Lie algebras of type $A$}\label{se:A}

The groups of automorphisms of the simple classical Lie algebras are explicitly computed in \cite[Chapter IX]{Jacobson}. For the special linear simple Lie algebra $\frsl_n(\bF)$ we obtain:

\begin{itemize}
\item Any automorphism of $\frsl_2$ is the restriction of an automorphism of $\Mat_2(\bF)$, so there is a natural isomorphism $\Aut\frsl_2(\bF)\simeq \Aut\Mat_2(\bF)$.

\item Any automorphism of $\frsl_n(\bF)$, $n\geq 3$, is either the restriction of an automorphism of $\Mat_n(\bF)$ or of the negative of an antiautomorphism of $\Mat_n(\bF)$ Denote by $\Antiaut R$ the set of antiautomorphisms of the algebra $R$. Then by restriction we get an isomorphism
    \[
    \Aut\frsl_n(\bF)\simeq \Aut\Mat_n(\bF)\cup\bigl(-\Antiaut\Mat_n(\bF)\bigr)\ \text{(disjoint union).}
    \]
\end{itemize}

Given any MAD of $\frsl_2(\bF)$, by identifying $\Aut\frsl_2(\bF)$ with $\Aut\Mat_2(\bF)$, it induces a fine grading on $\Mat_2(\bF)$ and two MADs are conjugated in $\Aut\frsl_2(\bF)$ if and only if so are they considered as MADs of $\Aut\Mat_2(\bF)$.

Therefore the fine gradings on $\frsl_2(\bF)$ are just the restrictions to $\frsl_2(\bF)$ of the fine gradings on $\Mat_2(\bF)$ which are determined in Theorem \ref{th:fine_gradings_R}. Note that if $\Gamma:R=\oplus_{g\in G}R_g$ is a fine grading of $R=\Mat_2(\bF)$, then $1\in R_0$ and $R_g\subseteq \frsl_2(\bF)$ for any $g\ne 0$, as any automorphism of $R$ leaves invariant $\frsl_2(\bF)=[R,R]$ and the homogeneous components are the common eigenspaces of the elements in the associated MAD. Hence $\Gamma$ induces the grading (also denoted by $\Gamma$) with $\frsl_2(\bF)_0=R_0\cap \frsl_2(\bF)$ and $\frsl_2(\bF)_g=R_g$ for any $g\ne 0$. (This argument applies to $\frsl_n(\bF)$ for arbitrary $n\geq 2$.)

Thus, there are, up to equivalence, two fine gradings on $\frsl_2(\bF)$: the Cartan grading over $\bZ$ and the division grading over $\bZ_2^2$ with
\begin{equation}\label{eq:sl2}
\frsl_2(\bF)_{(\bar 1,\bar 0)}=\bF q_1,\
\frsl_2(\bF)_{(\bar 0,\bar 1)}=\bF q_2,\
\frsl_2(\bF)_{(\bar 1,\bar 1)}=\bF q_3,
\end{equation}
with $q_1,q_2,q_3$ in \eqref{eq:quaternion_units}.

\smallskip

Assume now that $n\geq 3$, then there are two types of MADs in $\Aut\frsl_n(\bF)$. A MAD $M$ will be said to be \emph{inner} if it is contained in $\Aut\Mat_n(\bF)$ (once $\Aut\frsl_n(\bF)$ is identified to
$\Aut\Mat_n(\bF)\cup\bigl(-\Antiaut\Mat_n(\bF)\bigr)$), otherwise it is called \emph{outer}. If a MAD $M$ is inner, it is a MAD in $\Aut\Mat_n(\bF)$ so it induces a fine grading on $R=\Mat_n(\bF)$. Proposition \ref{pr:fine_gradings_varphi} shows then that a MAD $N$ in $\Aut R$ remains a MAD in $\Aut\frsl_n(\bF)$ unless $n=2^m$ for some $m$ and the associated grading $\Gamma$ satisfies that $\calI(R,\Gamma)=[Q^{\otimes \bar m}]$ with either $\bar m=m$ ($R\simeq Q^{\otimes m}$) or $\bar m=m-1$ ($R\simeq \Mat_2(\bF)\otimes Q^{\otimes\bar m}$).

On the other hand, if $M$ is an outer MAD, it contains the negative of an antiautomorphism $\varphi$ and then, with $M^{int}=M\cap\Aut\Mat_n(\bF)$, it follows that
\[
M=M^{int}\cup M^{int}(-\varphi)\ \text{(disjoint union).}
\]
(Note that if $\psi$ is another antiautomorphism with $-\psi\in M$, then $-\psi=(\psi\varphi^{-1})(-\varphi)$ belongs to $M^{int}(-\varphi)$ since $\psi\varphi^{-1}$ is an automorphism of $\Mat_n(\bF)$ contained in $M$.)

Moreover, the maximality of $M$ and the fact that $\varphi^2=(-\varphi)^2\in M^{int}$ shows that $M^{int}$ induces a fine $\varphi$-grading $\Gamma$ of $R=\Mat_n(\bF)$. Besides, the grading $\Gamma_\varphi$ induced by $M$ in $\frsl_n(\bF)$ is given by the eigenspaces of the action of $M=M^{int}\cup M^{int}(-\varphi)$, that is, it is given by the eigenspaces of $M^{int}\cup\{\varphi\}$. Thus, if $\Gamma:R=\oplus_{g\in G}R_g$ is the $\varphi$-grading induced by $M^{int}$, then $\varphi^2\vert_{R_g}=\alpha_g id$ for some $0\ne\alpha_g\in \bF$ with $\alpha_0=1$ (see Definition \ref{de:varphi_grading}) and we may fix a square root $\sqrt{\alpha_g}$ (with $\sqrt{1}=1$). The homogeneous components of the grading $\Gamma_\varphi$ induced by $M$ are
\[
\frsl_n(\bF)_{g^+}=\{x\in R_g: \varphi(x)=\sqrt{\alpha_g}x\},\
\frsl_n(\bF)_{g^-}=\{x\in R_g: \varphi(x)=-\sqrt{\alpha_g}x\}
\]
for $g\ne 0$, and
\[
\frsl_n(\bF)_{0^+}=\{x\in R_0: \varphi(x)=x,\ \trace(x)=0\},\
\frsl_n(\bF)_{0^-}=\{x\in R_0: \varphi(x)=-x\}.
\]

\smallskip

\begin{remark}\label{re:grading_group_A}
Let $M$ be an outer MAD as above, because of Remark \ref{re:varphiB_gradings} $M=M^{int}\cup M^{int}(-\varphi)$ for an antiautomorphism $\varphi$ of order $2$ or $4$. Let $\Gamma: R=\oplus_{g\in G} R_g$ be the $\varphi$-grading induced by $M^{int}$. It restricts to a grading $\Gamma_0:\frsl_n(\bF)=\oplus_{g\in G}\frsl_n(\bF)_g$. If $\varphi^2=id$, then the fine grading induced by $M$ is
\[
\frsl_n(\bF)=\oplus_{h\in G\times\bZ_2}\frsl_n(\bF)_h,
\]
with
\[
\begin{split}
\frsl_n(\bF)_{(g,\bar 0)}&=\{x\in\frsl_n(\bF)_g: (-\varphi)(x)=x\}\\
    &=\{x\in\frsl_n(\bF)_g: \varphi(x)=-x\},\\[3pt]
\frsl_n(\bF)_{(g,\bar 1)}&=\{x\in\frsl_n(\bF)_g: \varphi(x)=x\}.
\end{split}
\]
In case $\varphi^4=id\ne \varphi^2$, let $\epsilon$ be a primitive fourth root of unity, then the fine grading induced by $M$ is
\[
\frsl_n(\bF)=\oplus_{h\in G\times\bZ_4}\frsl_n(\bF)_h,
\]
with
\[
\begin{split}
\frsl_n(\bF)_{(g,\bar \imath)}&=\{x\in\frsl_n(\bF)_g:
        (-\varphi)(x)=\epsilon^i x\}\\
    &=\{x\in\frsl_n(\bF)_g: \varphi(x)=-\epsilon^i x\}.
\end{split}
\]
Note that the subgroup of $G\times\bZ_4$ generated by the support of this grading is never the whole $G\times\bZ_4$, because for any $g\in G$, the restriction $(-\varphi)^2\vert_{R_g}=\varphi^2\vert_{R_g}$ acts as $\pm id$, so this subgroup is contained in
\[
\Bigl(\{g\in G: \varphi^2\vert_{R_g}=id\}\times\{\bar 0,\bar 2\}\Bigr)\cup
\Bigl(\{g\in G: \varphi^2\vert_{R_g}=-id\}\times\{\bar 1,\bar 3\}\bigr),
\]
which is a proper subgroup of $G\times \bZ_4$. \qed
\end{remark}

\medskip

Two MADs of different types (inner and outer) cannot be conjugated in $\Aut\frsl_n(\bF)$, because $\Aut\Mat_n(\bF)$ is a normal subgroup of $\Aut\frsl_n(\bF)$. If two MADs $M_1$ and $M_2$ of inner type are conjugated by an automorphism of $R=\Mat_n(\bF)$ then the induced fine gradings $\Gamma_1$ and $\Gamma_2$ of $R$ are equivalent and this happens if and only if $\calI(R,\Gamma_1)=\calI(R,\Gamma_2)$ (Theorem \ref{th:fine_gradings_R_invariant}). In case they are conjugated by $-\varphi$ for an antiautomorphism $\varphi$ and $e$ is a graded primitive idempotent of $(R,\Gamma_1)$, then $eRe$ is a central graded division subalgebra of $(R,\Gamma_1)$ and so is $\varphi(eRe)=\varphi(e)R\varphi(e)$ for $(R,\Gamma_2)$. Hence $\varphi(e)$ is a graded primitive idempotent of $(R,\Gamma_2)$. In particular the graded division algebras $eRe$ and $\varphi(e)R\varphi(e)$ are antiisomorphic. But the structure of the central graded division algebras (Proposition \ref{pr:graded_division}) shows that two such algebras are antiisomorphic if and only if they are isomorphic, because the algebras $A_n$ in \eqref{eq:divisionAn} have antiautomorphisms (the assignment $x\leftrightarrow y$ induces an antiautomorphism). Hence $\calI(R,\Gamma_1)=\calI(R,\Gamma_2)$ in this case too. (Recall that $\calI(R,\Gamma_1)$ is just the isomorphism class of the central graded division algebra $eRe$.)

On the other hand, if two MADs of outer type $M_1=M_1^{int}\cup M_1^{int}(-\varphi_1)$ and $M_2=M_2^{int}\cup M_2^{int}(-\varphi_2)$
are conjugated by the negative of an antiautomorphism $\psi$: $(-\psi)M_1(-\psi)^{-1}=\psi M_1\psi^{-1}=M_2$, then also $(\psi\varphi_1)M_1(\psi\varphi_1)^{-1}=\psi M_1\psi^{-1}=M_2$, so they are conjugated by an element $\phi$ in $\Aut\Mat_n(\bF)$. But then $\phi\varphi_1\phi^{-1}\in M_2^{int}\varphi_2$ and therefore two such MADs $M_1$ and $M_2$ are conjugated if and only if there is an element $\psi\in M_2^{int}$ such that the $\varphi_1$-grading $\Gamma_1$ induced by $M_1^{int}$ is equivalent to the $(\psi\varphi_2)$-grading $\Gamma_2$ induced by $M_2^{int}$. By Theorems \ref{th:calI_cal2I} and \ref{th:equivalence_varphi_gradings} this happens if and only if $\calI(R,\varphi_1,\Gamma_1)=\calI(R,\varphi_2,\Gamma_2)$ ($R=\Mat_n(\bF)$).

\smallskip

We summarize our discussion in the next result:

\begin{theorem}\label{th:gradings_sln}
Let $R$ be the matrix algebra $\Mat_n(\bF)$, $n\geq 2$. Then:
\begin{enumerate}
\renewcommand*{\labelenumii}{(\theenumi.\theenumii)}
\item Up to equivalence, the only fine gradings of $\frsl_2(\bF)$ are its Cartan grading and the division grading over $\bZ_2^2$ in \eqref{eq:sl2}.

    \item If $n\geq 3$, any fine grading of $\frsl_n(\bF)$ is either:
    \begin{enumerate}
        \item the restriction of a fine grading $\Gamma$ of $R$ such that $\calI(R,\Gamma)\ne [Q^{\otimes m}]$ with $2^m\in\{n,\frac{n}{2}\}$.
        \item the fine grading $\Gamma_\varphi$ induced by a fine $\varphi$-grading of $R$, where $\varphi$ is an antiautomorphism of $F$.
    \end{enumerate}
    Moreover, the gradings in different items are not equivalent. Two gradings of type \textup{(2.a)} are equivalent if and only if so are the gradings $\Gamma_1$ and $\Gamma_2$ of $R$, if and only if $\calI(R,\Gamma_1)=\calI(R,\Gamma_2)$; and two gradings $\Gamma_{\varphi_1}$ and $\Gamma_{\varphi_2}$ induced by the $ \varphi_i$-gradings $\Gamma_i$ ($i=1,2$) of type \textup{(2.b)} are equivalent if and only if $\calI(R,\varphi_1,\Gamma_1)=\calI(R,\varphi_2,\Gamma_2)$. \qed
\end{enumerate}
\end{theorem}

\smallskip

The possibilities of the fine gradings for $\frsl_4(\bC)$ have been treated in \cite{PPS02}.

Let us consider here a detailed example of greater complexity:

\begin{example}\label{ex:gradings_sl8}
\textbf{Fine gradings of $\frsl_8(\bF)$.}
\begin{description}
\item[\textbf{Inner}] Up to equivalence, there are the following five fine inner gradings on $\frsl_8(\bF)$.
    \begin{itemize}
    \item The Cartan grading over $\bZ^7$ (here $\calI(R,\Gamma)=[\bF]$).
    \item A $\bZ^3\times \bZ_2^2$-grading with $\calI(R,\Gamma)=[Q]$.
    \item A $\bZ\times \bZ_4^2$-grading with $\calI(R,\Gamma)=[A_4]$ (see \eqref{eq:divisionAn}).
    \item A $\bZ_2^2\times\bZ_4^2$-grading with $\calI(R,\Gamma)=[Q\otimes A_4]=[A_2\otimes A_4]$.
    \item A $\bZ_8^2$-grading with $\calI(R,\Gamma)=[A_8]$.
    \end{itemize}
    (Note that the fine gradings of $R=\Mat_8(\bF)$ with $\calI(R,\Gamma)=[Q^{\otimes 2}]$ or $[Q^{\otimes 3}]$ do not appear here as they induce outer gradings of $\frsl_8(\bF)$ because of Proposition \ref{pr:fine_gradings_varphi}.)

\item[\textbf{Outer}] Here the possibilities are given by the different classes $[D,(d_1,\ldots,d_p)]$, where $D$ is a central graded division subalgebra of $R=\Mat_8(\bF)$ with a graded involution, so that $D=Q^{\otimes r}$, $0\leq r\leq 3$, and $d_1,\ldots,d_p$ are homogeneous elements in $D$, $p$ is an even number with $0\leq p\leq \frac{8}{2r}$. Note that the definition of these classes (Definition \ref{de:equivalenceDx1xp}) implies that we can reorder the elements $d_1,\ldots,d_p$, that these can be scaled and that if, for instance $d_1=\cdots=d_q$, then by multiplying by $z=d_1^{-1}$ we may assume $d_1=\cdots=d_q=1$. In this way we may always assume that if $p\geq 1$, then $d_1=1$ and $1$ is the element in the sequence that appears a greater number of times. Note that we may choose an arbitrary graded involution on $D$ (see the proof of Theorem \ref{th:equivalence_varphi_gradings}). Therefore there are the following possibilities:

    \begin{description}
    \item[\fbox{$D=\bF$}] We must consider here the classes $[\bF,(1,\stackrel{p}{\ldots},1)]$, $p=0,2,4,6,8$. In this situation the grading is, up to equivalence the grading on $\frsl_8(\bF)$ induced by the $\varphi$-grading on $R=\Mat_8(\bF)\simeq\End_{\bF}(V)$, where $V$ is an eight dimensional vector space and $\varphi$ is given by the adjoint map relative to the symmetric bilinear form with coordinate matrix \[
        \diag\left(1,\ldots,1,
        \bigl(\begin{smallmatrix}0&1\\1&0\end{smallmatrix}\bigr),\ldots,
        \bigl(\begin{smallmatrix}0&1\\1&0\end{smallmatrix}\bigr)\right),
        \]
        where there are $p$ $1$'s and $\frac{1}{2}(8-p)$ blocks $\bigl(\begin{smallmatrix}0&1\\1&0\end{smallmatrix}\bigr)$. Hence $\varphi^2=id$ and therefore we have the following gradings:
        \begin{itemize}
        \item A $\bZ^4\times\bZ_2$-grading corresponding to $p=0$. Here the copy of $\bZ_2$ corresponds to the action of the involution $\varphi$.
        \item A $\bZ^3\times\bZ_2^2$-grading for $p=2$. Here the first copy of $\bZ_2$ corresponds to the grading of a good basis where $\degree(v_1)=0$ and $\degree(v_2)$ is an element of order $2$. The second copy corresponds to the action of $\varphi$.
        \item A $\bZ^2\times \bZ_2^4$-grading for $p=4$.
        \item A $\bZ\times \bZ_2^6$-grading for $p=6$.
        \item A $\bZ_2^8$-grading for $p=8$.
        \end{itemize}

    \item[\fbox{$D=Q$}] Here the equivalence classes to be considered are the classes $[Q,(d_1,\ldots,d_p)]$, with $p=0,2$ or $4$. We can always assume that each $d_i\in \{1,q_1,q_2,q_3\}$ in \eqref{eq:quaternion_units} and that $d_1=1$ is the element that appears the most. The grading is then, up to equivalence, the grading on $\frsl_8(\bF)$ induced by a $\varphi$-grading on $R\simeq \End_Q(V)$ ($\dim_Q(V)=4$) where $\varphi$ is given by the adjoint map relative to the sesquilinear form $B$ with coordinate matrix
        \[
        \diag\left(d_1,\ldots,d_p,
        \bigl(\begin{smallmatrix}0&1\\1&0\end{smallmatrix}\bigr),\ldots,
        \bigl(\begin{smallmatrix}0&1\\1&0\end{smallmatrix}\bigr)\right),
        \]
        the number of $2\times 2$ blocks being $\frac{1}{2}(4-p)$) (see Corollary \ref{co:varphiB_gradings} and Remark \ref{re:varphiB_gradings}). Here are the different possibilities:
        \begin{itemize}
        \item A $\bZ^2\times\bZ_2^3$-grading corresponding to $[Q,\emptyset]$. (The first two copies of $\bZ_2$ constitute the support of the grading of $Q$ and the last copy corresponds to the action of the involution $\varphi$.)
        \item A $\bZ\times \bZ_2^4$-grading which corresponds to $[Q,(1,1)]$.
        \item A $\bZ\times\bZ_4\times\bZ_2^2$-grading which corresponds to $[Q,(1,q_1)]$. Here we have a good basis $\calB=\{ v_1,v_2,v_3,v_4\}$ with $\degree(v_1)=0$, $\degree(v_2)=(0,\bar 1,\bar 0,\bar 0)\in \bZ\times \bZ_4\times\bZ_2\times\bZ_2$ and $\degree(v_3)=(1,\bar 0,\bar 0,\bar 0)=-\degree(v_4)$. The support of $Q$ is $0\times 2\bZ_4\times\bZ_2\simeq \bZ_2^2$ and the last copy of $\bZ_2$ corresponds to the action of the involution $\varphi$.
        \item A $\bZ_2^6$-grading which corresponds to $[Q,(1,1,1,1)]$.
        \item A $\bZ_4\times\bZ_2^4$-grading which corresponds to $[Q,(1,1,1,q_1)]$.
        \item A $\bZ_4\times\bZ_2^4$-grading which corresponds to $[Q,(1,1,q_1,q_1)]$. Here the fine $\varphi$-grading $\Gamma_V$ in Proposition \ref{pr:GammatildeV} is the grading on $\End_Q(V)$ where we have a good basis $\calB=\{v_1,v_2,v_3,v_4\}$ with $\degree(v_1)=0$, $\degree(v_2)=g_2$, $\degree(v_3)=g_3$ and $\degree(v_4)=g_4$ such that \eqref{eq:relationstilde}
            \[
            2g_2=2g_3+h_1=2g_4+h_1=0,
            \]
            with $h_1=\degree(q_1)$. The grading group $G$ of the $\varphi$-grading is generated by $g_2,g_3,g_4$ and $H=\group{h_1,h_2}=\Supp Q$, and hence it is generated too by the elements $g_3,g_2,g_4-g_3$ and $h_2$ with respective orders $4,2,2,2$, thus getting a $\varphi$-grading over $\bZ_4\times\bZ_2^3$. The action of $\varphi$ gives the last copy of $\bZ_2$.
        \item A $\bZ_4^2\times\bZ_2^2$-grading which corresponds to $[Q,(1,1,q_1,q_2)]$.
        \item A $\bZ_4^3$-grading which corresponds to $[Q,(1,q_1,q_2,q_3)]$. Here $\tau=\tau_o$ is the orthogonal involution of $A$, and as above we have a good basis $\calB=\{v_0,v_1,v_2,v_3\}$ (note that we have shifted the indices) with $\degree(v_0)=0$, $\degree(v_i)=g_i$, $i=1,2,3$, subject to
            \[
            2g_1+h_1=2g_2+h_2=2g_3+h_3=0
            \]
            ($h_i$ being the degree of $q_i$), and $\varphi$ is given by the adjoint map relative to the $\tau$-sesquilinear form $B$ with $B(v_i,v_i)=q_i$ $i=0,1,2,3$ with $q_0=1$, $B(v_i,v_j)=0$ for $i\ne j$. The subgroup generated by $g_1,g_2,g_3$ and $H$ is the group generated by $g_1,g_2$ and $g_1+g_2+g_3$ of respective orders $4,4,2$ (as $2(g_1+g_2+g_3)=h_1+h_2+h_3=0$). This group is $\bZ_4^2\times\bZ_2$ and it is the universal grading group of the $\varphi$-grading $\Gamma$. If $M$ is the MAD group of our grading, $M=M^{int}\cup M^{int}(-\varphi)$, and $M^{int}$ is then isomorphic to $\bZ_4^2\times\bZ_2$ and $M/M^{int}$ is isomorphic to $\bZ_2$. And since there is no homogeneous $z\in Q$ and involution $\tau'$ on $Q$ such that $z,zq_1,zq_2,zq_3$ are all symmetric, necessarily $M\simeq \bZ_4^3$. Alternatively one can compute explicitly the MAD $\Diag_{\Gamma}(R)$. (See Example \ref{ex:counterex_Bahturinetal}.)
        \end{itemize}

    \item[\fbox{$D=Q^{\otimes 2}$}] The possibilities now are:

        \begin{itemize}
        \item A $\bZ\times\bZ_2^5$-grading which corresponds to $[Q^{\otimes 2},\emptyset]$.
        \item A $\bZ_4\times\bZ_2^4$-grading which corresponds to $[Q^{\otimes 2},(1\otimes 1,1\otimes q_1)]$. (Note that the argument in the proof of Corollary \ref{co:graded_division_involution} shows that we may always take any homogeneous element of degree $\ne 0$  to be $1\otimes q_1$.)
        \end{itemize}

    \item[\fbox{$D=Q^{\otimes 3}$}] Only one more possibility appears here:
        \begin{itemize}
        \item A $\bZ_2^6$-grading which corresponds to $[Q^{\otimes 3},(1)]$. Actually the universal grading group here is $\bZ_2^7$, because the associated MAD group $M$ is $M^{int}\cup M^{int}(-\varphi)$, with $\varphi$ any graded involution of $Q^{\otimes 3}$ and $M^{int}\cong \bZ_2^6=(\bZ_2^2)^3$ the MAD of the division grading on the central graded division algebra $Q^{\otimes 3}$.
%The last sentence added on May 26, 2009
        \end{itemize}
    \end{description}
\end{description}

The conclusion is that, up to equivalence, there are exactly $21$ fine gradings of $\frsl_8(\bF)$.

Note that this number is easy to get, as it is enough to look at the different possibilities for $\calI(R,\varphi,\Gamma)$. The subtle point is to check what the grading group looks like in each case. \qed
\end{example}

\bigskip

\section{Fine gradings on orthogonal and symplectic Lie algebras}\label{se:finesosp}

Let $*$ be an involution of the matrix algebra $R=\Mat_n(\bF)$ and let $K(R,*)=\{x\in R: x^*=-x\}$ be the Lie algebra of skew symmetric elements for the involution. Thus, if $*$ is orthogonal, then $K(R,*)$ is isomorphic to the orthogonal Lie algebra $\frso_n(\bF)$, while if $*$ is symplectic, $n$ is even and $K(R,*)$ is isomorphic to the symplectic Lie algebra $\frsp_n(\bF)$.

The automorphism group of these Lie algebras is described in \cite[Chapter IX]{Jacobson}. The results there are rephrased in the following Proposition \ref{pr:AutR*}. Each automorphism of $R$ which commutes with the involution $*$ induces and automorphism of $K(R,*)$. Denote by $\Aut(R,*)$ this group of automorphisms: $\Aut(R,*)=\{ f\in \Aut R : f(x^*)=f(x)^*\ \forall x\in R\}$.

\begin{proposition}\label{pr:AutR*}
The restriction map
\[
\begin{split}
\Aut(R,*)&\longrightarrow \Aut K(R,*)\\
f\ &\mapsto \ f\vert_{K(R,*)}
\end{split}
\]
is an isomorphism of groups if $n\geq 5$ unless $n=6$ or $n=8$ and $*$ is orthogonal. \qed
\end{proposition}

Note that for orthogonal $*$, $K(\Mat_6(\bF),*)\simeq \frso_6(\bF)$ is isomorphic to $\frsl_4(\bF)$, while the automorphism group of $K(\Mat_8(\bF),*)\simeq \frso_8(\bF)$ is larger than the subgroup formed by the restrictions of the elements in $\Aut(\Mat_8(\bF),*)$ due to the existence of the triality automorphisms. This will be dealt with in the last section of the paper.

\smallskip

For $n\geq 5$ and $n\not\in\{6,8\}$ if $*$ is orthogonal, any MAD $M$ of $\Aut K(R,*)$ can be identified to a MAD of $\Aut(R,*)$, and hence it induces a fine $*$-grading of $R$, whose restriction to $K(R,*)$  is the fine grading induced by $M$. (If $R=\oplus_{g\in G}R_g$ is a $*$-grading, then its restriction $K(R,*)=\oplus_{g\in G}(K(R,*)\cap R_g)$ makes sense and it is a grading of $K(R,*)$.)

Moreover, two such MADs $M_1$ and $M_2$ are conjugated if and only if there exists a $\phi\in \Aut(R,*)$ such that $\phi M_1\phi^{-1}=M_2$, where $M_1$ and $M_2$ are considered as MADs in $\Aut(R,*)$ through the isomorphism in Proposition \ref{pr:AutR*}. Hence Theorems \ref{th:grading_involutions} and  \ref{th:calI_cal2I} imply our main result of this section:

\begin{theorem}\label{th:KR*}
Let $*$ be an involution of the matrix algebra $R=\Mat_n(\bF)$ and assume thatn $n\geq 5$ and that $n\ne 6,8$ if $*$ is orthogonal. The any fine grading of $K(R,*)$ is the restriction of a fine $*$-grading of $R$. Two such fine gradings $\Gamma_1$ and $\Gamma_2$ of $R$ restrict to equivalent gradings of $K(R,*)$ if and only if $\calI_2(R,*,\Gamma_1)=\calI_2(R,*,\Gamma_2)$. \qed
\end{theorem}

\begin{example}
\label{ex:gradings_sp8}
\textbf{Fine gradings of $\frsp_8(\bF)$.}\quad
Here $R=\Mat_8(\bF)$ and $*$ is a symplectic involution. The possibilities for the central graded division algebra attached to a fine $*$-grading $\Gamma$ are $D=Q^{\otimes m}$ for $m=0,1,2$ or $3$.
\begin{description}
\item[\fbox{$D=\bF$}] Here $\calI_2(R,*,\Gamma)=[(\bF,id),\emptyset]$ (there are no skew symmetric elements in $\bF$). Here $R$ is graded isomorphic to $\End_{\bF}(V)$ for an eight dimensional vector space $V$ endowed with a nondegenerate skew-symmetric form $B$. This gives a $\bZ^4$-grading.

\item[\fbox{$D=Q$}] Here $\calI_2(R,*,\Gamma)=[(Q,\tau),(d_1,\ldots,d_p)]$, $p=0,2,4$, for a graded involution $\tau$. Definition \ref{de:equivalenceDtaux1xp} shows that we may assume $\tau=\tau_s$, and hence $d_1,\ldots,d_p$ must be symmetric for $\tau_s$ since $*$ is symplectic, so that we may take $d_1=\ldots=d_p=1$. We are left with three possibilities:
    \begin{itemize}
    \item A $\bZ^2\times\bZ_2^2$-grading corresponding to $\calI_2(R,*,\Gamma)=[(Q,\tau_s),\emptyset)]$.
    \item A $\bZ\times\bZ_2^3$-grading corresponding to $\calI_2(R,*,\Gamma)=[(Q,\tau_s),(1,1)]$.
    \item A $\bZ_2^5$-grading corresponding to $\calI_2(R,*,\Gamma)=[(Q,\tau_s),(1,1,1,1)]$.
    \end{itemize}

\item[\fbox{$D=Q^{\otimes 2}$}] Here $\calI_2(R,*,\Gamma)=[(Q^{\otimes 2},\tau),(d_1,\ldots,d_p)]$, $p=0$ or $2$. As before, we may assume $\tau=\tau_s\otimes \tau_o$, which is symplectic. Also, if $p=2$, $d_1$ and $d_2$ are symmetric for $\tau$ and we may multiply $d_1$ and $d_2$ by $d_1^{-1}$ and assume $d_1=1$ ($\tau^{d_1}$ remains symplectic so it is again isomorphic to $\tau_s\otimes\tau_o$ by Corollary \ref{co:graded_division_involution}). The argument in the proof of this Corollary shows that we may assume that $d_2=1$ or $d_2=1\otimes q_1$. But the case of $d_2=1$ does not give a fine grading (Proposition \ref{pr:GammatildeV}), so we are left with two possibilities:
    \begin{itemize}
    \item A $\bZ\times\bZ_2^4$-grading corresponding to $\calI_2(R,*,\Gamma)=[(Q^{\otimes 2},\tau_s\otimes\tau_o),\emptyset]$.
    \item A $\bZ_4\times\bZ_2^3$-grading corresponding to $\calI_2(R,*,\Gamma)=[(Q^{\otimes 2},\tau_s\otimes\tau_o),(1,1\otimes q_1)]$.
    \end{itemize}

\item[\fbox{$D=Q^{\otimes 3}$}] Here $\calI_2(R,*,\Gamma)=[(Q^{\otimes 3},\tau),(d_1)]$ and, as above, we may assume $\tau=\tau_s\otimes\tau_o\otimes\tau_o$ and $d_1=1$. This gives a $\bZ_2^6$-grading.
\end{description}

The conclusion is that, up to equivalence, there are exactly $7$ fine gradings of $\frsp_8(\bF)$. \qed
\end{example}

\bigskip

\section{Fine gradings on $D_4$}\label{se:D4}

The orthogonal Lie algebra $\frso_8(\bF)$ is quite special among the orthogonal Lie algebras, due to the fact that if $*$ denotes an orthogonal involution on the matrix algebra $R=\Mat_8(\bF)$, so that we may identify $\frso_8(\bF)$ to $K(R,*)=\{x\in R: x^*=-x\}$, the automorphism group $\Aut\frso_8(\bF)$ is larger than the group $\Aut(R,*)$. In fact, $\Aut(R,*)$ is a subgroup of index $3$ in $\Aut\frso_8(\bF)$ due to the triality phenomenon. That is, there are outer automorphisms of order $3$.

The fine gradings of $D_4$ not involving outer automorphisms could in principle be determined following the ideas in Section \ref{se:finesosp}, but there is a subtle point here: two such fine gradings may not be equivalent under the action of $\Aut(R,*)$, while being so under the action of the full group $\Aut\frso_8(\bF)$. Actually, we will check that among the $15$ non equivalent fine gradings in Example \ref{ex:R*_D4}, two of them are equivalent under $\frso_8(\bF)$. On the other hand, the fine gradings whose associated MAD groups contain outer automorphisms not in $\Aut(R,*)$ have been determined in \cite{DMVpr}. Here a quite self contained proof will be provided not relying on computations on conjugacy classes in the Atlas \cite{Atlas}, although based on ideas already present in \cite{DMVpr}. The list of nonequivalent fine gradings in \cite{DMVpr} or \cite{DV} will be corrected.

\smallskip

The natural way to deal with the triality phenomenon is to start with the algebra of octonions $C$ over our ground field $\bF$. This algebra contains a basis
\begin{equation}\label{eq:basisC}
\calB=\{e_1,e_2,u_1,u_2,u_3,v_1,v_2,v_3\}
\end{equation}
with multiplication given by:
\[
\begin{split}
&e_j^2=e_j,\ (j=1,2)\\
&e_1u_i=u_i=u_ie_2,\ e_2v_i=v_i=v_ie_1,\ (i=1,2,3)\\
&u_iv_i=-e_1,\ v_iu_i=-e_2,\ (i=1,2,3)\\
&u_iu_{i+1}=-u_{i+1}u_i=v_{i+2},\ v_iv_{i+1}=-v_{i+1}v_i=u_{i+2},\ \text{(indices modulo $3$)}
\end{split}
\]
all the other products among basic elements being $0$. Note that $e_1$ and $e_2$ are orthogonal idempotents with $e_1+e_2=1$. This algebra is endowed too with a nondegenerate quadratic form $q:C\rightarrow \bF$ which is multiplicative ($q(xy)=q(x)q(y)$ for any $x,y\in C$) and whose associated polar form: $q(x,y)=q(x+y)-q(x)-q(y)$, is given by
\[
q(e_1,e_2)=1=q(u_i,v_i)\quad (i=1,2,3)
\]
and $q(a,b)=0$ for any other basic elements. This allows to identify the Lie algebra $\frso_8(\bF)$ to
\[
\frso(C,q)=\{ f\in\End_\bF(C): q\bigl(f(x),y\bigr)+q\bigl(x,f(y)\bigr)=0\ \forall x,y\in C\}.
\]
The diagonal linear maps in $\frso(C,q)$ relative to our basis $\calB$ constitute a Cartan subalgebra $\frh$ of $\frso(C,q)$ with the basis $\{h_0,h_1,h_2,h_3\}$ defined by:
\begin{equation}\label{eq:basish}
\begin{split}
h_0&=\diag(1,-1,0,0,0,0,0,0),\\
h_1&=\diag(0,0,1,0,0,-1,0,0),\\
h_2&=\diag(0,0,0,1,0,0,-1,0),\\
h_3&=\diag(0,0,0,0,1,0,0,-1).
\end{split}
\end{equation}
(We identify linear maps and their matrices relative to our basis $\calB$.)

The weights of $\frh$ in $C$ are $\pm \epsilon_0$, $\pm\epsilon_1$, $\pm\epsilon_2$ and $\pm\epsilon_3$, with $\epsilon_i(h_j)=0$ if $i\ne j$ and $\epsilon_i(h_i)=1$ ($i,j=1,2,3,4$).

Consider the invariant bilinear form on $\frso(C,q)$ given by $b\bigl(d_1,d_2)=\frac{1}{2}\trace(d_1d_2)$. Then $\{h_0,h_1,h_2,h_3\}$ is an orthonormal basis of $\frh$ and $\{\epsilon_0,\epsilon_1,\epsilon_2,\epsilon_3\}$ is the dual basis in $\frh^*$. Denote by $(.\vert .)$ the nondegenerate symmetric bilinear form induced by $b$ on $\frh^*$, so that $(\epsilon_i\vert\epsilon_j)=1$ for $i=j$ and $0$ otherwise. If we identify, as it is usual, $\frh$ and $\frh^*$ by means of $b$, then $\epsilon_i$ is identified to $h_i$ for any $i=0,1,2,3$.

The algebra $C$ is endowed with a natural involution $\sigma:C\rightarrow C$, $x\mapsto \bar x=q(x,1)1-x$, such that $x\bar x=\bar xx=q(x)1$ for any $x\in C$. Thus $\overline{xy}=\bar y\bar x$ for any $x,y\in C$. Note that $\sigma$ belongs to the orthogonal group $O(C,q)$. Besides
\[
q(xy,z)=q(y,\bar xz)=q(x,z\bar y)
\]
for any $x,y,z\in C$.

Given any automorphism $\eta\in\Aut C$ with $\eta^3=id$, consider the new multiplication on $C$ defined by
\begin{equation}\label{eq:Ceta}
x*y=\eta(\bar x)\eta^2(\bar y),
\end{equation}
and denote by $\bar C_\eta$ the algebra $(C,*)$ thus defined. Note that for any $x,y,z\in C$:
\begin{equation}\label{eq:qassoc}
\begin{split}
q(x*y,z)&=q\bigl(\eta(\bar x)\eta^2(\bar y),z\bigr)\\
  &=q\bigl(\eta(\bar x),z\eta^2(y)\bigr)\quad \text{($\eta$ commutes with $\sigma$)}\\
  &=q\bigl(\bar x,\eta^2(z)\eta(y)\bigr)\quad \text{($\Aut C\subseteq O(C,q)$)}\\
  &=q\bigl(x,\overline{\eta^2(z)\eta(y)}\bigr)\\
  &=q\bigl(x,\eta(\bar y)\eta^2(\bar z)\bigr)\\
  &=q\bigl(x,y*z\bigr).
\end{split}
\end{equation}

For $\eta=id$, the algebra $\bar C_{id}$ is the so called \emph{para-octonion algebra} (with product $x\bullet y=\bar x\bar y$). It turns out that the algebra $\bar C_\eta$ is either isomorphic to the para-octonion algebra or the so called \emph{pseudo-octonion algebra} defined by Okubo \cite{Oku78} (see \cite{EP96} or \cite{KMRT98}), which can be defined as the algebra $\bar C_\tau$, where $\tau\in\Aut C$ is defined by
\begin{equation}\label{eq:tau}
\begin{split}
&\tau(e_j)=e_j,\ (j=1,2)\\
&\tau(u_i)=u_{i+1},\ \tau(v_i)=v_{i+1},\ (i=1,2,3\ \text{modulo $3$}).
\end{split}
\end{equation}

Let us consider the special orthogonal group $SO(C,q)$ and define:
\[
\begin{split}
\Tri(C,\eta,q)&=\{(f_0,f_1,f_2)\in SO(C,q)^3: f_0(x*y)=f_1(x)*f_2(y)\ \forall x,y\in C\},\\[2pt]
\tri(C,\eta,q)&=\{ (d_0,d_1,d_2)\in\frso(C,q)^3: d_0(x*y)=d_1(x)*y+x*d_2(y)\ \forall x,y\in C\},
\end{split}
\]
where the product $*$ is the multiplication in $\bar C_\eta$ (equation \eqref{eq:Ceta}). Because of equation \eqref{eq:qassoc}, we may define alternatively the trilinear map $Q_\eta(x,y,z)=q(x*y,z)=q(y*z,x)$ and then:
\[
\begin{split}
\Tri(C,\eta,q)&=\{ (f_0,f_1,f_2)\in SO(C,q)^3:
    Q_\eta\bigl(f_1(x),f_2(y)y,f_0(z)\bigr)\\
    &\null\hspace{2in} =Q_\eta(x,y,z)\ \forall x,y,z\in C\},\\[3pt]
\tri(C,\eta,q)&=\{ (d_0,d_1,d_2)\in\frso(C,q)^3:
    Q_\eta\bigl(d_1(x),y,z\bigr)+Q_\eta\bigl(x,d_2(y),z\bigr)\\
    &\null\hspace{2in} +Q_\eta\bigl(x,y,d_0(z)\bigr)=0\ \forall x,y,z\in C\}.
\end{split}
\]

Then $\Tri(C,\eta,q)$ is a subgroup of $SO(C,q)^3$ (componentwise multiplication) and $\tri(C,\eta,q)$ is a Lie subalgebra of $\frso(C,q)^3$. Also, as $Q(x,y,z)=Q(y,z,x)$ for any $x,y,z\in C$ it follows that a triple $(f_0,f_1,f_2)$ lies in $\Tri(C,\eta,q)$  if and only if so does $(f_2,f_0,f_1)$, and a triple $(d_0,d_1,d_2)$ lies in $\tri(C,\eta,q)$ if and only if so does $(d_2,d_0,d_1)$. Let us denote by $\theta_\eta$ both the group automorphism $\Tri(C,\eta,q)\rightarrow \Tri(C,\eta,q): (f_0,f_1,f_2)\mapsto (f_2,f_0,f_1)$ and the Lie algebra automorphism $\tri(C,\eta,q)\rightarrow \tri(C,\eta,q): (d_0,d_1,d_2)\mapsto (d_2,d_0,d_1)$. The \emph{Principle of Triality} (see \cite{KMRT98} or \cite{Eld00}) can be recast as follows:

\begin{proposition} Under the above conditions:
\begin{romanenumerate}
\item The projection $\pi_0:\Tri(C,\eta,q)\rightarrow SO(C,q);\ (f_0,f_1,f_2)\mapsto f_0$ is a surjective group homomorphism with kernel $\{(id,id,id), (id,-id,-id)\}$, and the projection $\pi_0:\tri(C,\eta,q)\rightarrow \tri(C,\eta,q);\ (d_0,d_1,d_2)\mapsto d_0$ is a Lie algebra isomorphism.
\item The triality automorphism $\theta_\eta$ induces a group automorphism of the projective special orthogonal group (again denoted by $\theta_\eta$) $PSO(C,q)\rightarrow PSO(C,q): [f_0]\mapsto [f_2]$ (where $(f_0,f_1,f_2)\in \Tri(C,\eta,q)$ and $[f]$ denotes the class of the element $f\in SO(C,q)$ in the quotient $PSO(C,q)=SO(C,q)/\{\pm id\}$), and a Lie algebra automorphism $\frso(C,q)\rightarrow \frso(C,q): d_0\mapsto d_2$ (with $(d_0,d_1,d_2)\in \tri(C,\eta,q)$). \qed
\end{romanenumerate}
\end{proposition}

\smallskip

Denote by $\theta$ the triality automorphism $\theta_{id}$ corresponding to the para-octonion algebra. Also, for any $f$ in the orthogonal group $O(C,q)$, denote by $\iota_f$ the automorphism of $\frso(C,q)$ given by $\iota_f(d)=fdf^{-1}$. Recall the natural involution $\sigma: x\mapsto \bar x$ of $C$:

\begin{proposition}\label{pr:autsoCq}
Under the above conditions:
\begin{enumerate}
\item The automorphisms $\iota_\sigma$ and $\theta$ of $\frso(C,q)$ generate a subgroup isomorphic to the symmetric group $S_3$ on three symbols.
\item $\Aut\frso(C,q)$ is the semidirect product of $\iota_{SO(C,q)}$ (which is isomorphic to $PSO(C,q)$) and the subgroup isomorphic to $S_3$ in item \textup{(1)}.
\item Given an automorphism $\eta\in\Aut C$ with $\eta^3=id$, the triality automorphism $\theta_\eta$ coincides with $\iota_\eta\theta=\theta\iota_\eta$.
\item The centralizer in $\Aut\frso(C,q)$ of the triality automorphism $\theta_\eta$ is the direct product of $\iota_{\Aut(C,*)}$ ($*$ as in \eqref{eq:Ceta}) and the subgroup generated by $\theta_\eta$ (which is a cyclic group of order $3$).
\end{enumerate}
\end{proposition}

\begin{proof}
Part (1) is well-known. For $(d_0,d_1,d_2)\in\tri(C,id,q)$, $d_0(x\bullet y)=d_1(x)\bullet y+x\bullet d_2(y)$ ($x\bullet y=\bar x\bar y$) for any $x,y\in C$, and hence for any $x,y\in C$, since $\sigma: x\mapsto \bar x$ is an involution we get:
\[
\begin{split}
(\sigma d_0\sigma)(x\bullet y)&=\sigma d_0\bigl(\sigma(y)\bullet\sigma(x)\bigr)\\
  &=\sigma\bigl(d_1(\sigma(y))\bullet \sigma(x)+\sigma(y)\bullet d_2(\sigma(x))\bigr)\\
  &=x\bullet (\sigma d_1\sigma)(y)+(\sigma d_2\sigma)(x)\bullet y.
\end{split}
\]
Hence $\theta\iota_\sigma(d_0)=\sigma d_1\sigma=\iota_\sigma\theta^2(d_0)$. Thus $\iota_\sigma\theta=\theta^2\iota_\sigma$ and the group generated by $\iota_\sigma$ and $\theta$ is isomorphic to $S_3$.

The result in (2) appears in \cite[Chapter IX]{Jacobson}. Note that the group $\iota_{SO(C,q)}$ is the group of inner automorphisms of $\frso(C,q)$.

As for (3), this is proved in \cite{EldHamburgo}. We include a proof for completeness. For $d_0\in\frso(C,q)$,
\[
\begin{split}
d_0(x*y)&=d_0\bigl(\eta(x)\bullet\eta^2(y)\bigr)=\theta^2(d_0)\bigl(\eta(x)\bigr)\bullet \eta^2(y) +\eta(x)\bullet\theta(d_0)\bigl(\eta^2(y)\bigr)\\
 &=\eta\bigl(\eta^2\theta^2(d_0)\eta^{-2}(x)\bigr)\bullet \eta^2(y)+
  \eta(x)\bullet\eta^2\bigl(\eta\theta(d_0)\eta^{-1}(y)\bigr)\\
  &=\bigl(\eta^2\theta^2(d_0)\eta^{-2}(x)\bigr)*y+x*\bigl(\eta\theta(d_0)\eta^{-1}(y)\bigr),
\end{split}
\]
so that $\theta_\eta(d_0)=\eta\theta(d_0)\eta^{-1}=\iota_\eta\theta(d_0)$. Also,
\[
\begin{split}
\eta d_0\eta^{-1}(x\bullet y)&=\eta d_0\bigl(\eta^{-1}(x)\bullet\eta^{-2}(y)\bigr)\\
    &=\eta\bigl(\theta^2(d_0)(\eta^{1}(x))\bullet\eta^{-1}(y)\bigr) +
    \eta\bigl(\eta^{-1}(x)\bullet\theta(d_0)(\eta^{-1}(y))\bigr)\\
    &=\eta\theta^2(d_0)\eta^{-1}(x)\bullet y +x\bullet \eta\theta(d_0)\eta^{-1}(y),
\end{split}
\]
so that $\theta\iota_\eta=\iota_\eta\theta$ and (3) follows.

Finally, consider an automorphism $\varphi\in\Aut\frso(C,q)$ such that $\varphi\theta_\eta=\theta_\eta\varphi$. Because of (3) and since $\Aut(C,*)$ is contained in $SO(C,q)$, the projection of $\theta_\eta$ onto the subgroup $S_3$ of item (1) is $\theta$ and hence $\varphi=\psi\theta_\eta^j$ for $j=0,1,2$ and $\psi\in\iota_{SO(C,q)}$: $\psi=\iota_{f_0}$, $f_0\in SO(C,q)$.
Take elements $f_1,f_2\in SO(C,q)$ such that $(f_0,f_1,f_2)\in\Tri(C,\eta,q)$. Then for any $d_0\in \frso(C,q)$ and $x,y\in C$,
\[
\begin{split}
f_0d_0f_0^{-1}(x*y)&=f_0d_0\bigl(f_1^{-1}(x)*f_2^{-1}(y)\bigr)\\
    &=f_0\bigl(\theta_\eta^2(d_0)(f_1^{-1}(x))*f_2^{-1}(y)\bigr) +
       f_0\bigl(f_1^{-1}(x)*\theta_\eta(d_0)(f_2^{-1}(y))\bigr)\\
    &=f_1\theta_\eta^2(d_0)f_1^{-1}(x)*y+x*f_2\theta_\eta(d_0)f_2^{-1}(y).
\end{split}
\]
Hence $\theta_\eta\iota_{f_0}=\iota_{f_2}\theta_\eta$ (see also \cite{EldHamburgo}). But $\iota_{f_0}=\psi$ and $\psi\theta_\eta=\theta_\eta\psi$. we conclude that $\iota_{f_0}=\iota_{f_2}$ or $f_2=\pm f_0$ and also $f_1=\pm f_0$.
Hence there exists $\epsilon\in\{\pm 1\}$ such that $f_0(x*y)=\epsilon f_0(x)*f_0(y)$ for any $x,y\in C$, and hence $\epsilon f_0\in\Aut(C,*)$ and $\psi=\iota_{f_0}=\iota_{-f_0}\in \iota_{\Aut(C,*)}$.
\end{proof}

\smallskip

Let us go back to our basis $\calB$ of $C$ in \eqref{eq:basisC} and to the natural Cartan subalgebra $\frh$ of $\frso(C,q)$ in \eqref{eq:basish}. Let $U=\bF u_1+\bF u_2+\bF u_3$ and $V=\bF v_1+\bF v_2+\bF v_3$. Given any bijective linear map $f_U\in GL(U)$ with $\det f_U=1$ and its ``dual'' map $f_V\in GL(V)$ (that is, $q\bigl(f_U(u),f_V(v)\bigr)=q(u,v)$ for any $u\in U$ and $v\in V$) it is clear that the linear map on $C$ defined by:
\begin{equation}\label{eq:fUfV}
f(e_j)=e_j,\ j=1,2,\quad f(u)=f_U(u)\ \forall u\in U,\quad f(v)=f_V(v)\ \forall v\in V,
\end{equation}
is an automorphism of $C$. Also, given any trace zero linear map $d_U\in \frsl(U)$, consider the trace zero linear map $d_V$ of $V$ given by $q\bigl(d_U(u),v\bigr)+q\bigl(u,d_V(v)\bigr)=0$ for any $u\in U$ and $v\in V$.
Then the linear map $d$ of $C$ defined by:
\[
d(e_j)=0,\ j=1,2,\quad d(u)=d_U(u)\ \forall u\in U,\quad d(v)=d_V(v)\ \forall v\in V,
\]
is a derivation of $C$. By its own definition, if $d\in \der C$, then $(d,d,d)\in\tri(C,id,q)$ so that $\theta(d)=d$. In particular,
\[
\theta(h_1-h_2)=h_1-h_2,\quad \theta(h_2-h_3)=h_2-h_3.
\]
Now, using that $d(x\bullet y)=\theta^2(d)(x)\bullet y+x\bullet\theta(d)(y)$ for any $d\in\frso(C,q)$ and $x,y\in C$, it follows that
\[
\theta(h_0)=-\frac{1}{2}(h_0+h_1+h_2+h_3),\quad \theta(h_1+h_2+h_3)=\frac{3}{2}h_0-\frac{1}{2}(h_1+h_2+h_3).
\]
Therefore:
\begin{equation}\label{eq:theta}
\begin{split}
\theta(h_0)&=\frac{1}{2}(-h_0-h_1-h_2-h_3),\\
\theta(h_1)&=\frac{1}{2}(h_0+h_1-h_2-h_3),\\
\theta(h_2)&=\frac{1}{2}(h_0-h_1+h_2-h_3),\\
\theta(h_3)&=\frac{1}{2}(h_0-h_1-h_2+h_3).
\end{split}
\end{equation}
In particular, the automorphism $\theta$ leaves invariant the Cartan subalgebra $\frh$ and hence induces the corresponding linear automorphism of $\frh^*$ (also denoted by $\theta$; recall that $\frh$ and $\frh^*$ are identified through the bilinear form given by one half the trace form, so that $h_i\leftrightarrow\epsilon_i$ for any $i=0,1,2,3$. Hence we just change the $h_i$'s by the $\epsilon_i$'s in the formula above).

On the other hand, the automorphism $\tau$ of $C$ in \eqref{eq:tau} induces the automorphism $\iota_\tau\in\Aut\frso(C,q)$ which leaves $\frh$ invariant too and hence induces a linear automorphism of $\frh^*$ (which again will be denoted too by $\iota_\tau$):
\begin{equation}\label{eq:iotatau}
\iota_\tau: \epsilon_0\mapsto\epsilon_0,\quad \epsilon_1\mapsto\epsilon_2\mapsto\epsilon_3\mapsto\epsilon_1.
\end{equation}
Let $\Phi$ denote, as usual, the set of roots of $\frh$ in $\frso(C,q)$. Then
\[
\Phi=\{\pm\epsilon_i\pm\epsilon_j: 0\leq i< j\leq 3\}.
\]
The root system $\Phi$ decomposes as the disjoint union $\Phi=\Phi_1\cup \Phi_2\cup \Phi_3$ with
\[
\Phi_i=\{ \pm\epsilon_0\pm\epsilon_i, \pm\epsilon_{i'}\pm\epsilon_{i''}\}
\quad (\{i,i',i''\}=\{1,2,3\}),
\]
and this is the only possible decomposition in three subsets of size $8$ such that for any two elements $\alpha,\beta$ in each subset, either $\beta\in\{\pm\alpha\}$ or $(\alpha\vert\beta)=0$. In other words, for any $i=1,2,3$ and any $\alpha\in\Phi_i$, $\Phi_i=\{\pm\alpha\}\cup\{\beta\in\Phi: (\alpha\vert\beta)=0\}$.

For any root $\alpha$ consider (see \cite[III.10]{Humphreys}) the corresponding reflection in $\frh^*$: $\sigma_\alpha(\beta)=\beta-(\beta\vert\alpha)\alpha$ (note that $(\alpha\vert\alpha)=2$ for any $\alpha\in\Phi$). These reflections generate the Weyl group $W$. Then for $i\ne j$, $\sigma_{\epsilon_i-\epsilon_j}$ permutes $\epsilon_i$ and $\epsilon_j$ and leaves fixed $\epsilon_h$ for $h\ne i,j$, while $\sigma_{\epsilon_i+\epsilon_j}$ takes $\epsilon_i\leftrightarrow -\epsilon_j$ and leaves fixed $\epsilon_h$ for $h\ne i,j$. In particular the automorphism $\iota_\tau$ in \eqref{eq:iotatau} equals $\sigma_{\epsilon_1-\epsilon_2}\sigma_{\epsilon_2-\epsilon_3}$, so it belongs to the Weyl group.

On the other hand, the automorphism $\theta$ satisfies:
\[
\epsilon_0-\epsilon_1\stackrel{\theta}{\mapsto}-\epsilon_0-\epsilon_1\stackrel{\theta}{\mapsto} \epsilon_2+\epsilon_3\stackrel{\theta}{\mapsto} \epsilon_0-\epsilon_1,\quad
\epsilon_1-\epsilon_3\stackrel{\theta}{\mapsto}\epsilon_1-\epsilon_3,
\]
and we get the system of simple roots $\Pi=\{\epsilon_0-\epsilon_1,\epsilon_1-\epsilon_3,-\epsilon_0-\epsilon_1,\epsilon_2+\epsilon_3\}$ which is fixed by $\theta$ ($\theta(\Pi)=\Pi$), but where $\theta$ permutes cyclically the `outer' roots (this shows that $\theta$ is an outer automorphism).

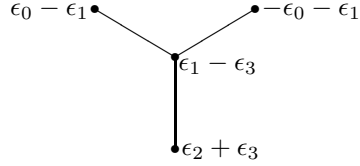
\[
\begin{gathered}
\begin{picture}(100,70)(-30,-50)
 \put(0,0){\circle*{3}}
 \put(0,-35){\circle*{3}}
 \put(-30,18){\circle*{3}}
 \put(30,18){\circle*{3}}
 \put(0,0){\line(5,3){30}}
 \put(0,0){\line(-5,3){30}}
 \put(0,0){\line(0,-1){35}}
 \put(-47,18){\makebox(0,0){$\epsilon_0-\epsilon_1$}}
 \put(51,18){\makebox(0,0){$-\epsilon_0-\epsilon_1$}}
 \put(16,-3){\makebox(0,0){$\epsilon_1-\epsilon_3$}}
 \put(17,-35){\makebox(0,0){$\epsilon_2+\epsilon_3$}}
\end{picture}
\end{gathered}
\]

The natural involution $\sigma$, which permutes $e_1$ and $e_2$ and sends $u_i$ to $-u_i$ and $v_i$ to $-v_i$ for any $i$, also induces an automorphism $\iota_\sigma\in\Aut\frso(C,q)$ which leaves invariant the Cartan subalgebra $\frh$ and induces the linear automorphism of $\frh^*$: $\epsilon_0\mapsto-\epsilon_0$, $\epsilon_i\mapsto\epsilon_i$ ($i=1,2,3$). Hence
\[
\epsilon_0-\epsilon_1\stackrel{\iota_\sigma}{\leftrightarrow}-\epsilon_0-\epsilon_1,\quad
\epsilon_1-\epsilon_3\stackrel{\iota_\sigma}{\mapsto}\epsilon_1-\epsilon_3,\quad
\epsilon_2+\epsilon_3\stackrel{\iota_\sigma}{\mapsto}\epsilon_2+\epsilon_3.
\]
Thus $\iota_\sigma$ and $\theta$ generate the group $\Gamma$ of diagram automorphisms of $\Phi$ relative to the system $\Pi$ of simple roots. Finally, the group of automorphisms of the root system: $\Aut\Phi$ (see \cite[\S 12.2]{Humphreys}), is the semidirect product of $W$ and $\Gamma$. The Weyl group is generated by the reflections $\sigma_{\epsilon_i-\epsilon_j}$, which induce transpositions of the elements in the basis $\{\epsilon_0,\epsilon_1,\epsilon_2,\epsilon_3\}$ of $\frh^*$ and hence generate a subgroup isomorphic to the symmetric group $S_4$, and by the elements $\sigma_{\epsilon_i+\epsilon_j}\sigma_{\epsilon_i-\epsilon_j}$ which changes signs to $\epsilon_i$ and $\epsilon_j$ and leaves fixed $\epsilon_h$ for $i\ne j$. Then $W$ is the semidirect product
\begin{equation}\label{eq:WZS4}
W=Z \rtimes S_4,
\end{equation}
where $S_4$ denotes the group of permutations of $\{\epsilon_0,\epsilon_1,\epsilon_2,\epsilon_3\}$ and $Z$ denotes the group generated by the linear maps $\nu:\frh^*\rightarrow \frh^*$ such that $\nu(\epsilon_i)=\pm\epsilon_i$, $i=0,1,2,3$ with an even number of minus signs. Hence $Z$ is isomorphic to $\{(a_0,a_1,a_2,a_3)\in\bZ_2^4: a_0+a_1+a_2+a_3=0\}$ which is isomorphic to $\bZ_2^3$. In particular, as it is well known, $\vert W\vert=8\times 4!=2^6\times 3$, and $\vert\Aut\Phi\vert=\vert W\rtimes \Gamma\vert=2^7\times 3^2$.

The elements $\iota_\tau$ and $\theta$ commute, so they generate a Sylow $3$-subgroup of $\Aut\Phi$ isomorphic to $\bZ_3\times\bZ_3$ (the direct product of two cyclic groups of order $3$).

\begin{lemma}\label{le:order3}
$\Aut\Phi$ does not contain elements of order $9$ and any element of order $3$ is conjugated to either $\iota_\tau$, $\theta$ or $\iota_\tau\theta$.
\end{lemma}
\begin{proof}
All the Sylow $3$-subgroups of $\Aut\Phi$ are conjugated and hence the first assertion follows and any element of order $3$ is conjugated to a nontrivial element in our distinguished Sylow $3$-subgroup: $\langle\iota_\tau\rangle\times\langle\theta\rangle$. Now, this is contained in the larger group
\begin{equation}\label{eq:Se1e2e3S3}
S_{\{\epsilon_1,\epsilon_2,\epsilon_3\}}\times S_3,
\end{equation}
where $S_{\{\epsilon_1,\epsilon_2,\epsilon_3\}}$ denotes the permutation group of the set $\{\epsilon_1,\epsilon_2,\epsilon_3\}$ and $S_3$ is the subgroup generated by $\iota_\sigma$ and $\theta$. (These two subgroups commute elementwise.) Note that $\iota_\tau$ is conjugated to $\iota_\tau^2$ in $S_{\{\epsilon_1,\epsilon_2,\epsilon_3\}}$, while $\theta$ and $\theta^2$ are conjugated in $S_3$. The Lemma follows.
\end{proof}

\smallskip

Consider the subgroups:
\[
\begin{split}
T&=\{\varphi\in\Aut\frso(C,q): \varphi\vert_\frh=id\},\\
N&=\{\varphi\in\Aut\frso(C,q): \varphi(\frh)\subseteq\frh\}.
\end{split}
\]
Any $\varphi\in T$ is an inner automorphism \cite[Proposition IX.3]{Jacobson}, so that there is an element $f\in SO(C,q)$ such that $\varphi=\iota_f$. Since $\varphi\vert_\frh=id$, $f$ must leave invariants the weight spaces of $C$ relative to the action of $\frh$ and hence there are nonzero scalars $\beta_0,\beta_1,\beta_2,\beta_3\in \bF^\times$ such that:
\begin{equation}\label{eq:T}
f=\diag(\beta_0,\beta_0^{-1},\beta_1,\beta_2,\beta_3,\beta_1^{-1},\beta_2^{-1},\beta_3^{-1}).
\end{equation}
Also, any $\varphi\in N$ induces an automorphism of $\Phi$. The corresponding homomorphism
\[
\pi: N\rightarrow \Aut\Phi
\]
is surjective with kernel $T$ \cite[Chapter IX, exercise 11]{Jacobson}. Since $T$ is commutative, given any $\nu\in\Aut\Phi$ and $\varphi\in N$ such that $\pi(\varphi)=\nu$, the centralizer $T^{(\varphi)}=\{\psi\in T: \psi\varphi=\varphi\psi\}$ depends only on $\nu$, and we may write $T^{(\nu)}$ to denote it. In particular, given any permutation $p$ of $\{ 1,2,3\}$, consider the automorphism $f_p$ of $C$ given by
\begin{equation}\label{eq:fp}
\begin{split}
f_p(e_j)&=e_j,\ j=1,2,\\
f_p(u_i)&=(-1)^pu_{p(i)},\ f_p(v_i)=(-1)^pv_{p(i)},\ i=1,2,3,
\end{split}
\end{equation}
where $(-1)^p$ denotes the signature of the permutation $p$. Equation \eqref{eq:fUfV} shows that $f_p$ is indeed an automorphism of $C$ such that the corresponding automorphism of $\frso(C,q)$: $\iota_{f_p}:d\mapsto f_pdf_p^{-1}$ is in $N$. Note that if $p$ is the cycle $(123)$, $f_p$ turns out to be our automorphism $\tau$ in \eqref{eq:tau}. The image under $\pi$ of the subgroup consisting of the $\iota_{f_p}$'s is precisely the group $S_{\{\epsilon_1,\epsilon_2,\epsilon_3\}}$ in \eqref{eq:Se1e2e3S3}. It makes sense then to consider $T^{((12))}$, the centralizer in $T$ of the transposition $(12)$ in $S_{\{\epsilon_1,\epsilon_2,\epsilon_3\}}$.

\begin{lemma}\label{le:centralizers}
\null\quad
\begin{enumerate}
\item The centralizer in $\Aut\Phi$ of $\theta$ is the cartesian product $S_{\{\epsilon_1,\epsilon_2,\epsilon_3\}}\times\{\pm id\}\times\langle\theta\rangle$.
\item $T^{(-id)}$ is an elementary $2$-group.
\item $T^{(\theta)}=T\cap\iota_{\Aut C}$ is isomorphic to the two dimensional torus $(\bF^\times)^2$.
\item $T^{(\theta)}\cap T^{((1,2))}$ is isomorphic to the one dimensional torus $\bF^\times$.
\item $T^{(\theta)}\cap T^{(-(1,2))}$ is isomorphic to the one dimensional torus $\bF^\times$.
\item $T^{(\iota_\tau\theta)}$ is finite.
\end{enumerate}
\end{lemma}
\begin{proof}
Since $\Aut\Phi=W\rtimes \Gamma$ and $\Centr_{\Gamma}(\theta)=\langle\theta\rangle$ ($\Centr_G$ denotes the centralizer in the group $G$), for (1) it is enough to compute the centralizer $\Centr_W(\theta)$ in the Weyl group $W$ of $\theta$. Note that $\theta$ fixes the elements $\epsilon_1-\epsilon_2$ and $\epsilon_2-\epsilon_3$ (see \eqref{eq:theta}) and the subspace $\bF \epsilon_0+\bF(\epsilon_1+\epsilon_2+\epsilon_3)$, and the minimal polynomial of the action on $\theta$ on this subspace is $X^2+X+1$. Hence any element $\gamma\in W$ which commutes with $\theta$ must leave invariant the subspaces $\bF(\epsilon_1-\epsilon_2)+\bF(\epsilon_2-\epsilon_3)$ and $\bF \epsilon_0+\bF(\epsilon_1+\epsilon_2+\epsilon_3)$. The interpretation of the elements of $W$ as permutations of $\{\epsilon_0,\epsilon_1,\epsilon_2,\epsilon_3\}$ followed by ``even sign changes'' (see \eqref{eq:WZS4}) proves (1).

The automorphism $-id$ of $\Aut\Phi$ is the image under $\pi$ of the conjugation by the order $2$ automorphism $\ex$ of $C$ given by
\[
\ex(e_1)=e_2,\ \ex(u_i)=v_i\ (i=1,2,3),\ \ex^2=id.
\]
With $f$ as in \eqref{eq:T}, $\iota_f\iota_{\ex}=\iota_{\ex}\iota_f$ if and only if $f\circ \ex=\pm\ex\circ f$, and this is equivalent either to the condition $\beta_i\in\{\pm 1\}$ for $i=0,1,2,3$, or to the condition $\beta_i\in \{\pm\xi\}$ for $i=0,1,2,3$, with $\xi^4=1\ne \xi^2$.
Therefore $T^{(-id)}$ is generated by the order $2$ elements $\iota_f$, with $f=\diag(\beta_0,\beta_0^{-1},\beta_1,\beta_2,\beta_3,
\beta_1^{-1},\beta_2^{-1},\beta_3^{-1})$ and $(\beta_0,\beta_1,\beta_2,\beta_3)$ equal to either $(1,-1,1,1)$, $(1,1,-1,1)$, $(1,1,1,-1)$ and $(\xi,\xi,\xi,\xi)$ (note that for $(\beta_0,\beta_1,\beta_2,\beta_3)=(-1,-1,-1,-1)$ the associated $\iota_f$ is the identity), and hence it is isomorphic to $\bZ_2^4$. This proves (2).

The fact that $T^{(\theta)}=T\cap\iota_{\Aut C}$ follows from part (4) of Proposition \ref{pr:autsoCq} (note that the automorphism group of the octonions and of the para-octonions coincide). Now, with $f$ as in \eqref{eq:T}, $f$ is an automorphism of $C$ if and only if $\beta_0=1$ (as $f(e_1)$ must be an idempotent too) and $\beta_1\beta_2\beta_3=1$. Therefore
\[
T\cap\iota_{\Aut C}=\{\iota_f: f=\diag(1,1,\alpha,\beta,(\alpha\beta)^{-1},\alpha^{-1},\beta^{-1},\alpha\beta),\ \alpha,\beta\in\bF^\times\}
\]
and (3) follows. Now, for item (4) we must check which elements in $T\cap\iota_{\Aut C}$ commute with the automorphism $\iota_{f_{(1,2)}}$ ($f_{(1,2)}$ as in \eqref{eq:fp}), and for item (5) with the automorphism $\iota_{\ex}\iota_{f_{(1,2)}}$, and it easily follows that
\[
T^{(\theta)}\cap T^{((1,2))}=\{\iota_f: f=\diag(1,1,\alpha,\alpha,\alpha^{-2},\alpha^{-1},\alpha^{-1},\alpha^2),\ \alpha\in\bF^\times\},
\]
and
\[
T^{(\theta)}\cap T^{(-(1,2))}=\{\iota_f: f=\diag(1,1,\alpha,\alpha^{-1},1,\alpha^{-1},\alpha,1),\ \alpha\in\bF^\times\}.
\]

Finally, $\iota_\tau\theta$ is the triality automorphism $\theta_\tau$ (Proposition \ref{pr:autsoCq}) and hence $T^{(\iota_\tau\theta)}=T\cap\iota_{\Aut\bar C_\tau}$. Let $f$ be as in \eqref{eq:T}, then $\iota_f\in T^{(\iota_\tau\theta)}$ if and only if (changing $f$ to $-f$ if necessary, see the proof of item (4) of Proposition \ref{pr:autsoCq}) $f\in \Aut\bar C_\tau$. In this case $f(e_1*e_1)=f(e_2)=\beta_0^{-1}e_2$, but $f(e_1)*f(e_1)=\beta_0^2e_2$, and hence $\beta_0^3=1$. Now,
\[
\begin{split}
&f(e_1*v_1)=f(e_2\bar v_3)=-f(v_3)=-\beta_3^{-1}v_3,\\
&f(e_1)*f(v_1)=\beta_0\beta_1^{-1}e_1*v_1=-\beta_0\beta_1^{-1}v_3,
\end{split}
\]
so $\beta_1=\beta_0\beta_3$ and $\beta_2=\beta_0\beta_1$, $\beta_3=\beta_0\beta_2$ too. Also,
\[
\begin{split}
&f(u_1*u_1)=f(u_2u_3)=f(v_1)=\beta_1^{-1}v_1,\\
&f(u_1)*f(u_1)=\beta_1^2u_1*u_1=\beta_1^2v_1,
\end{split}
\]
so that $\beta_1^3=1$, and also $\beta_2^3=1=\beta_3^3$. Hence $T^{(\iota_\tau\theta)}$ is finite.
\end{proof}

We will need the following general result. It can be deduced too, as in \cite{DMVpr}, from conjugacy results in algebraic groups.

\begin{lemma}\label{le:MinT}
Let $\frg=\oplus_{g\in G}\frg_g$ be a grading of a finite dimensional semisimple Lie algebra $\frg$. Then there is an homogeneous Cartan subalgebra $\frh$ of $\frg$ (that is, $\frh=\oplus_{g\in G}(\frh\cap\frg_g)$) such that $\frh\cap\frg_0$ is a Cartan subalgebra of $\frg_0$.
\end{lemma}
\begin{proof}
We may assume that $G$ is the group generated by the support of the grading and hence, in particular, that $G$ is finitely generated. The proof will be done by induction on the number of generators. Let $g\in G$ be a generator with $G=\langle g\rangle\times G'$ for a subgroup $G'$.

If the order of $g$ is infinite, then $\frg=\oplus_{i\in\bZ}\frg_{[i]}$, with $\frg_{[i]}=\sum_{g'\in G'}\frg_{ig+g'}$. This gives a $\bZ$-grading of $\frg$. The grading derivation $d\in\der\frg$ such that $d(x)=ix$ for any $x\in\frg_{[i]}$ is inner (as $\frg$ is semisimple) and hence $d=\ad_{z}$ for some $z\in\frg_{[0]}$. Since the Killing form of $\frg$ is nondegenerate, the subalgebra $\frg_{[0]}$ is reductive and any element in its center  $Z(\frg_{[0]})$ is ad-diagonalizable. Hence $\frg_{[0]}=Z(\frg_{[0]})\oplus \frs$ with $\frs=[\frg_{[0]},\frg_{[0]}]$. Note that $\frs$ is $G'$-graded. The induction hypothesis allows us to find a Cartan subalgebra $\frh'$ of $\frs$ which is homogeneous and such that $\frh'\cap\frs_0$ is a Cartan subalgebra of $\frs_0$. Then $\frh=Z(\frg_{[0]})\oplus \frh'$ is an homogeneous Cartan subalgebra of $\frg$, and $\frg_0=Z(\frg_{[0]})_0\oplus\frs_0$, so $\frh\cap\frg_0$ is a Cartan subalgebra of $\frg_0$.

Otherwise $g$ has finite order, say $m$, and hence $\frg=\oplus_{i=0}^{m-1}\frg_{[i]}$ with $\frg_{[i]}=\oplus_{g'\in G'}\frg_{ig+g'}$. Note that $\frg_{[0]}\ne 0$, otherwise the elements of $\frg_{[1]}$ would be ad-nilpotent and $[\frg_{[1]},\frg_{[m-1]}]=0$, which would imply that $\kappa(\frg_{[1]},\frg)=0$ ($\kappa$ denotes the Killing form) and $\frg_{[1]}=0$. This forces the elements of $\frg_{[2]}$ to be ad-nilpotent (for any $x\in\frg_{[2]}$ and $y\in\frg_{[i]}$, there is an $r$ such that $\ad_x^r(y)\in\frg_{[0]}+\frg_{[1]}=0$). As before, we would conclude that $\frg_{[2]}=0$, and so on (see \cite[Lemma 8.1]{Kac}), thus getting a contradiction. Again $\frg_{[0]}=Z(\frg_{[0]})\oplus \frs$ is a reductive subalgebra as before and there is a Cartan subalgebra $\frh'$ of $\frs$ such that $\frh'\cap\frs_0$ is a Cartan subalgebra of $\frs_0$. Consider $\frh_{[0]}=Z(\frg_{[0]})\oplus\frh'$. This is an homogeneous ad-diagonalizable subalgebra of $\frg$ and a Cartan subalgebra of $\frg_{[0]}$. Let $\frc$ be its centralizer in $\frg$, which is homogeneous too. On the other hand, $\frh_{[0]}$ is contained in a Cartan subalgebra $\frh$ of $\frg$, which is abelian ($\frg$ is semisimple) so $\frh\subseteq\frc$ and hence $\frh$ is a Cartan subalgebra of $\frc$. Thus there is the root decomposition of $\frc$: $\frc=\frh\oplus\bigl(\oplus_{\{\alpha\in\Phi:\alpha(\frh_{[0]})= 0\}}\frg_\alpha\bigr)$, where $\Phi$ denotes the root system of $\frh$. In particular the Killing form restricts to a nondegenerate form on $\frc$ and $\frc=\frh+[\frc,\frc]$. If different from $0$, the semisimple subalgebra $[\frc,\frc]$ is homogeneous and $[\frc,\frc]\cap\frg_{[0]}\subseteq \frc\cap\frg_{[0]}=\frh_{[0]}$. But we have $\frh_{[0]}\subseteq Z(\frc)$. Hence $[\frc,\frc]\cap\frg_{[0]}=0$ and we conclude as before that $[\frc,\frc]_{[1]}=0$ and, eventually, $[\frc,\frc]=0$. Therefore $\frh=\frc$ is a Cartan subalgebra of $\frg$ such that $\frh\cap\frg_0$ is a Cartan subalgebra of $\frg_0$.
\end{proof}

\smallskip

We are ready for a proof of the main result of this section (see \cite{DMVpr}). Let us denote by $R$ the associative algebra $\End_\bF(C)\simeq \Mat_8(\bF)$ and by $*$ the orthogonal involution associated to the norm $q$ of the octonion algebra $C$.

\begin{theorem}\label{th:R*outer}
Let $M$ be a MAD in $\Aut\frso(C,q)$. Then either it is conjugated to a MAD in $\Aut(R,*)$ or it contains an order $3$ outer automorphism.
\end{theorem}
\begin{proof}
By Lemma \ref{le:MinT}, there is Cartan subalgebra $\frh$ of $\frso(C,q)$ such that $M$ is contained in $\{\varphi\in\Aut\frso(C,q):\varphi(\frh)\subseteq \frh\}$. Since all the Cartan subalgebras are conjugated we may assume that $\frh$ is the Cartan subalgebra of diagonalizable endomorphisms in $\frso(C,q)$ relative to our basis $\calB$ of $C$ in \eqref{eq:basisC}.

Consider the projection $\pi:N\rightarrow \Aut\Phi=W\rtimes S_3$. If $\pi(M)$ were contained in $W\rtimes\langle\iota_\sigma\theta^i\rangle$ for some $i=0,1,2$, then as $\pi\bigl(\theta^{i}M\theta^{-i}\bigr)\subseteq W\rtimes\langle\iota_\sigma\rangle$, we have that $M$ is conjugated to $\theta^{i}M\theta^{-i}$, which is contained in $\iota_{SO(C,q)}\rtimes\langle\iota_\sigma\rangle=\iota_{O(C,q)}=\Aut(R,*)$. Therefore we may assume that $\pi(M)$ contains elements of order $3$, not all contained in $W$. Lemma \ref{le:order3} shows then that, after passing to a conjugate of $M$ if necessary, we may assume that $M$ contains and element $g$ such that either $\pi(g)=\theta$ or $\pi(g)=\iota_\tau\theta$.

In the latter case: $\pi(g)=\iota_\tau \theta$, we have $M\cap T\subseteq T^{(\iota_\tau\theta)}$, which is finite (Lemma \ref{le:centralizers}), but then, with $\varphi=\iota_\tau\theta$, the map
\[
\begin{split}
(\bF^\times)^4&\longrightarrow T^{(\varphi)}=\{t\in T: t\varphi=\varphi t\}\\
(\beta_0,\beta_1,\beta_2,\beta_3)&\mapsto t(\varphi t\varphi t^{-1})(\varphi^2 t\varphi^{-2}),
\end{split}
\]
where $t=\iota_f$ with $f=\diag(\beta_0,\beta_0^{-1},
\beta_1,\beta_2,\beta_3,\beta_1^{-1},\beta_2^{-1},\beta_3^{-1})$, is a continuous map (in the Zariski topology) of the torus $(\bF^\times)^4$ (connected) into the finite (and hence discrete) group $T^{(\varphi)}$. Hence this map is constant and thus $t(\varphi t\varphi t^{-1})(\varphi^2 t\varphi^{-2})=id$ for any such $t$. But $g=t\varphi$ for some $t\in T$ and hence
\[
g^3=t(\varphi t\varphi^{-1})(\varphi^2 t\varphi^{-2})\varphi^3=\varphi^3=id,
\]
and $M$ contains the order $3$ automorphism $g$, which is outer, as it induces the automorphism $\iota_\tau\theta$ in $\Aut\Phi$, which does not belong to $W$. (This argument is borrowed from \cite{DMVpr}.)

In the first case: $\pi(g)=\theta$, since $\pi(M)$ is an abelian subgroup of $\Aut(\Phi)$ containing $\theta$, it is contained in $S_{\{\epsilon_1,\epsilon_2,\epsilon_3\}}\times\{\pm id\}\times\langle\theta\rangle$ (Lemma \ref{le:centralizers}) and we may assume that its projection in $S_{\{\epsilon_1,\epsilon_2,\epsilon_3\}}$ does not contain elements of order $3$ (otherwise both $\iota_\tau$ and $\theta$ would belong to $\pi(M)$ and hence $\iota_\tau\theta$ would belong to $\pi(M)$, a case already dealt with). After a cyclic permutation of the indices $i=1,2,3$ we may assume that $\theta\in\pi(M)\subseteq \langle (1,2)\rangle\times\{\pm id\}\times\langle\theta\rangle$, and hence there are the following possibilities:
\[
\pi(M)=\langle\theta\rangle,\ \langle -id,\theta\rangle,\ \langle (1,2),\theta\rangle,\ \langle-(1,2),\theta\rangle\ \text{or}\ \langle (1,2),-id,\theta\rangle.
\]
But since $M$ is a maximal abelian diagonalizable subgroup of $\Aut\frso(C,q)$, $M\cap T=\cap_{\nu\in\pi(M)}T^{(\nu)}$. If $-id\in\pi(M)$ then $M\cap T\subseteq T^{(-id)}$ which is a $2$-group (Lemma \ref{le:centralizers}), so that $M$ is finite of size $\vert M\cap T\vert\vert\pi(M)\vert =2^m\times 3$ for some $m$, it follows that $M$ contains an order $3$ element, as required.
Otherwise $\pi(M)=\langle\theta\rangle$, $\langle (1,2),\theta\rangle$, or $\langle -(1,2),\theta\rangle$, and Lemma \ref{le:centralizers} shows that $M\cap T$ is a torus. But $M$  is the direct product  of the connected component of $1$ and a finite group: $M=M^0\times F$ (see
\cite[\S 16.2]{HumphreysGroups}), and $M\cap T$ is connected, so $M\cap T\subseteq M^0$. It follows that the order of $F$ divides the order of $\pi(M)=M/M\cap T$ which is $2^m\times 3$ for some $m$. Again, since $\pi(M^0)=\{id\}$ as $M^0$ is connected, this means that there is an order $3$ element in $F$ which projects onto $\theta$.
\end{proof}

\smallskip

The classification, up to conjugacy, of the MADs containing order $3$ outer automorphisms is now easy (see \cite{DMVpr}):

\begin{theorem}\label{th:fineouterD4}
Up to equivalence, there are exactly three fine gradings of the Lie algebra $\frso_8(\bF)$ whose associated MAD contains an outer automorphism of order $3$. Their universal grading groups and types are the following:
\begin{itemize}
\item A $\bZ^2\times\bZ_3$-grading of type $(26,1)$.
\item A $\bZ_2^3\times\bZ_3$-grading of type $(14,7)$.
\item A $\bZ_3^3$-grading of type $(24,2)$.
\end{itemize}
\end{theorem}
\begin{proof}
Let $M$ be a MAD containing an outer automorphism of order $3$. According to \cite[Chapter 8]{Kac}, there are, up to conjugacy, only two such automorphisms, namely $\theta$ with fixed subalgebra $\der C$, or $\iota_\tau\theta=\theta_\tau$ with fixed subalgebra $\der\bar C_\tau$ (the Lie algebra of derivations of the pseudo-octonion algebra). Proposition \ref{pr:autsoCq} shows that in the first case $(\theta\in M$), $M$ is contained in the centralizer of $\theta$ in $\Aut\frso(C,q)$, which is $\iota_{\Aut C}\times \langle\theta\rangle$, and in the second case ($\theta_\tau\in M$), $M\subseteq \iota_{\Aut\bar C_\tau}\times\langle \theta_\tau\rangle$.

Hence, by maximality, $M=M'\times\langle\theta\rangle$ or $M=M''\times\langle\theta_\tau\rangle$ for a MAD $M'$ in $\Aut C$ or a MAD $M''$ in $\Aut \bar C_\tau$. There are, up to equivalence, just two fine gradings on $C$, with universal grading groups $\bZ_2^3$ and $\bZ^2$ (see \cite{Eld98} or \cite{Eldprb}). The first one gives the $\bZ_2^3\times\bZ_3$-grading of type $(14,7)$ which has been considered in \cite[Proposition 5.9]{Eldprb}, while the MAD of the second case is precisely $T^{(\theta)}\times\langle\theta\rangle$ in the notation previously used in this section. As the eigenspaces for the action of $\theta$ on $\frso(C,q)$ are $\der C$ and two copies of the natural module $C_0=\{x\in C:q(x,1)=0\}$ for $\der C$, the action of $T^{(\theta)}$ on each such eigenspace is the weight space decomposition for the Cartan subalgebra $\frh\cap \der C$. It follows that its type is $(26,1)$. Also, $T^{(\theta)}$ contains the element $\iota_{\tau_\omega}$ with $\tau_\omega=\diag(1,1,1,\omega,\omega^2,1,\omega^2,\omega)$, for a primitive cubic root $\omega$ of $1$ (see \cite[(3.24)]{Eldprb}). The element $\iota_{\tau_\omega}\theta$ is an outer automorphism of order $3$ in $\Aut\frso(C,q)$ with fixed subalgebra $\der \bar C_{\iota_\omega}$. But $\bar C_{\iota_\omega}$ is (isomorphic) to the pseudo-octonion algebra. Hence this MAD may be described too as $T^{(\iota_\omega\theta)}\times\langle \iota_\omega\theta\rangle=T^{(\theta_{\tau_\omega})}\times\langle\theta_{\tau_\omega}\rangle$, and contains an outer automorphism of the second type.

Finally, if $M$ is a MAD of the form $M''\times\langle\theta_\tau\rangle$, $M''$ gives a fine grading of the pseudo-octonion algebra and there are two possibilities \cite[Remark 4.6]{Eldprb}: a $\bZ_3^2$-grading or a $\bZ^2$-grading. The latter one gives a $\bZ^2\times\bZ_3$-grading of $\frso(C,q)$ already considered (it is conjugated to $T^{(\theta_{\tau_\omega})}\times\langle\theta_{\tau_\omega}\rangle$ above), while the first one gives the fine $\bZ_3^3$-grading of type $(24,2)$ in \cite[Proposition 5.10]{Eldprb}. Alternatively, the pseudo-octonion algebra can be defined (as was done originally in \cite{Oku78}) on the space of trace zero $3\times 3$ matrices and, as such, $\Aut\bar C_\tau\simeq\Aut\Mat_3(\bF)$. Theorem \ref{th:fine_gradings_R} gives the two possible gradings over $\bZ_3^2$ and $\bZ^2$ (see \cite[Remark 4.7]{Eldprb}).
\end{proof}

In order to finish the classification of the fine gradings of $\frso_8(\bF)$ up to equivalence, note that if $M$ is a MAD in $\Aut\frso_8(\bF)$, Theorem \ref{th:R*outer} shows that either the associated grading is one of the $15$ gradings in Example \ref{ex:R*_D4}, or one of the three fine gradings in Theorem \ref{th:fineouterD4}. The $15$-gradings are indeed fine, as otherwise, they could be extended with the use of some order $3$ outer automorphism and this is not possible according to the possibilities in Theorem \ref{th:fineouterD4}. Among the $15$ gradings in Example \ref{ex:R*_D4} there are three of them which share the same universal group:

\medskip
\noindent $\boxed{\Gamma_0}$:\quad The $\bZ_2^3\times\bZ_4$-grading of $\frso_8(\bF)\simeq K(R,*)$ with invariant $[(Q,\tau_o),(1,1,1,q_1)]$ in the notations of Section \ref{se:finesosp}. This corresponds to the identification of $R=\Mat_8(\bF)$ with $\End_Q(V)$, where $V$ is a graded $Q$-module of dimension $4$ with a good basis $\{v_1,v_2,v_3,v_4\}$ and a hermitian form relative to $\tau_o$, $B:V\times V\rightarrow Q$, such that the basis above is orthogonal and $B(v_i,v_i)=1$ for $i=1,2,3$ and $B(v_4,v_4)=q_1$. The involution $*$ is given by $B(rv,w)=B(v,r^*w)$ for any $r\in $ and $v,w\in V$. Besides $\degree(v_i)=g_i$ for any $i$, with $g_1=0$, and $2g_2=2g_3=2g_4+h_1=0$ (see \eqref{eq:relationstilde}).

Note that $(E_{ii}\otimes q)^*=E_{ii}\otimes q^{\tau_o}$ for $i=1,2,3$, while $(E_{44}\otimes q)^*=E_{44}\otimes q_1q^{\tau_o}q_1$. Then with $K=K(R,*)$, $R_0=\sum_{i=1}^4 E_{ii}\otimes\bF$ has dimension $4$ and $K_0=0$, and for $h\in H=\Supp Q$ (recall $h_i=\degree (q_i)$, $i=1,2,3$), $R_h=\sum_{i=1}^4 E_{ii}\otimes Q_h$ and hence $K_{h_1}=0$, $K_{h_2}=E_{44}\otimes \bF q_2$ and $K_{h_3}=\sum_{i=1}^3E_{ii}\otimes \bF q_3$. All the other homogeneous spaces in $R$ have dimension $2$, and their intersection with $K$ have dimension $1$, so that the type of this grading in $K$ is $(25,0,1)$.

\medskip

\noindent $\boxed{\Gamma_1}$:\quad The $\bZ_2^3\times\bZ_4$-grading of $\frso_8(\bF)\simeq K(R,*)$ with invariant $[(Q,\tau_o),(1,1,q_1,q_1)]$. This corresponds to the identification of $R=\Mat_8(\bF)$ with $\End_Q(V)$, where $V$ is a graded $Q$-module of dimension $4$ with a good basis $\{v_1,v_2,v_3,v_4\}$ and a hermitian form relative to $\tau_o$, $B:V\times V\rightarrow Q$, such that the basis above is orthogonal and $B(v_i,v_i)=1$ for $i=1,2$ and $B(v_i,v_i)=q_1$ for $i=3,4$. The same arguments as above gives that $\dim K_{h_2}=\dim K_{h_3}=2$ and all the other nonzero homogeneous spaces have dimension $1$. Hence the type of the grading is $(24,2)$. Here $\degree(v_i)=g_i$ with $g_1=0$ and $2g_2=2g_3+h_1=2g_4+h_1=0$, and the grading group is generated by the elements $g_3,g_2,g_3-g_4,h_2$ whose orders are $4$, $2$, $2$, $2$ respectively. The subalgebra $K\subo$ of $K=K(R,*)$ spanned by the homogeneous subspaces $K_g$ with $2g=0$ is spanned by the homogeneous subspaces:
\begin{itemize}
    \item $K_{h_3}=(\bF E_{11}+\bF E_{22})\otimes q_3$, $K_{h_2}=(\bF E_{33}+\bF E_{44})\otimes q_2$,
    \item $K_{g_2}=\bF(E_{12}-E_{21})\otimes 1$, $K_{g_2+h_1}=\bF(E_{12}-E_{21})\otimes q_1$, $K_{g_2+h_2}=\bF(E_{12}-E_{21})\otimes q_2$, $K_{g_2+h_3}=\bF(E_{12}+E_{21})\otimes q_3$,
    \item $K_{g_3-g_4}=\bF(E_{34}-E_{43})\otimes 1$, $K_{g_3-g_4+h_1}=\bF(E_{34}-E_{43})\otimes q_1$, $K_{g_3-g_4+h_2}=\bF(E_{34}+E_{43})\otimes q_2$, $K_{g_3-g_4+h_3}=\bF(E_{34}-E_{43})\otimes q_3$.
\end{itemize}
This subalgebra $K\subo$ is the direct sum of four ideals, all of them isomorphic, as ungraded algebras, to $\frsl_2(\bF)$, namely:
    \[
    \begin{split}
    \frs_1&=\espan{(E_{11}+E_{22})\otimes q_3,(E_{12}-E_{21})\otimes q_1,(E_{12}-E_{21})\otimes q_2},\\
    \frs_2&=\espan{(E_{11}-E_{22})\otimes q_3,(E_{12}-E_{21})\otimes 1,(E_{12}+E_{21})\otimes q_3},\\
    \frs_3&=\espan{(E_{33}+E_{44})\otimes q_2,(E_{34}-E_{43})\otimes q_1,(E_{34}-E_{43})\otimes q_3},\\
    \frs_4&=\espan{(E_{33}-E_{44})\otimes q_2,(E_{34}-E_{43})\otimes 1,(E_{34}-E_{43})\otimes q_2}.
    \end{split}
    \]
Note that $K_{h_2}\oplus K_{h_3}$ is a Cartan subalgebra of $\frg=\frso_8(\bF)=K(R,*)$, and that the basic elements
    \begin{equation}\label{eq:basisKs}
    \begin{split}
    &(E_{11}+E_{22})\otimes q_3\ \text{and}\ (E_{11}-E_{22})\otimes q_3\ \text{of $K_{h_3}$}\\
    &(E_{33}+E_{44})\otimes q_2\ \text{and}\ (E_{33}-E_{44})\otimes q_2\ \text{of $K_{h_2}$}
    \end{split}
    \end{equation}
are the only elements of $K_{h_3}$ and $K_{h_2}$, up to multiplication by scalars,  which are obtained as a product of homogeneous elements of $K\subo$.

Let us compute the associated MAD subgroup $M_1=\Diag_{\Gamma_1}(R)$. With the arguments in the proof of Theorem \ref{th:equivalence_varphi_gradings} any $\psi\in M_1$ is given by conjugation by an element $a$ of the form $a=(E_{11}+\alpha E_{22}+\beta E_{33}+\gamma E_{44})\otimes d$, for an homogeneous element $d\in Q$. The fact that $\psi$ commutes with $*$ is equivalent to the condition $aa^*=1$ (after scaling $a$ if necessary), that is:
    \[
    dd^{\tau_o}=1,\ \alpha^2dd^{\tau_o}=1,\ \beta^2dq_1d^{\tau_o}q_1=1,\ \gamma^2dq_1d^{\tau_o}q_1=1,
    \]
so that we may take (as we may scale the element $a$) $d$ to be $q_i$ for $i=0,1,2,3$ (with $q_0=1$) and then $\alpha\in\{\pm 1\}$ and either $\beta,\gamma\in \{\pm 1\}$ if $d=1$ or $d=q_1$, or $\beta,\gamma\in\{\pm\epsilon\}$ if $d=q_2$ or $d=q_3$ where $\epsilon$ is a primitive fourth root of $1$.

Note that if $a\in\Mat_r(\bF)$, $b\in\Mat_s(\bF)$ and we consider the Kronecker product $\Mat_{rs}(\bF)=\Mat_r(\bF)\otimes\Mat_s(\bF)$, then $\det(a\otimes b)=\det(a)^s\det(b)^r$. (It is enough, by Zariski density, to consider $a$ and $b$ diagonal matrices and the result is clear.)

Hence, with $a$ as above, $a\in \Mat_4(\bF)\otimes Q\simeq\Mat_4(\bF)\otimes\Mat_2(\bF)\simeq\Mat_8(\bF)$ and we have
\[
\det(a)=(\alpha\beta\gamma)^2(\det d)^2=\alpha^2(\beta^2)^2(\det d)^2=(\pm 1)^2=1.
\]
Therefore $aa^*=1$ and $\det a=1$. Thus $\psi(x)=axa^{-1}$ is an automorphism in the connected component (as an algebraic group) $G^0=\iota_{SO(8)}$ of $G=\Aut\frso_8(\bF)$. We conclude that the associated MAD $M_1$ is contained in $G^0$.

\medskip

\noindent$\boxed{\Gamma_2}$:\quad The $\bZ_2^3\times\bZ_4$-grading of $\frso_8(\bF)\simeq K(R,*)$ with invariant $[(Q^{\otimes 2},\tau_o^{\otimes 2}),(1\otimes 1,1\otimes q_1)]$. This corresponds to the identification of $R=\Mat_8(\bF)$ with $\End_{Q^{\otimes 2}}(V)$, where $V$ is a two dimensional graded module over the graded division algebra $Q^{\otimes 2}$, with a good basis  $\{v_1,v_2\}$ which is orthogonal for a hermitian form relative to $\tau_o^{\otimes 2}$, $B:V\times V\rightarrow Q$, such that $B(v_1,v_1)=1\otimes 1$ and $B(v_2,v_2)=1\otimes q_1$. It turns out that again the type of the grading is $(24,2)$ and the subalgebra $K\subo$ of $K=K(R,*)\simeq\frso_8(\bF)$ spanned by the homogeneous subspaces $K_g$ with $2g=0$ is a direct sum of four copies of $\frsl_2(\bF)$ with the same properties as for $\Gamma_1$. Also $\Diag_{\Gamma_2} R$ is contained in $G^0=\iota_{SO(8)}$ Actually, any $\psi\in\Diag_{\Gamma_2}(R)$ is given by the conjugation by an element $a=(E_{11}+\alpha E_{22})\otimes (q_i\otimes q_j)$ ($0\leq i,j\leq 3$) with $\alpha^2=1$ if $(q_iq_i^{\tau_o})(q_jq_j^{\tau_o})=1$ and $\alpha^4=1\ne \alpha^2$ if $(q_iq_i^{\tau_o})(q_jq_j^{\tau_o})=-1$. Hence $\det a=\alpha^4(\det q_i)^4(\det q_j)^4=1$.

\bigskip

Since $\Gamma_0$ is of type $(25,0,1)$, while the type of both $\Gamma_2$ and $\Gamma_3$ is $(24,2)$, it is clear that $\Gamma_0$ is not equivalent to $\Gamma_1$ or $\Gamma_2$ in $\Aut\frso_8(\bF)$. However, although $\Gamma_1$ and $\Gamma_2$ are not equivalent as $*$-gradings of $R$ (actually $\dim R_0=4$ for $\Gamma_1$, while $\dim R_0=2$ for $\Gamma_2$), we have the following result:

\begin{proposition}\label{pr:gamma1gamma2}
Up to equivalence, there is a unique fine grading of $\frso_8(\bF)$ with universal grading group $\bZ_2^3\times\bZ_4$ and type $(24,2)$. In other words, the gradings $\Gamma_1$ and $\Gamma_2$ of $\frso_8(\bF)$ are equivalent.
\end{proposition}
\begin{proof}
Consider the MADs $M_1$ and $M_2$ in $G=\Aut\frso_8(\bF)$ associated to $\Gamma_1$ and $\Gamma_2$. We have already checked that $M_1$ and $M_2$ are contained in $G^0=\iota_{SO(8)}\simeq PSO(8)$, the connected component of $G$ containing the neutral element. Also, $G$ is the semidirect product of the normal subgroup $G^0$ and the subgroup $S_3$ generated by $\iota_\sigma$ and $\theta$ (see Proposition \ref{pr:autsoCq}). Then $G$ contains exactly three subgroups of index $3$ containing $G^0$, namely the subgroups $G^i$ generated by $G^0$ and by $\iota_\sigma\theta^i$ ($i=1,2,3$) and note that we have $G^3=\iota_{O(C,q)}$, which can be identified to $\Aut(R,*)$. Let $f\in G$ be an element in the normalizer of $M_i$: $f\in N_i:=\Norm_G(M_i)$, $i=1,2$. Then $f$ either fixes or permutes the two homogeneous subspaces of $\Gamma_i$ of dimension $2$ and hence $f^2$ fixes each one of these homogeneous subspaces. Moreover, $f^2$ induces a permutation in each of these subspaces of the two
one-dimensional subspaces (see \eqref{eq:basisKs}) which can be obtained as the product of two homogeneous spaces in $K\subo$. In particular $f^4$ fixes each one of these one-dimensional spaces: denote by $h_1,h_2,h_3,h_4$ the generators of this subspaces and by $\frh$ the subspace they generate (it is a Cartan subalgebra of $\frso_8(\bF)$). The eigenvalues of the action of each $h_i$ on $K\subo$ are $\alpha_i,-\alpha_i,0$ for some $0\ne\alpha_i\in\bF$ (recall that $K\subo$ is a direct sum of four copies of $\frsl_2(\bF)$), and hence necessarily we have $f^4(h_i)=\pm h_i$ for each $i=1,2,3,4$. We conclude that $f^8$ belongs to the subgroup $\{g\in G: g\vert_\frh=id\}$ which is a torus isomorphic to $(\bF^\times)^4$ as it is the subgroup that fixes elementwise a Cartan subalgebra, and hence it is conjugated to our previous $T$ (see \eqref{eq:T}). In particular $f^8$ is contained in $G^0$.
Since $G/G^0$ is isomorphic to $S_3$,
it turns out that $N_iG^0/G^0$ is either trivial or a group of order $2$ and hence $N_i$ is contained in $G_{j_i}$ for some $j_i=1,2,3$. Hence $M_1'=\theta^{j_1}M_1\theta^{-j_1}$ and $M_2'=\theta^{j_2}M_2\theta^{-j_2}$ are two MADs both contained in $G^0$ (as $G^0$ is a normal subgroup of $G$) and with normalizers $N_1'$ and $N_2'$ contained in $G^3$.

We are left with two cases:
\begin{itemize}
\item If $M_1'$ and $M_2'$ are conjugated in $G^3$ we are done: $M_1$ and $M_2$ are conjugated an hence $\Gamma_1$ and $\Gamma_2$ are equivalent.
\item Otherwise, $M_1'$ and $M_2'$ are representatives of the MADs associated to the two $\bZ_2^3\times\bZ_4$-gradings of type $(24,2)$ in $\frso_8(\bF)$ inherited from $*$-gradings of $R=\Mat_8(\bF)$. But then $\theta M_1'\theta^{-1}(\subseteq G^0\subseteq G^3)$ is conjugated to either $M_1'$ or $M_2'$ in $G^3=\Aut(R,*)$. If $\theta M_1'\theta^{-1}$ were conjugated to $M_1'$, there would exist $\psi\in G^3$ such that $\theta M_1'\theta^{-1}=\psi M_1'\psi^{-1}$, so that $\psi^{-1}\theta\in N_1'\subseteq G^3$, a contradiction since $\psi\in G^3$ but $\theta\not\in G^3$. Therefore, $\theta M_1'\theta^{-1}$ is conjugated to $M_2'$ and hence $M_1$ and $M_2$ are conjugated.
\end{itemize}
\end{proof}

\smallskip

The next result summarizes the classification of the fine gradings on $\frso_8(\bF)$.

\begin{theorem}\label{th:D4}
Up to equivalence, there are $17$ fine gradings of the simple Lie algebra $\frso_8(\bF)$. Their universal grading groups and types are the following:
\begin{enumerate}
\item $\bZ^4$ (Cartan grading), of type $(24,0,0,1)$,
\item $\bZ^3\times\bZ_2$ of type $(25,0,1)$,
\item $\bZ^2\times\bZ_2^3$ of type $(26,1)$,
\item $\bZ\times\bZ_2^5$ of type $(28)$,
\item $\bZ_2^7$ of type $(28)$,
\item $\bZ^2\times\bZ_2^2$ of type $(20,4)$,
\item $\bZ\times\bZ_2^3$ of type $(25,0,1)$,
\item $\bZ\times\bZ_2\times\bZ_4$ of type $(24,2)$,
\item $\bZ_2^5$ of type $(24,0,0,1)$,
\item $\bZ_2^3\times\bZ_4$ of type $(25,0,1)$,
\item $\bZ_2^3\times\bZ_4$ of type $(24,2)$,
\item $\bZ_2\times\bZ_4^2$ of type $(26,1)$,
\item $\bZ\times\bZ_2^4$ of type $(28)$,
\item $\bZ_2^6$ of type $(28)$,
\item $\bZ^2\times\bZ_3$ of type $(26,1)$,
\item $\bZ_2^3\times\bZ_3$ of type $(14,7)$,
\item $\bZ_3^3$ of type $(24,2)$.
\end{enumerate}
\end{theorem}
\begin{proof}
This is just the collection given in Example \ref{ex:R*_D4} with the two gradings over $\bZ_2^3\times\bZ_4$ of type $(24,2)$ appearing as the same grading because of Proposition \ref{pr:gamma1gamma2}, plus the three extra gradings in Theorem \ref{th:fineouterD4}. The computation of the types is straightforward, along the same lines used in the computation with the $\Gamma_i$'s, $i=0,1,2$.
\end{proof}

\end{document}